\documentclass[oneside,reqno]{amsart}
\usepackage{url}
\usepackage[algoruled,vlined]{algorithm2e}
\usepackage{hyperref} 
\hypersetup{hidelinks}
\usepackage{amssymb,amsthm,amssymb}
\usepackage{geometry}
\usepackage{xcolor}
\usepackage{amsmath}
\usepackage{graphicx}
\usepackage{subfig}
\usepackage{wrapfig}
\usepackage{dsfont}
\usepackage{etoolbox}
\usepackage{tikz}
\usepackage{lipsum}
\usepackage{overpic}
\usepackage{setspace}
\usepackage{csquotes}
\usepackage{enumerate}
\usepackage{enumitem}
\usepackage[numbers]{natbib}
\usepackage{float}
\restylefloat{table}
\date{}
\newtheorem{theorem}{Theorem}[section]
\newtheorem{lemma}[theorem]{Lemma}

\theoremstyle{definition}

\theoremstyle{remark}
\newtheorem{remark}[theorem]{Remark}

\numberwithin{equation}{section}

\begin{document}
	\title{Local well-posedness for the Zakharov system in dimension $d\leqslant 3$}
	\author{Akansha Sanwal}
	\maketitle
	\vspace{-10mm}
	\begin{abstract}
		The Zakharov system in dimension $d\leqslant 3$ is shown to be locally well-posed in Sobolev spaces $H^s \times H^l$, 
		%(for a wide range of exponents $s$ and $l$)
		 extending the previously known result. We construct new solution spaces by modifying the $X^{s,b}$ spaces, specifically by introducing temporal weights. We use contraction mapping principle to prove local well-posedness in the same. The result obtained is sharp up to endpoints.
	\end{abstract} 

	\section{Introduction}
We consider the Cauchy problem for the Zakharov system
\begin{align}
	i\partial_t u +\Delta u& = nu\\ \nonumber
	\partial_t^2n-\Delta n &= \Delta |u|^2
\end{align}
with initial data
\begin{equation}
	u(x,0)= u_0(x) ,~~~n(x,0)= n_0(x), ~~~\partial_t{n}(x,0)= n_1(x),
\end{equation}
where $(u_0,n_0,n_1) \in H^s \times H^l \times H^{l-1}$.

Here $u:  \mathbb{R}^{d+1} \rightarrow  \mathbb{C}$ and $n: \mathbb{R}^{d+1} \rightarrow  \mathbb{R}$ denote the slowly changing amplitude of a high frequency electric field and the deviation of ion density from its equilibrium, respectively. This system was derived by Zakharov \cite{zakharov1972collapse} to model the propagation of Langmuir waves in a plasma.

	\subsection{Known results}
	Local well-posedness in $H^2 \times H^1\times L^2$ was shown in \cite{zbMATH00597438}
	 along with some smoothing for the Schr{\"o}dinger part for $d\leqslant 3$. The result was improved by Bourgain and Colliander \cite{zbMATH00972697}
	  who proved well-posedness of the two and three dimensional system in the energy space ($H^1 \times L^2$) using Fourier restriction norms and contraction mapping. Their result also includes small data global well-posedness if the $H^1$ norm of the Schr\"odinger data is sufficiently small. In \cite{zbMATH01117825}, local well-posedness for dimension $d\geqslant 1$ was shown for a range of indices $(s,l)$ depending on $d$ by refining the contraction mapping argument in Fourier restriction spaces. In particular, for the two and three dimensional system, local well-posedness is proved in $H^s \times H^l$ for $(s,l)$ satisfying $l\leqslant s \leqslant l+1$, $l\geqslant 0$ and $2s-l\geqslant 1$. For the 1D system, the result covers the range $-\frac{1}{2}\leqslant s-l\leqslant 1, 2s-l\geqslant \frac{1}{2}$. The best available result for local well-posedness of the system in three dimensions is proved in \cite{zbMATH05964196}
	 which covers the range $l> -\frac{1}{2}$, $l\leqslant s \leqslant l+1$ and $2s-l>\frac{1}{2}$ (see Figure 1) while that in two space dimensions is proved in \cite{zbMATH05555745} 
	 for $(s,l) = (0,-\frac{1}{2})$. Both the aforementioned results exploit the transversality of the characteristic hypersurfaces in the resonant interactions to obtain sharp multilinear estimates. These estimates, referred to as generalisations of the Loomis Whitney inequality \cite{zbMATH03054426,zbMATH05769043,zbMATH02207822}, shall play a vital role in proving the non-linear estimates in sections 3 and 4.
	  For the 1D system, in \cite{zbMATH01690751}, global well-posedness without any smallness assumption on the initial data is proved in $H^2\times L^2$ for $\frac{9}{10}<s<1$.  In \cite{zbMATH05264249}, local well-posedness is shown in spaces $\widehat{H^{s,p}}$ for $1<p<2$. Here, the regularity of the initial data can be lowered so that $s<0$ and $l<-\frac{1}{2}$ by using modified $X^{s,b}$ spaces. Global well-posedness using mass conservation is proved in \cite{zbMATH05318537} for $(s,l)=(0,-\frac{1}{2})$ which is the largest space where the system is locally well-posed. In \cite{zbMATH06529170}, global well-posedness and scattering results are proved.\\
	  
	  We throw some light on ill-posedness results now.  In \cite{zbMATH05142144}, for the 1D system, it is shown that the constraint $2s-l\geqslant \frac{1}{2}$ is optimal in the sense that norm inflation occurs otherwise. For $l<-\frac{3}{2}$ and $s=0$, it is shown that the data to solution map fails to be uniformly continuous from $H^s(\mathbb{R})\times H^l(\mathbb{R}) \times H^{l-1}(\mathbb{R})$ to $C([0,T];H^s(\mathbb{R}) )\times C([0,T] ;H^l(\mathbb{R})) \times C([0,T];H^{l-1}(\mathbb{R}))$. This nearly matches the 1D result of \cite{zbMATH01955363} where the authors prove ill-posedness for $s<0$ and $l\leqslant -\frac{3}{2}$. The failure of $C^2$ differentiability of the data to solution map is proved for $l<-\frac{1}{2}$. In \cite[section 9]{candy2019zakharov}, the authors prove that for a general dimension $d\geqslant 1$, the flow map is not $C^2$ for $s,l$ and $d$ satisfying
	  $l<\frac{d}{2}-2, ~s-l> 2, ~2s-l<\frac{d-2}{2}$ or $s-l<-1$. For $d=3$, this is the complement of the region described by $(1.3)$ up to the boundaries. For $d=2$, stronger counterexamples in \cite[section 6]{zbMATH05555745} prove that the flow map is not $C^2$ for $2s-l<\frac{1}{2}$ or $l<-\frac{1}{2}$. In \cite[Proposition 3.2]{zbMATH01117825}, using the close relation of the system $(1.1)-(1.2)$ to the cubic non-linear Schr\"odinger equation, the authors prove that the data to solution map for the system fails to be Lipschitz continuous at the point $(s,l) = (\frac{1}{2}, -\frac{1}{2})$. Using power series expansion for the solutions, as done in \cite{zbMATH07040496} for the Schr\"odinger equation, norm inflation results have been proved for the Zakharov system in dimension $d\geqslant 1$ in \cite{grubemastersthesis}. In particular, for $d=3$ it is proved that all the boundaries but $l=-\frac{1}{2}$ are sharp. The interested reader can look \cite{domingues2019note} for a summarised account of ill-posedenss results for the system $(1.1)-(1.2)$ in dimension $d\geqslant 1$. \\

	In this article, we provide a unified local well-posedness result for the system $(1.1)-(1.2)$ in $d\leqslant3$ 
	which holds in the region $(1.3)$ (see Figures 1 and 2) using the contraction mapping principle in modified $X^{s,b}$ spaces. Our results reads
	% The region $(1.3)$ is the optimal in the sense that the flow map for $(2.1)$ fails to be $C^2$ at the origin for values outside $(1.3)$ (see Section 9 in \cite{3}). We modify $X^{s,b}$ spaces by introducing weights which depend upon the temporal frequencies.
	\begin{theorem}
		The system $(1.1)$ with initial conditions $(1.2)$ is locally well-posed for $d\leqslant 3$ provided
		\begin{equation}
		l > -\frac{1}{2}, \quad \max\Big\{l-1,\frac{l}{2}+\frac{1}{4}\Big\} < s <l+2. 
		\end{equation}
	\end{theorem}

\begin{remark}
%\textbf{(i)}
%		 The flow map for the system $(1.1)-(1.2)$ is defined by
%	\begin{align*}
%		f: B_r \rightarrow H^s(\mathbb{R}^d)\times H^l(\mathbb{R}^d) \times H^{l-1}(\mathbb{R}^d), \quad
%		(u_0, n_0,n_1)\mapsto (u(t),n(t),\partial_t n(t)),
%	\end{align*}
%	where $B_r$ denotes $\{(u_0, n_0, n_1): \Vert u_0\Vert_{H^s} + \Vert n_0\Vert_{H^l}+\Vert n_1\Vert_{H^{l-1}}<r\}$. The contraction mapping principle guarentees that $f$ is analytic. If one can prove that the map $f$ is not $C^k$ for some $k$, it shows that well-posedness
		\textbf{(i)} From the existing literature, it follows that $(s,l)=(0,-\frac{1}{2})$ is the optimal point at which local well-posedness can be proved by a contraction mapping argument. This also corresponds to the scaling critical regularity for the three dimensional system, \cite[Section 2, pp 7-8]{zbMATH01117825}. Our result does not cover this because of the choice of our function spaces and losses which arise from the application of the estimates in lemma 2.4 and 2.6. Hence the endpoint remains an open problem for $d=3$.\\
		\textbf{(ii)} The regularity $(s,l) = (1,0)$ which corresponds to the energy space for the system $(1.1)-(1.2)$ lies on the boundary in the results covered by \cite{zbMATH01117825,zbMATH05964196}. The inclusion of the same as an interior point of the region of well-posedness is noteworthy.\\
		%	\textcolor{blue}{	\textbf{(ii)} Partial boundaries of the region in $(1.3)$ can be covered by using only the product estimate (lemma 2.9) to prove the estimates. This is precisely the region $s>1$ and $l>\frac{3}{2}$. Since the low-high modulation interaction for the wave estimates cannot be controlled optimally without the bilinear estimate $(2.15)$, we cannot achieve the lower right boundary $s-l=-1$.} (Boundaries cannot be achieved now.)\\
		\textbf{(iii)} Using the normal form approach of \cite{zbMATH06369466}, one can improve the local well-posedness result for the system $(1.1)-(1.2)$. This was achieved in \cite{rubelmastersthesis} for $d\geqslant 2$. However, the result does not cover the negative regularity region for the wave solution.\\
	%	\textbf{(iv)} \textcolor{blue}{Norm inflation results of [FG].}
	\end{remark}
\vspace{-5mm}

\begin{figure}[H]
	\centering
	\begin{minipage}{0.49\textwidth}
		\centering
		\begin{tikzpicture}[scale=0.84]
		\draw[thick] (-1.1,2)--(8.1,2);
		\draw[thick] (-1.1,4)--(8.1,4);
		\draw[thick] (-1.1,6)--(8.1,6);
		\draw[thick] (-1.1,8)--(8.1,8);             %% All parallel lines to l axis
		
		\draw[thick] (2,0)--(2,8.1);
		\draw[thick] (4,0)--(4,8.1);
		\draw[thick] (6,0)--(6,8.1);
		\draw[thick] (8,0)--(8,8.1);              %%  All parallel lines to s axis
		
		\filldraw[color=gray, opacity=0.5] (6,8)--(-1,1)--(-1,0)--(1,1)--(8,8);      %% shaded area
		\draw[ thin, dashed] (-1,1)--(-1,0)--(1,1);     %% shaded area
		\draw[thin](6,8)--(-1,1);
		\draw[thin](1,1)--(8,8);
		
		\draw[ultra thick] (-1.1,0)--(0,0)node[align=center, below] {0}--(2,0)node[align=center, below] {1}--(4,0)node[align=center, below] {2}--(6,0)node[align=center, below] {3}--(8,0)node[align=center, below] {4}--(8.1,0);     %% l axis
		
		\draw[ultra thick] (0,0)--(0,2) node[align=center, left] {1}--(0,4) node[align=center, left] {2}--(0,6) node[align=center, left] {3}--(0,8) node[align=center, left] {4}--(0,8.1);     %% s axis

		%%rotated labels
			
			\node[align=center] at (1.9,4.3) {\rotatebox{45}{$s-l=1$}};
				\node[align=center] at (3.8,3.4) {\rotatebox{45}{$s=l$}};
					\node[align=center] at (0.3,0.4) {\rotatebox {25}{\tiny $2s-l=\frac{1}{2}$}};
						\node[align=center] at (-1.3,0.55) {\rotatebox{90}{\tiny$l=-\frac{1}{2}$}};  
		
		\node[align=center, above] at (0.3,7.5) {{\boldmath $s$}};  %% s axis name
		\node[align=center, right] at (7.5,0.3) {{\boldmath $l$}}; %% l axis name
		\end{tikzpicture}
		\caption{Region of well-posedness for $d=3$ in \cite{zbMATH05964196}}
	\end{minipage}
	\begin{minipage}{0.49\textwidth}
		\centering
		\begin{tikzpicture}[scale=0.84]
		\draw[thick] (-1.1,2)--(8.1,2);
		\draw[thick] (-1.1,4)--(8.1,4);
		\draw[thick] (-1.1,6)--(8.1,6);
		\draw[thick] (-1.1,8)--(8.1,8);             %% All parallel lines to l axis
		
		\draw[thick] (2,0)--(2,8.1);
		\draw[thick] (4,0)--(4,8.1);
		\draw[thick] (6,0)--(6,8.1);
		\draw[thick] (8,0)--(8,8.1);              %%  All parallel lines to s axis
		
		\filldraw[color=gray, opacity=0.5] (4,8)--(-1,3)--(-1,0)--(5,3)--(8,6)--(8,8)--(4,8);      %% shaded area
		\draw[thin, dashed] (4,8)--(-1,3)--(-1,0)--(5,3)--(8,6);      %% shaded area

		\draw[ultra thick] (-1.1,0)--(0,0)node[align=center, below] {0}--(2,0)node[align=center, below] {1}--(4,0)node[align=center, below] {2}--(6,0)node[align=center, below] {3}--(8,0)node[align=center, below] {4}--(8.1,0);     %% l axis
		
		\draw[ultra thick] (0,0)--(0,2) node[align=center, left] {1}--(0,4) node[align=center, left] {2}--(0,6) node[align=center, left] {3}--(0,8) node[align=center, left] {4}--(0,8.1);     %% s axis
		
%		\node[align=center] at (1.4,0.45) {$2s-l=\frac{1}{2}$};  
%		\node[align=center] at (6.5,3.5) {$s=l-1$};
%		\node[align=center] at (1,6.2) {$s=l+2$};           %% all equations
		
		\node[align=center, above] at (0.3,7.5) {{\boldmath$s$}};  %% s axis name
		\node[align=center, right] at (7.5,0.3) {{\boldmath $l$}}; %% l axis name
		%%rotated labels
			\node[align=center] at (1,5.3) {\rotatebox{45}{ $s-l=2$}};
		 
			\node[align=center] at (-1.3,1.2) {\rotatebox{90}{$l=-\frac{1}{2}$}};  
		%	\node[align=center] at (3.3,2.5) {\rotatebox{45}{\boldmath$s-l=-\frac{1}{2}$}}; 
			\node[align=center] at (6.4,3.9) {\rotatebox{45}{$s-l=-1$}}; 
			\node[align=center] at (1.4,0.7) {\rotatebox{26}{$2s-l=\frac{1}{2}$}}; 
		\end{tikzpicture}
		\caption{New region of well posedness for $d\leq 3$}
		
	\end{minipage}
	\text{The strips extend infinitely to the upper right in both the regions.}
\end{figure}

%\begin{figure}[H]
%	\centering
%	\begin{minipage}{0.48\textwidth}
%		\centering
%		\begin{overpic}[width=0.999 \textwidth,tics=10]{O.PNG}
%			
%		
%			\put (15,44){\begin{tiny}$s=l+1$\end{tiny}}
%			\put (39,32){\begin{tiny}$s=l$\end{tiny}}
%			\put (14,10.25){\begin{tiny}$2s-l=\frac{1}{2}$\end{tiny}}
%			
%		\end{overpic}
%	
%		\caption{Region of well-posedness in \cite{bejenaru2011convolutions} for $d=3$}
%	\end{minipage}\hfill
%	\begin{minipage}{0.48\textwidth}
%		\centering
%		\begin{overpic}[width=0.9999 \textwidth,tics=10]{N.PNG}
%			
%			%	\put (87,1) {\begin{tiny}$l$ \end{tiny}} 	
%			%	\put (10,58){\begin{tiny}$s$\end{tiny}}
%			\put (15,66){\begin{tiny}$s=l+2$\end{tiny}}
%			\put (67,41){\begin{tiny}$s=l-1$\end{tiny}}
%			\put (27,17){\begin{tiny}$2s-l=\frac{1}{2}$\end{tiny}}
%			
%		\end{overpic}
%		%	\includegraphics[width=0.9\textwidth]{3DN.PNG} % second figure itself
%		\caption{New region of well-posedness for $d\leqslant 3$}
%	\end{minipage}
%	\text{The strips extend infinitely to the upper-right in both the regions.}
%\end{figure}

	\subsubsection{Outline} In section 2, we introduce the notation and previously known estimates. We also define our function spaces and state the required properties of $X^{s,b}$ spaces therein. Known multilinear estimates are also enlisted. Sections 3 and 4 are devoted to the proof of the crucial multilinear estimates for the Schr{\"o}dinger and wave non-linearity, respectively. Finally, in section 5, we give a short proof of theorem 1.1.

	%\textcolor{red}{In the appendix, we give a short proof of the trilinear estimate $(2.8)$.}
	
%	\vspace{0.3cm}
\section{Notation and preliminaries}

\subsection{Reduced system}
For $|\nabla|$:=$\sqrt{-\Delta}$, using the transformation $v=n-i|\nabla|^{-1}\partial_t n $, the system $(1.1)$ formally reduces to the following  first order system:
\begin{align}
i\partial_t u +\Delta u& = uRe(v)\\ \nonumber
i\partial_t v+|\nabla| v &= -|\nabla| |u|^2.
\end{align}

We observe that $(u,v)$ solves $(2.1)$ iff $(u,Re(v))$ solves $(1.1)$.
Henceforth, we shall consider the system $(2.1)$ for our analysis.

\subsection{Fourier multipliers}
Let $\chi \in C_0^{\infty}(\mathbb{R})$ be non-negative such that $\chi (r) = 1$ for $|r| \leqslant 1$, $\chi (r) = 0$ for $|r|\geqslant 2$. Set $\phi(r) :=\chi(r)-\chi(2r)$ and $\phi_{\lambda}(r):=\phi(r/\lambda)$. Then
\begin{equation*}
\sum_{\lambda \in 2^{\mathbb{N}}\\,\lambda \geqslant 1}\phi_{\lambda}(r) =1 , ~~~\phi_1 = \chi.
\end{equation*}
For $\lambda \in 2^\mathbb{N}, \lambda > 1$, we define the Fourier multipliers
\begin{center}
	\begin{tabular}{cccc}
		$P_{\lambda} = \phi_{\lambda}(|\nabla|)$,& $P_{\lambda}^{(t)} = \phi_{\lambda}(|\partial_t|)$, $C_{\lambda} = \phi_{\lambda}(|i\partial_t +\Delta|)$, $Q_{\lambda} = \phi_{\lambda}(|i\partial_t \pm|\nabla||)$,
	\end{tabular}
\end{center}
and for $\lambda = 1$, we define
\begin{center}
	\begin{tabular}{cccc}
		$P_1 = \chi(|\nabla|)$,& $P_1^{(t)} = \chi(|\partial_t|)$, $C_1 = \chi (|i\partial_t +\Delta|)$, $Q_1 = \chi (|i\partial_t \pm|\nabla||)$.
	\end{tabular}
\end{center}
%	and for $\lambda = 1$, we define
%	\begin{center}
%		\begin{tabular}{cccc}
%			$P_1 = \chi (t)(|\nabla|)$,& $P_1^{(t)} = \chi (t)(|\partial_t|)$, $C_1 = \chi (t)(|i\partial_t +\Delta|)$, $Q_1 = \chi (t)(|i\partial_t +|\nabla||)$.
%		\end{tabular}
%	\end{center}
Thus, $P_{\lambda}$ is a Fourier multiplier which localises the spatial frequencies to the set $\{\frac{\lambda}{2} \leqslant |\xi| \leqslant 2\lambda \}$ while $P_{\lambda}^{(t)}$ localises the temporal frequencies to the set $\{\frac{\lambda}{2} \leqslant |\tau| \leqslant 2\lambda \}$. In a similar spirit, $C_{\lambda}$ and $Q_{\lambda}$ localise the space-time Fourier support to distances $\sim \lambda$ from the paraboloid and the cone, respectively.
More precisely, for $L \in 2^\mathbb{N}$ the space-time Fourier supports of $C_L, C_1, Q_L$ and $Q_1$ can be written as 
\begin{align*}
&\mathcal{S}_{L} = \Big\{(\tau,\xi)\in \mathbb{R} \times \mathbb{R}^d : \frac{L}{2}\leqslant|\tau+|\xi|^2|\leqslant 2L \Big\}\\
&\mathcal{S}_{1} = \Big\{(\tau,\xi)\in \mathbb{R} \times \mathbb{R}^d : |\tau+|\xi|^2|\leqslant 2 \Big\}\\
&\mathcal{W}_{L}^{\pm} = \Big\{(\tau,\xi)\in \mathbb{R} \times \mathbb{R}^d : \frac{L}{2}\leqslant|\tau\pm|\xi||\leqslant 2L \Big\}\\
&\mathcal{W}_{1}^{\pm} = \Big\{(\tau,\xi)\in \mathbb{R} \times \mathbb{R}^d : |\tau\pm|\xi||\leqslant 2 \Big\}.
\end{align*}
To restrict the Fourier support to larger sets, we use the notation
\begin{align*}
&P_{\leqslant\lambda} = \sum_{\mu \in 2^{\mathbb{N}},\mu \leqslant \lambda}\phi_{\lambda}(|\nabla|), \quad P_{\leqslant \lambda}^{(t)} = \sum_{\mu \in 2^{\mathbb{N}}, \mu \leqslant \lambda}\phi_{\lambda}(|\partial_t|),\\
&C_{\leqslant\lambda} = \sum_{\mu \in 2^{\mathbb{N}},\mu \leqslant \lambda}\phi_{\lambda}(|i\partial_t +\Delta|), \quad Q_{\leqslant\lambda} = \sum_{\mu \in 2^{\mathbb{N}},\mu \leqslant \lambda}\phi_{\lambda}(|i\partial_t \pm|\nabla||).
\end{align*}
For $\lambda \in 2^{\mathbb{N}}$, we use the shorthand $f_{\lambda}=P_{\lambda}f$.\\
We use the notation $A\lesssim B$ to indicate $A\leqslant cB$, for a constant $c>0$ and $A\sim B$ when both $A\lesssim B$ and $B\lesssim A$ hold. We use $A\ll B$ to denote $A\leqslant cB$ for a constant $c$ much smaller than 1.

	\subsection{Function spaces}
The space-time Fourier transform of a function $f$ is denoted by $\mathcal{F}(f)$, where $(t,x)$ are the physical space variables and $(\tau,\xi)$ the corresponding Fourier variables. We use angled brackets to denote $\langle x\rangle :=(1+x^2)^{\frac{1}{2}}$, and $x+$ denotes a number slightly bigger than $x$ in the following. \\

We set $\theta:=\frac{1}{2}+$. Given $s,l \in \mathbb{R}$, we define the parameters $0\leqslant a, b<2-2\theta,\theta',s'$ and $\beta$ as follows:
\begin{align}\nonumber
&	a:=\begin{cases}
s-l-3+4\theta, & s-l\geqslant 1\\ 
0,& s-l <1
\end{cases}
\text{,} \quad
b:=\begin{cases}
0,& s-l>0\\
(l-s)+2\theta-1,& s-l\leqslant 0
\end{cases}
\text{, }\quad
\theta':=\begin{cases}
\theta~, s-l\geqslant 0\\
1~, s-l <0
\end{cases}
\text{,} \quad\\
&s':=\begin{cases}
s, &s-l\geqslant 0\\
s+2\theta-2, & s-l<0
\end{cases}
\quad\text{ and} \quad
\beta:=\begin{cases}
l+a+2\theta-1, &s-l\geqslant 1\\
l,  &-\frac{1}{2}\leqslant s-l< 1\\
s-\theta+1, &s-l<-\frac{1}{2}
\end{cases}.
\end{align}

For $\lambda \in 2^{\mathbb{N}}$, $s,a,b, \theta \in \mathbb{R}$, we control the frequency localised Schr{\"o}dinger component of the Zakharov evolution by the following
\begin{equation*}
\Vert u_{\lambda} \Vert_{S_{\lambda}^{s,l,\theta}} = \lambda^s\Vert u_{\lambda}\Vert_{L_t^{\infty}L_x^2} + \lambda^{s}\Vert \langle \tau+|\xi|^2\rangle^{\theta}\mathcal{F}(C_{\ll \lambda^2}u_{\lambda})\Vert_{L_{\tau,\xi}^2}
+\lambda^{s'-2a+b}\Vert \langle \tau+|\xi|^2\rangle^{\theta'} \mathcal{F}((\lambda+|\partial_t|)^a C_{\gtrsim \lambda^2}u_{\lambda})\Vert_{L_{\tau,\xi}^2}.
\end{equation*}
To control the frequency localised Schr{\"o}dinger non-linearity, we define
\begin{align*}	\Vert F_{\lambda} \Vert_{N_{\lambda}^{s,l,\theta-1}} &= \lambda^{s+2\theta-3}\Vert  P_{\ll \lambda^2}^{(t)}F_{\lambda}\Vert_{L_t^{\infty}L_x^2}  + \lambda^{s}\Vert \langle \tau+|\xi|^2\rangle^{\theta-1} \mathcal{F}(C_{\ll \lambda^2}F_{\lambda})\Vert_{L_{\tau,\xi}^2}\\
&\quad 	+\lambda^{s'-2a+b}\Vert \langle \tau+|\xi|^2\rangle^{\theta'-1}\mathcal{F}((\lambda+|\partial_t|)^a C_{\gtrsim \lambda^2}F_{\lambda})\Vert_{L_{\tau,\xi}^2}.
\end{align*}
For $l, \beta,a, \theta \in \mathbb{R}$, the evolution of the frequency localised wave component is controlled by the following norm
\begin{equation*}
\Vert v_{\lambda}\Vert_{W_{\lambda}^{l,s,\theta}} =  \lambda^{l}\Vert v_{\lambda}\Vert_{L_t^{\infty}L_x^2}+ \lambda^{l-a}\Vert \langle \tau-|\xi|\rangle^{\theta}\mathcal{F}((\lambda+|\partial_t|)^a Q_{\ll  \lambda^2}v_{\lambda})\Vert_{L_{\tau,\xi}^2}+  \lambda^{\beta-1}\Vert \langle \tau-|\xi|\rangle \mathcal{F}(Q_{\gtrsim \lambda^2}v_{\lambda})\Vert_{L_{\tau,\xi}^2},
\end{equation*}
and the RHS of the wave equation is controlled by
\begin{equation*}
\Vert G_{\lambda}\Vert_{R_{\lambda}^{l,s,\theta-1}} = \lambda^{l+2\theta-3}\Vert G_{\lambda}\Vert_{L_t^{\infty}L_x^2} + \lambda^{l-a}\Vert \langle \tau-|\xi|\rangle^{\theta-1}\mathcal{F}((\lambda+|\partial_t|)^a Q_{\ll  \lambda^2}G_{\lambda})\Vert_{L_{\tau,\xi}^2}
+ \lambda^{\beta-1}\Vert \mathcal{F}( Q_{\gtrsim \lambda^2}G_{\lambda})\Vert_{L_{\tau,\xi}^2}.
\end{equation*}
The parameters $a,b,s', \beta$ and $\theta'$ are required to prove the bilinear estimates in the full region $(1.3)$.
$a$ measures the loss of regularity for the Schr{\"o}dinger component in the low ($\ll\lambda^2$)  temporal frequency region as can be seen from the weight $m_S(\tau):=\big(\frac{\lambda+|\tau|}{\lambda^2}\big)^a, 0\leqslant a< 1$. Note that
\begin{equation*}
m_S(\tau)\sim \begin{cases}
\lambda^{-a}, ~~~|\tau|\lesssim \lambda,\\
1,~~~~~ |\tau|\sim \lambda^2,\\
\end{cases}\\
\text{and} ~~m_S(\tau)\gg 1, ~~~|\tau|\gg \lambda^2,
\end{equation*}
while \begin{equation*}
m_W(\tau):=\Big (\frac{\lambda+|\tau|}{\lambda}\Big)^a \gtrsim 1.
\end{equation*}
The parameter $b$ gives a gain in regularity to the Schr{\"o}dinger evolution in the high modulation regime and helps to achieve bilinear estimates for the wave non-linearity in the region $(1.3)$. $a$ and $b$ are not non-zero simultaneously and hence their sum is always less than $1$. It is this restriction on the upper bounds for the parameters $a$ and $b$ that does not allow us to achieve the boundaries ($s-l=2,-1$) of the region described by $(1.3)$.\\
In the region $l+1\leqslant s< l+2$, where the Schr{\"o}dinger component is more regular, we require $a>0$. In the \enquote{balanced} region $l< s <l+1$, we choose $a=0=b$. In the final regime i.e. $l-1< s\leqslant l$, where the wave component is more regular, we choose $a=0, b>0$. Similarly, for the high modulation wave regularity $\beta$, we have  $\beta \sim s-1$ when $1\leqslant s-l <2$, which is greater than or equal to the \enquote{ideal} regularity $l$. In the balanced regime, it is chosen to be $l$, while in the thin strip $-1<s-l\leqslant -\frac{1}{2}$, where the wave is more regular, we choose $\beta$ to be $\sim s+\frac{1}{2}$ which is less than $l$. For the regularity $s'$ of the high modulation Schr\"odinger norm, we see that the change in the exponent of the modulation weight viz $\theta'$, compensates for the loss. Note that in the low modulation, the $S_{\lambda}$ and $W_{\lambda}$ norms are exactly the $X^{s,\theta}$ norms with $\theta>\frac{1}{2}$.

We also observe that the choice of the parameter $\theta' (>\frac{1}{2})$ enables us to use various properties of the standard $X^{s,b}$ spaces and the bilinear estimates from \cite{zbMATH05555745,zbMATH05964196} (see Section $2.6$) appropriately.
\begin{remark} 
	Equation $(2.2)$ provides us with one possible choice for the parameters. We reckon other choices might be plausible too.
\end{remark}

The evolution of the full Schr{\"o}dinger solution and non-linearity is controlled by the following norms
\begin{equation*}
\Vert u \Vert_{S^{s,l,\theta}} = \Big(\sum_{\lambda\in 2^{\mathbb{N}}}\Vert u_{\lambda}\Vert_{S^{s,l,\theta}_{\lambda}}^2\Big)^{\frac{1}{2}}, ~~~\Vert F \Vert_{N^{s,l,\theta-1}} = \Big(\sum_{\lambda\in 2^{\mathbb{N}}}\Vert F_{\lambda}\Vert_{N^{s,l,\theta-1}_{\lambda}}^2\Big)^{\frac{1}{2}},
\end{equation*}
and that of the wave solution and wave non-linearity is controlled by
\begin{equation*}
\Vert v \Vert_{W^{l,s,\theta}} = \Big(\sum_{\lambda\in 2^{\mathbb{N}}}\Vert v_{\lambda}\Vert_{W^{l,s,\theta}_{\lambda}}^2\Big)^{\frac{1}{2}}, ~~~\Vert G \Vert_{R^{l,s,\theta-1}} = \Big(\sum_{\lambda\in 2^{\mathbb{N}}}\Vert G_{\lambda}\Vert_{R^{l,s,\theta-1}_{\lambda}}^2\Big)^{\frac{1}{2}}.
\end{equation*}

For $0<T\leqslant 1$, we localise the norms in time by defining
\begin{align*}
\Vert u\Vert_{S^{s,l,\theta}(T)} = \inf\{ \Vert \tilde{u}\Vert_{S^{s,l,\theta}}: \tilde{u} \in S^{s,l,\theta}, \tilde{u}\big|_{(0,T)\times \mathbb{R}^d} = u\}
\end{align*}
The norms $N^{s,l,\theta-1}(T), W^{l,s,\theta}(T)$ and $R^{l,s,\theta-1}(T)$ are defined similarly.

	\subsection{Properties of $X^{s,b}$ spaces} Since our function spaces are a variant of the standard $X^{s,b}$ spaces, we note some properties of the same in the following. The proofs can be found in \cite[Lemma 2.1, 2.2]{zbMATH01117825}, \cite[section 2.6]{zbMATH05035890}.\\
%	\begin{enumerate}[label=(\roman*)]
\textbf{(i)} Let $u$ be a solution to the problem $\partial_t u = L(u) +F $ with $L = ih(\nabla/i)$ for some real valued polynomial $h$
%iu_t+\varphi(-i\nabla_x)u=0$
with $u(0) = f$. For a smooth time cutoff $\chi$ such that $\chi (t) = 1, |t|\leqslant 1, \chi (t) = 0, |t|>2$ and $\chi_T(t):=\chi\big(\frac{t}{T}\big)$, $0<T\leqslant 1$ and  $s,b\in \mathbb{R}$, we have
\begin{equation}
\Vert \chi(t) e^{tL}f \Vert_{X^{s,b}_{\tau=h(\xi)}}\lesssim \Vert f\Vert_{H^s}. %((0,T)\times \mathbb{R}^d)}
\end{equation}
\textbf{(ii)} If $v$ is a solution to the problem $\partial_t v =L(v) + F$, with $ L= ih(\nabla/i)~ v(0)=0$, we have for $-\frac{1}{2}<b'\leqslant 0\leqslant b\leqslant b'+1$,
\begin{equation}
\Vert \chi_{T}(t)v\Vert_{X^{s,b}_{\tau=h(\xi)}} \lesssim T^{1+b'-b}\Vert F\Vert_{X^{s,b'}_{\tau=h(\xi)}}.
\end{equation}
\textbf{(iii)} For $\frac{2}{q}+\frac{d}{r}=\frac{d}{2}, 2\leqslant r<\infty$, the following Strichartz estimate holds
\begin{equation}
\Vert u \Vert_{L_t^q L_x^r} \lesssim \Vert u\Vert_{X^{0,\frac{1}{2}+}_{\Delta}},
\end{equation}
where the $\Delta$ in the subscript denotes that the space in consideration is for the Schr{\"o}dinger equation.
For the wave part, we only use
\begin{equation}
\Vert u\Vert_{L_{t}^{\infty}L_x^2} \lesssim \Vert u\Vert_{X^{0,\frac{1}{2}+}_{|\nabla|}}.
\end{equation}
\textbf{(iv)} For $b>\frac{1}{2}$, $X^{s,b}(\mathbb{R}\times \mathbb{R}^d)\hookrightarrow C(\mathbb{R};H^s(\mathbb{R}^d)),~ X^{s,b}((0,T)\times \mathbb{R}^d)\hookrightarrow C((0,T);H^s(\mathbb{R}^d))$.\\
\textbf{(v)} For $2\leqslant p <\infty, b\geqslant \frac{1}{2}-\frac{1}{p}$, we have 
\begin{equation}
\Vert u \Vert_{L_t^p L_x^2}\lesssim \Vert u\Vert_{X^{0,b}}.
\end{equation}
\textbf{(vi)} For any $s \in \mathbb{R}$ and $-\frac{1}{2}<b'<b<\frac{1}{2}$, we have 
\begin{equation}
\Vert u\Vert_{X^{s,b'}(T)} \lesssim T^{b-b'}\Vert u\Vert_{X^{s,b}(T)}.
\end{equation}
\textbf{(vii)} Duality and complex conjugation: 
\begin{equation*}
(X^{s,b}_{\tau=h(\xi)})' = X^{-s,-b}_{\tau=-h(-\xi)}, \quad \Vert \overline{u}\Vert_{X^{s,b}_{\tau=-h(-\xi)}} = \Vert u\Vert_{X^{s,b}_{\tau=h(\xi)}}.
\end{equation*}

	\subsection{Duhamel formulae and energy inequalities}
% \[\cite{10}, Lemma 2.8\]
Let $\mathcal{I}_{S}$ and $\mathcal{I}_{W}$ be the solution operators for the inhomogeneous Schr{\"o}dinger equation and the half wave equation, respectively, i.e.
\begin{equation*}
\mathcal{I}_{S}[F](t) = -i\int_0^t e^{i(t-s)\Delta}F(s)ds, \quad \mathcal{I}_{W}[G](t) = -i\int_0^t e^{i(t-s)|\nabla|}G(s)ds.
\end{equation*}
Then we have the following estimates.
\begin{lemma} (Energy inequality for the Schr{\"o}dinger equation)
	Let $a,b,s',\theta' \in \mathbb{R}$ be as defined in $(2.2)$. For any $\lambda \in 2^{\mathbb{N}}$ and a smooth time cutoff $\chi$, we have:
	\begin{align*}
	&\Vert \chi (t) e^{it\Delta}u_0\Vert_{S^{s,l,\theta}_{\lambda}} \lesssim \lambda^{s}\Vert u_0\Vert_{L_x^2},\\
	&\Vert \chi(t) \mathcal{I}_{S}[F_{\lambda}]\Vert_{S^{s,l,\theta}_{\lambda}} \lesssim \Vert F_{\lambda}\Vert_{N^{s,l,\theta-1}_{\lambda}}.
	\end{align*}
\end{lemma}
\begin{proof}
	We first note the following property of the linear group wrt temporal frequency and modulation localisation, 	which we shall use repeatedly
	\begin{equation}
	C_{\star}f = e^{tL}P_{\star}^{(t)}(e^{-tL}f).
	\end{equation}
	The norm for the free solution is given by
\begin{align*}
\Vert \chi(t) e^{it\Delta}u_{0}\Vert_{S^{s,l,\theta}_{\lambda}}
% &\leqslant \Vert e^{it\Delta}u_{0}\Vert_{S^{s,l,\theta}_{\lambda}}\\
& = \lambda^s \Vert \chi(t) e^{it\Delta}u_0\Vert_{L_t^{\infty} L_x^2} + \lambda^s \Vert \langle \tau+|\xi|^2 \rangle^{\theta}\mathcal{F}(C_{\ll \lambda^2}\chi(t) e^{it\Delta}u_0)\Vert_{L_{\tau,\xi}^2} \\
&+ \lambda^{s'-2a+b}\Vert \langle \tau+|\xi|^2 \rangle^{\theta'} \mathcal{F}((\lambda+|\partial_t|)^a C_{\gtrsim \lambda^2}\chi(t) e^{it\Delta}u_{0})\Vert_{L_{\tau,\xi}^2}.
\end{align*}
Using $(2.5)$ and $(2.3)$, we have
\begin{align*}
\lambda^s \Vert \chi(t)e^{it\Delta}u_0\Vert_{L_{t}^{\infty}L_x^2} \lesssim \lambda^s \Vert \langle \tau+|\xi|^2 \rangle^{\theta}\mathcal{F}(\chi(t)e^{it\Delta}u_0)\Vert_{L_{\tau,\xi}^2} \lesssim \lambda^s \Vert u_0\Vert_{L_x^2},
\end{align*}
and for the second term using $(2.9)$, we have
\begin{align*}
\lambda^s \Vert \langle \tau+|\xi|^2 \rangle^{\theta}\mathcal{F}(C_{\ll \lambda^2}\chi(t)e^{it\Delta}u_0)\Vert_{L_{\tau,\xi}^2} &= \lambda^s \Vert \langle \tau+|\xi|^2 \rangle^{\theta}\mathcal{F}(e^{it\Delta}P_{\ll \lambda^2}^{(t)}(\chi(t)u_0))\Vert_{L_{\tau,\xi}^2}\\
&=\lambda^s \Vert P_{\ll \lambda^2}^{(t)}(\chi(t)u_0)\Vert_{H_t^{\theta}L_x^2} \lesssim \lambda^s \Vert u_0\Vert_{L_x^2}.
%	&\lesssim \lambda^s \Vert \chi(t)\Vert_{H_t^{\theta}}\Vert u_0\Vert_{L_x^2}\\
\end{align*}
A similar computation gives
\begin{align*}
\lambda^{s'-2a+b}\Vert \langle \tau+|\xi|^2 \rangle^{\theta'} \mathcal{F}((\lambda+|\partial_t|)^a C_{\gtrsim \lambda^2}\chi(t) e^{it\Delta}u_{0})\Vert_{L_{\tau,\xi}^2}
%&= \lambda^{s'-2a+b}\Vert \langle \tau+|\xi|^2 \rangle^{\theta'} (\lambda+|\tau|)^a \mathcal{F}(e^{it\Delta}(P_{\gtrsim \lambda^2}^{(t)}\chi (t) u_0))\Vert_{L_{\tau,\xi}^2}
=\lambda^{s'-2a+b} \Vert P_{\gtrsim \lambda^2}^{(t)}(\chi(t)u_0)\Vert_{H_t^{\theta'+a}L_x^2}
%	&\lesssim \lambda^{s-2a+b+2\theta-2}\Vert P_{\gtrsim \lambda^2}^{(t)}(\chi (t)u_0)\Vert_{H_t^{1+a}L_x^2}\\
\lesssim \lambda^s \Vert u_0\Vert_{L_x^2}.
\end{align*}
We now consider the $S_{\lambda}$ norm of the Duhamel integral. The $L_t^{\infty}L_x^2$ term is decomposed as follows
\begin{align*}
 \lambda^s \Big \Vert \chi(t) \int_0^t e^{i(t-s)\Delta}C_{\ll \lambda^2}F_{\lambda}(s)ds\Big \Vert_{L_{t}^{\infty}L_x^2} + \lambda^s \Big \Vert \chi(t) \int_0^t e^{i(t-s)\Delta}C_{\gtrsim \lambda^2}F_{\lambda}(s)ds\Big \Vert_{L_{t}^{\infty}L_x^2}
\end{align*}
Using properties $(2.5)$ and $(2.4)$ respectively, the first term can be bounded by $\lambda^{s} \Vert \langle \tau+|\xi|^2 \rangle^{\theta-1}\mathcal{F}(C_{\ll \lambda^2}F_{\lambda})\Vert_{L_{\tau,\xi}^2}$.\\
The second term can be further written as
\begin{align*}
\lesssim \lambda^s \Big (\Big \Vert \chi(t) \int_0^t e^{i(t-s)\Delta}C_{\sim \lambda^2}P_{\ll \lambda^2}^{(t)}F_{\lambda}(s)ds\Big \Vert_{L_{t}^{\infty}L_x^2} + \Big \Vert \chi(t)\int_0^t e^{i(t-s)\Delta}C_{\gtrsim \lambda^2}P_{\gtrsim \lambda^2}^{(t)}F_{\lambda}(s)ds\Big \Vert_{L_{t}^{\infty}L_x^2}\Big ).
\end{align*}
Using $\Vert \mathcal{I}_S[C_{\gtrsim \mu}G]\Vert_{L_{t}^{\infty}L_x^2} \lesssim \mu^{-1}\Vert C_{\gtrsim \mu}G\Vert_{L_{t}^{\infty}L_x^2}$, the first term can be bounded by $\lambda^{s-2} \Vert P_{\ll \lambda^2}^{(t)}F_{\lambda}\Vert_{L_{t}^{\infty}L_x^2}$ while for the second term properties $(2.5)$ and $(2.4)$ can be used to obtain the  bound $\lambda^{s'-2a+b}\Vert \langle \tau+|\xi|^2\rangle^{\theta'-1}\mathcal{F}((\lambda+|\partial_t|)^a C_{\gtrsim \lambda^2}F_{\lambda})\Vert_{L_{\tau,\xi}^2}$ by noting that the temporal weight $m_{S}(\tau)\gtrsim 1$ and $b\geqslant 0$.\\
The low modulation norm of the Duhamel integral is also decomposed into
\begin{align*}
&\lambda^s  \Big \Vert \langle \tau+|\xi|^2 \rangle^{\theta}\mathcal{F}(C_{\ll \lambda^2} \chi(t) \int_0^t e^{i(t-s)\Delta}C_{\ll \lambda^2}F_{\lambda}(s)ds)\Big \Vert_{L_{\tau,\xi}^2} \\
&\quad +\lambda^s  \Big \Vert \langle \tau+|\xi|^2 \rangle^{\theta}\mathcal{F}(C_{\ll \lambda^2} \chi(t) \int_0^t e^{i(t-s)\Delta}C_{\gtrsim \lambda^2}F_{\lambda}(s)ds)\Big \Vert_{L_{\tau,\xi}^2}=:\text{(I)+(II)}
\end{align*}
An application of $(2.4)$ provides the correct bound for (I). To handle (II), we consider the following cases pertaining to the high modulation norm for the Schr\"odinger non-linearity:\\
{\boldmath $ (i)~a=0=b: $} A straightforward application of $(2.4)$ gives
\begin{align*}
\text{(II)} \lesssim \lambda^s \Vert \langle \tau+|\xi|^2\rangle^{\theta-1}\mathcal{F}( C_{\gtrsim \lambda^2}F_{\lambda})\Vert_{L_{\tau,\xi}^2}.
\end{align*}
{\boldmath $ (ii)~a>0, b=0: $}
 We note that we can bound (II) by $\lambda^{s'-2a} \Vert \langle \tau+|\xi|^2\rangle^{\theta'-1}\mathcal{F}((\lambda+|\partial_t|)^a C_{\gtrsim \lambda^2}F_{\lambda})\Vert_{L_{\tau,\xi}^2}$ using $(2.4)$, provided the temporal frequencies of $F_{\lambda}$ are $\gtrsim \lambda^2$. We consider the case when the temporal frequencies of $F_{\lambda}$ are $\ll \lambda^2$. Note that this implies that $F_{\lambda}$ has a modulation of size $\sim \lambda^2$. We use $(2.9)$ and Sobolev embedding to obtain
\begin{align*}
\text{(II)}&=\lambda^{s}\Big \Vert  P_{\ll \lambda^2}^{(t)}(\chi(t) \int_0^t P_{\sim \lambda^2}^{(s)}e^{-is\Delta}F_{\lambda}(s)ds)\Big \Vert_{H_t^{\theta}L_x^2}\\
&\lesssim \lambda^{s}\Big \Vert P_{\ll \lambda^2}^{(t)}(\chi(t) \int_0^t P_{\sim \lambda^2}^{(s)}e^{-is\Delta}F_{\lambda}(s)ds)\Big \Vert_{W_t^{1,p}L_x^2}, ~~~p=\frac{2}{3-2\theta}\\
&\lesssim \lambda^{s} \Big (\Big \Vert P_{\ll \lambda^2}^{(t)}(\chi(t) \int_0^t P_{\sim \lambda^2}^{(s)}e^{-is\Delta}F_{\lambda}(s)ds)\Big \Vert_{L_t^pL_x^2}
+ \Big \Vert  P_{\ll \lambda^2}^{(t)}(\chi(t) \int_0^t P_{\sim \lambda^2}^{(s)}e^{-is\Delta}F_{\lambda}(s)ds)'\Big\Vert_{L_t^{p}L_x^2} \Big)\\
&=:\text{(II.1)+(II.2)},
\end{align*}
where the prime denotes derivative wrt time.\\
Using H\"older's inequality, unitarity of $e^{it\Delta}$ and $\Vert \mathcal{I}_S[C_{\gtrsim \mu}G]\Vert_{L_{t}^{\infty}L_x^2} \lesssim \mu^{-1}\Vert C_{\gtrsim \mu}G\Vert_{L_t^{\infty}L_x^2}$, we have
\begin{align*}
\text{(II.1)}\lesssim \lambda^{s} \Vert \chi(t)\Vert_{L_t^p L_x^{\infty}}\Big \Vert  \int_0^t P_{\sim \lambda^2}^{(s)}e^{-is\Delta}F_{\lambda}(s)ds\Big \Vert_{L_t^{\infty}L_x^2}
\lesssim \lambda^{s-2}\Vert  P_{\ll \lambda^2}^{(t)}F_{\lambda} \Vert_{L_t^{\infty}L_x^2}.
\end{align*}
(II.2) is further written as
\begin{align*}
\text{(II.2)}\lesssim \lambda^{s}\Big( \Big \Vert P_{\ll \lambda^2}^{(t)}(\chi'(t)  \int_0^t P_{\sim \lambda^2}^{(s)}e^{-is\Delta}F_{\lambda}(s)ds )\Big \Vert_{L_t^pL_x^2} + \Vert  P_{\ll \lambda^2}^{(t)}(\chi(t) P_{\sim \lambda^2}^{(t)}e^{-it\Delta}F_{\lambda})\Vert_{L_t^pL_x^2}\Big).
\end{align*}
The first term above can be handled like (II.1). For the second term, we use Bernstein's inequality and decompose the time cutoff $\chi(t)$ to obtain
\begin{align*}
&\lesssim \lambda^{s+2\theta-1}\Vert  P_{\ll \lambda^2}^{(t)}(\chi(t) P_{\sim \lambda^2}^{(t)}e^{-it\Delta}F_{\lambda})\Vert_{L_t^1L_x^2}\\
&\lesssim \lambda^{s+2\theta-1}(\Vert  P_{\ll \lambda^2}^{(t)}(P_{\ll \lambda^2}^{(t)}\chi(t) P_{\sim \lambda^2}^{(t)}e^{-it\Delta}F_{\lambda})\Vert_{L_t^1L_x^2} +\Vert  P_{\ll \lambda^2}^{(t)}(P_{\gtrsim \lambda^2}^{(t)}\chi(t) P_{\sim \lambda^2}^{(t)}e^{-it\Delta}F_{\lambda})\Vert_{L_t^1L_x^2}).
\end{align*}
The first term above does not contribute while for the second, we have
\begin{align*}
&\lesssim \lambda^{s+2\theta-1}\Vert P_{\gtrsim \lambda^2}^{(t)}\chi(t)\Vert_{L_t^1L_x^{\infty}} \Vert P_{\sim \lambda^2}^{(t)}e^{-it\Delta}F_{\lambda}\Vert_{L_t^{\infty}L_x^2}
\lesssim \lambda^{s+2\theta-3}\Vert C_{\sim \lambda^2} F_{\lambda}\Vert_{L_t^{\infty}L_x^2}\lesssim \Vert F_{\lambda}\Vert_{N^{s,l,\theta-1}_{\lambda}},
\end{align*} 
where in the last inequality we use that the $L^1$ norm of a time cut-off at high temporal frequencies $(\gtrsim \lambda^2)$ is $\lesssim \lambda^{-2}$, see \cite[Lemma 2.4]{zbMATH06488454} for a proof.\\
{\boldmath $ (iii)~b>0, a=0: $} As for case $(i)$, we have
\begin{align*}
\text{(II)}\lesssim \lambda^s \Vert \langle \tau+|\xi|^2\rangle^{\theta-1}\mathcal{F}( C_{\gtrsim \lambda^2}F_{\lambda})\Vert_{L_{\tau,\xi}^2} \lesssim \lambda^{s+2\theta-2}\Vert \mathcal{F}(C_{\gtrsim \lambda^2}F_{\lambda})\Vert_{L_{\tau,\xi}^2}\lesssim \Vert F_{\lambda}\Vert_{N^{s,l,\theta-1}_{\lambda}},
\end{align*}
where the last inequality follows by noting that $b>0$ and $s'=s+2\theta-2$ for the given case.\\
For the high modulation norm of the Duhamel integral, we again use the decomposition
\begin{align*}
&\lambda^{s'-2a+b}  \Big \Vert \langle \tau+|\xi|^2\rangle^{\theta'} \mathcal{F}((\lambda+|\partial_t|)^a C_{\gtrsim \lambda^2}\chi(t)\int_0^t e^{i(t-s)\Delta}C_{\gtrsim \lambda^2}F_{\lambda}(s)ds) \Big \Vert_{L_{\tau,\xi}^2}\\
&\quad  + \lambda^{s'-2a+b} \Big \Vert \langle \tau+|\xi|^2\rangle^{\theta'} \mathcal{F}((\lambda+|\partial_t|)^a C_{\gtrsim \lambda^2}\chi(t)\int_0^t e^{i(t-s)\Delta}C_{\ll \lambda^2}F_{\lambda}(s)ds) \Big \Vert_{L_{\tau,\xi}^2}=:\text{(III)+(IV)}
\end{align*}
For (III), $(2.4)$ suffices. For (IV), we again consider three cases:\\
{\boldmath $ (i)~a=0=b: $} From the definition of the norm and $(2.4)$,
\begin{align*}
\text{(IV)} \lesssim \lambda^{s}\Vert \langle \tau+|\xi|^2\rangle^{\theta-1}\mathcal{F}(C_{\gtrsim \lambda^2}F_{\lambda})\Vert_{L_{\tau,\xi}^2} \lesssim \Vert F_{\lambda}\Vert_{N^{s,l,\theta-1}_{\lambda}}.
\end{align*}
{\boldmath $ (ii)~ a>0, b=0: $} We prove the estimate for $a=1$. Interpolation with the case $a=0$ then leads us to the desired result. From the defintion of the $S_{\lambda}$ norm, we have for $a=1$, using $(2.9)$ and Sobolev embedding
\begin{align*}
\text{(IV)}&= \lambda^{s-2}\Big \Vert \langle \tau+|\xi|^2\rangle^{\theta} \mathcal{F}((\lambda+|\partial_t|) C_{\gtrsim \lambda^2}\chi(t)\int_0^t e^{i(t-s)\Delta}C_{\ll \lambda^2}F_{\lambda}(s)ds) \Big \Vert_{L_{\tau,\xi}^2}\\
& = \lambda^{s-2} \Big \Vert (\lambda+|\partial_t|) (P_{\gtrsim \lambda^2}^{(t)}\chi(t)\int_0^t P_{\ll \lambda^2}^{(s)}e^{-is\Delta}F_{\lambda}(s)ds)\Big\Vert_{H_t^{\theta}L_x^2}\\
& \lesssim \lambda^{s-2} \Big \Vert (\lambda+|\partial_t|) (P_{\gtrsim \lambda^2}^{(t)}\chi(t)\int_0^t P_{\ll \lambda^2}^{(s)}e^{-is\Delta}F_{\lambda}(s)ds)\Big\Vert_{W^{1,p}_t L_x^2}, ~~p=\frac{2}{3-2\theta}\\
&\lesssim \lambda^{s-2}\Big \Vert (\lambda+|\partial_t|) (P_{\gtrsim \lambda^2}^{(t)}\chi(t)\int_0^t P_{\ll \lambda^2}^{(s)}e^{-is\Delta}F_{\lambda}(s)ds)\Big\Vert_{L_t^p L_x^2}\\
&\quad  + \lambda^{s-2}\Big \Vert (\lambda+|\partial_t|) (P_{\gtrsim \lambda^2}^{(t)}\chi(t)\int_0^t P_{\ll \lambda^2}^{(s)}e^{-is\Delta}F_{\lambda}(s)ds)'\Big\Vert_{L_t^p L_x^2}=:\text{(IV.1)+(IV.2)},
\end{align*}
where the prime denotes derivative wrt $t$.\\
Using the product estimate for $a=1,\frac{1}{p}=\frac{1}{2}+\frac{1}{q}$, we have
\begin{align*}
\text{(IV.1)}&\lesssim \lambda^{s-3}\Vert (\lambda+|\partial_t|)\chi(t)\Vert_{L_t^qL_x^{\infty}}\Big \Vert (\lambda+|\partial_t|)\int_0^t P_{\ll \lambda^2}^{(s)}e^{-is\Delta}F_{\lambda}(s)ds\Big \Vert_{L_{t,x}^2} \lesssim \lambda^s \Vert \langle \tau+|\xi|^2 \rangle^{\theta-1}\mathcal{F}(C_{\ll \lambda^2}F_{\lambda})\Vert_{L_{\tau,\xi}^2}.
\end{align*}
On applying the product rule to (IV.2), we find that the first term is similar to (IV.1). Using the product estimate for $\frac{1}{p}=\frac{1}{q}+\frac{1}{2}$, the second term can be bounded by
\begin{align*}
&\lambda^{s-2}\Vert (\lambda+|\partial_t|)P_{\gtrsim \lambda^2}^{(t)}(\chi(t) P_{\ll \lambda^2}^{(t)}e^{-it\Delta}F_{\lambda})\Vert_{L_t^pL_x^2}\\
& \lesssim \lambda^{s-3}\Vert (\lambda+|\partial_t|)\chi(t)\Vert_{L_t^qL_x^{\infty}} \Vert (\lambda+|\partial_t|) P_{\ll \lambda^2}^{(t)}e^{-it\Delta}F_{\lambda}\Vert_{L_{t,x}^2} \lesssim \lambda^{s-1}\Vert C_{\ll \lambda^2}F_{\lambda}\Vert_{L_{t,x}^2} \lesssim \Vert F_{\lambda}\Vert_{N^{s,l,\theta-1}_{\lambda}}.
\end{align*}
{\boldmath $ (iii)~ b>0, a=0: $}
From the defintion of the $S_{\lambda}$ norm in the case $b>0$ and using $(2.9)$, with the prime denoting derivative wrt time, we have
\begin{align*}
\text{(IV)}&\lesssim \lambda^{s+2\theta-2+b}\Big( \Big \Vert P_{\gtrsim \lambda^2}^{(t)}(\chi(t) \int_0^t P_{\ll \lambda^2}^{(s)}e^{-is\Delta}F_{\lambda}(s)ds) \Big \Vert_{L_{t,x}^2} + \Big \Vert (P_{\gtrsim \lambda^2}^{(t)}(\chi(t)\int_0^t P_{\ll \lambda^2}^{(s)}e^{-is\Delta}F_{\lambda}(s)ds))'\Big \Vert_{L_{t,x}^2}\Big)\\
& \quad=: \text{(IV.3)+(IV.4)}.
\end{align*}
Unitarity of $e^{it\Delta}$ and an application of $(2.4)$ gives
\begin{align*}
\text{(IV.3)}	\lesssim \lambda^{s+2\theta-2+b} \Big \Vert e^{-it\Delta} \chi (t) \int_0^t e^{i(t-s)\Delta}C_{\ll \lambda^2}F_{\lambda}(s)ds \Big\Vert_{L_{t,x}^2}
%	&= \lambda^{s+b+2\theta-2}\Vert \chi(t) \mathcal{I}_S[C_{\ll \lambda^2}F_{\lambda}]\Vert_{X^{0,0}}\\
&\lesssim \lambda^{s+2\theta-2+b} \Vert \langle \tau+|\xi|^2\rangle^{-1}\mathcal{F}(C_{\ll \lambda^2}F_{\lambda})\Vert_{L_{\tau,\xi}^2}\\
&\lesssim \lambda^{s+2\theta-2+b} \Vert \langle \tau+|\xi|^2\rangle^{\theta-1}\mathcal{F}(C_{\ll \lambda^2}F_{\lambda})\Vert_{L_{\tau,\xi}^2}\\
&\lesssim \lambda^{s} \Vert \langle \tau+|\xi|^2\rangle^{\theta-1}\mathcal{F}(C_{\ll \lambda^2}F_{\lambda})\Vert_{L_{\tau,\xi}^2}.
\end{align*}
Using the product rule, we have
\begin{align*}
\text{(IV.4)}&\lesssim \lambda^{s+b+2\theta-2}\Big(\Big \Vert P_{\gtrsim \lambda^2}^{(t)}(\chi'(t)\int_0^t P_{\ll \lambda^2}^{(s)}e^{-is\Delta}F_{\lambda}(s)ds)\Big \Vert_{L_{t,x}^2} +  \Vert P_{\gtrsim \lambda^2}^{(t)}(\chi(t) P_{\ll \lambda^2}^{(t)}e^{-it\Delta}F_{\lambda})\Vert_{L_{t,x}^2}\Big)\\
&\quad =:\text{(IV.41)+(IV.42)}
\end{align*}
(IV.41) can be handled exactly in the same way as (IV.3) while for (IV.42), we decompose the time cutoff
\begin{align*}
\text{(IV.42)}&\lesssim \lambda^{s+b+2\theta-2} \big(\Vert P_{\gtrsim \lambda^2}^{(t)}(P_{\ll \lambda^2}^{(t)}\chi(t) P_{\ll \lambda^2}^{(t)}e^{-it\Delta}F_{\lambda})\Vert_{L_{t,x}^2} + \Vert P_{\gtrsim \lambda^2}^{(t)}(P_{\gtrsim \lambda^2}^{(t)}\chi(t) P_{\ll \lambda^2}^{(t)}e^{-it\Delta}F_{\lambda})\Vert_{L_{t,x}^2 }\big).
\end{align*}
The first term does not contribute while we bound the second term by
\begin{align*}
&\lesssim \lambda^{s+b+2\theta-2} \Vert (P_{\gtrsim \lambda^2}^{(t)}\chi(t) - \chi(t) P_{\gtrsim \lambda^2}^{(t)})P_{\ll \lambda^2}^{(t)}e^{-it\Delta}F_{\lambda}\Vert_{L_{t,x}^2} + \Vert \chi(t) P_{\gtrsim \lambda^2}^{(t)}P_{\ll \lambda^2}^{(t)}e^{-it\Delta}F_{\lambda}\Vert_{L_{t,x}^2}\\
&\lesssim \lambda^{s+b+2\theta-4}\Vert P_{\ll \lambda^2}^{(t)}e^{-it\Delta}F_{\lambda}\Vert_{L_{t,x}^2}
\lesssim \lambda^{s+2\theta-3}\Vert e^{-it\Delta}C_{\ll \lambda^2}F_{\lambda}\Vert_{L_{t,x}^2}
\lesssim \lambda^s \Vert \langle \tau+|\xi|^2\rangle^{\theta-1}\mathcal{F}(C_{\ll \lambda^2}F_{\lambda})\Vert_{L_{\tau,\xi}^2},
\end{align*}
noting that the second term in the first display above vanishes.
\end{proof}

	\begin{lemma}(Energy inequality for the wave equation) Let $l,\beta,\theta,a \in \mathbb{R}$ be as defined in $(2.2)$. For any $\lambda \in 2^{\mathbb{N}}$ and a smooth time cutoff $\chi$, we have:
	\begin{align*}
	&\Vert \chi(t) e^{it|\nabla|}v_0\Vert_{W^{l,s,\theta}_{\lambda}} \lesssim \lambda^l\Vert v_0\Vert_{L_x^2},\\
	&\Vert \chi(t)\mathcal{I}_{W}[G_{\lambda}]\Vert_{W^{l,s,\theta}_{\lambda}} \lesssim \Vert G_{\lambda}\Vert_{R^{l,s,\theta-1}_{\lambda}}.
	\end{align*}
\end{lemma}
\begin{proof}
We will be short here as most of the steps will be same as in Lemma 2.2.
	\begin{align*}
	\Vert \chi(t) e^{it|\nabla|}v_0 \Vert_{W^{l,s,\theta}_{\lambda}} &= \lambda^l \Vert   \chi(t) e^{it|\nabla|}v_0  \Vert_{L_t^{\infty}L_x^2} + \lambda^{l-a}\Vert \langle \tau-|\xi|\rangle^{\theta}\mathcal{F}((\lambda+|\partial_t|)^a Q_{\ll \lambda^2} \chi(t)e^{it|\nabla|}v_0)\Vert_{L_{\tau,\xi}^2} \\
	&+ \lambda^{\beta-1}\Vert \langle \tau-|\xi|\rangle \mathcal{F}(Q_{\gtrsim \lambda^2}\chi(t)e^{it|\nabla|}v_0)\Vert_{L_{\tau,\xi}^2}
	\end{align*}
The first term above can be controlled using properties $(2.6)$ and $(2.3)$ respectively.
	For the second term, we have using $(2.9)$
	\begin{align*}
	& \lambda^{l-a}\Vert \langle \tau-|\xi|\rangle^{\theta}\mathcal{F}((\lambda+|\partial_t|)^a Q_{\ll \lambda^2} \chi(t)e^{it|\nabla|}v_0)\Vert_{L_{\tau,\xi}^2} \\
%	&= \lambda^{l-a}\Vert \langle \tau-|\xi|\rangle^{\theta}\mathcal{F}((\lambda+|\partial_t|)^a e^{it|\nabla|}P_{\ll \lambda^2}^{(t}\chi(t)v_0\Vert_{L_{\tau,\xi}^2}\\
	&=\lambda^{l-a}\Vert \langle \tau-|\xi|\rangle^{\theta}(\lambda+|\tau|)^a \mathcal{F}( e^{it|\nabla|}P_{\ll \lambda^2}^{(t)}(\chi(t)v_0))\Vert_{L_{\tau,\xi}^2}\\
	&\lesssim \max \big \{\lambda^{l-a} \Vert \langle \tau-|\xi|\rangle^{\theta+2a}\mathcal{F}(e^{it|\nabla|}P_{\ll \lambda^2}^{(t)}(\chi(t)v_0))\Vert_{L_{\tau,\xi}^2}, \lambda^l \Vert \langle \tau-|\xi|\rangle^{\theta} \mathcal{F}(e^{it|\nabla|}P_{\ll \lambda^2}^{(t)}(\chi(t)v_0))\Vert_{L_{\tau,\xi}^2}\big \}\\
	&\lesssim \lambda^l \Vert P_{\ll \lambda^2}^{(t)}(\chi(t)v_0)\Vert_{H_t^{\theta+2a}L_x^2}
	%	&\lesssim \lambda^l \Vert \chi(t)\Vert_{H_t^{\theta+2a}L_x^2} \Vert v_0\Vert_{L_x^2}\\
	\lesssim \lambda^l \Vert v_0\Vert_{L_x^2}.
	\end{align*}
	Similarly, for the last term, we have from the choice of the parameter $\beta$
	\begin{align*}
	\lambda^{\beta-1}\Vert \langle \tau-|\xi|\rangle \mathcal{F}(Q_{\gtrsim \lambda^2}\chi(t)e^{it|\nabla|}v_0)\Vert_{L_{\tau,\xi}^2} 
	&= \lambda^{\beta-1}\Vert \langle \tau-|\xi|\rangle \mathcal{F}(e^{it|\nabla|}P_{\gtrsim \lambda^2}^{(t)}(\chi(t)v_0))\Vert_{L_{\tau,\xi}^2}\\
	&=\lambda^{\beta-1}\Vert P_{\gtrsim \lambda^2}^{(t)}(\chi(t)v_0)\Vert_{H_t^1L_x^2}
	%&\lesssim \lambda^{\beta-1}\Vert \chi(t)\Vert_{H_t^1}\Vert v_0\Vert_{L_x^2}\\
	\lesssim \lambda^l \Vert v_0\Vert_{L_x^2}.
	\end{align*}
	We consider the $W_{\lambda}$ norm for the Duhamel integral now. The non-linearity $G_{\lambda}$ in the $L_t^{\infty}L_x^2$ term is decomposed into high and low modulation and is treated exactly as in Lemma 2.3.
	The low modulation norm of the  Duhamel integral is written as
	\begin{align*}
	&\lesssim \lambda^{l-a}\Big \Vert \langle \tau-|\xi|\rangle^{\theta}\mathcal{F}((\lambda+|\partial_t|)^a Q_{\ll \lambda^2}(\chi(t) \int_0^t e^{i(t-s)|\nabla|}Q_{\ll \lambda^2}G_{\lambda}(s)ds))\Big \Vert_{L_{\tau,\xi}^2} \\
	&+ \lambda^{l-a}\Big \Vert \langle \tau-|\xi|\rangle^{\theta}\mathcal{F}((\lambda+|\partial_t|)^a Q_{\ll \lambda^2}(\chi(t) \int_0^t e^{i(t-s)|\nabla|}Q_{\gtrsim \lambda^2}G_{\lambda}(s)ds))\Big\Vert_{L_{\tau,\xi}^2}=:\text{(I.1)+(I.2)}.
	\end{align*}
	(I.1) is controlled using $(2.4)$. For (I.2), we consider two cases:\\
	{\boldmath $(i)~a>0$:} Using $(2.4)$ and the choice of the parameters $a$ and $\beta$, we have
	\begin{align*}
	\text{(I.2)}& 
	%&=\lambda^{l-a}\Vert \langle \tau-|\xi|\rangle^{\theta}\mathcal{F}((\lambda+|\partial_t|)^a Q_{\ll \lambda^2}(\chi(t) \int_0^t e^{i(t-s)|\nabla|}Q_{\gtrsim \lambda^2}G_{\lambda}(s)ds))\Vert_{L_{\tau,\xi}^2} \\
	\lesssim \lambda^{l-a}\Big\Vert \langle \tau-|\xi|\rangle^{\theta} (\lambda+|\tau|)^a \mathcal{F}(Q_{\ll \lambda^2} \chi(t) \int_0^t e^{i(t-s)|\nabla|}Q_{\gtrsim \lambda^2}G_{\lambda}(s)ds)\Big\Vert_{L_{\tau,\xi}^2}\\
	&	\lesssim \lambda^{l+a+2\theta}\Vert \langle \tau-|\xi|\rangle^{-1}\mathcal{F}(Q_{\gtrsim \lambda^2}G_{\lambda})\Vert_{L_{\tau,\xi}^2}
	\lesssim \lambda^{l+a+2\theta-2}\Vert \mathcal{F}(Q_{\gtrsim \lambda^2}G_{\lambda})\Vert_{L_{\tau,\xi}^2} 
	\lesssim \lambda^{\beta-1} \Vert \mathcal{F}(Q_{\gtrsim \lambda^2}G_{\lambda})\Vert_{L_{\tau,\xi}^2}.
	\end{align*}
	{\boldmath $(ii)~a=0$:}
Using $(2.9)$ and Sobolev embedding, (I.2) can be bounded by
	\begin{align*}
	%\lambda^{l}\Vert \langle \tau-|\xi|\rangle^{\theta}\mathcal{F}( Q_{\ll \lambda^2}(\chi(t) \int_0^t e^{i(t-s)|\nabla|}Q_{\ll \lambda^2}G_{\lambda}(s)ds))\Vert_{L_{\tau,\xi}^2} &
	%  = \lambda^l \Vert P_{\ll \lambda^2}^{(t)} (\chi\int_0^t P_{\gtrsim \lambda^2}^{(t)}e^{-is|\nabla|}G_{\lambda}(s)ds)\Vert_{H_t^{\theta}L_x^2} \\
	\lambda^{l}\Big \Vert P_{\ll \lambda^2}^{(t)} (\chi(t)\int_0^t P_{\gtrsim \lambda^2}^{(s)}e^{-is|\nabla|}G_{\lambda}(s)ds)\Big\Vert_{W_t^{1,p}L_x^2}, ~~~p=\frac{2}{3-2\theta},
	\end{align*}
	which is written in equivalent norm as
	\begin{align*}
	\lambda^l \Big(\Big\Vert P_{\ll \lambda^2}^{(t)} (\chi(t)\int_0^t P_{\gtrsim \lambda^2}^{(s)}e^{-is|\nabla|}G_{\lambda}(s)ds)\Big\Vert_{L_t^pL_x^2} + \Big\Vert (P_{\ll \lambda^2}^{(t)} (\chi(t)\int_0^t P_{\gtrsim \lambda^2}^{(s)}e^{-is|\nabla|}G_{\lambda}(s)ds))'\Big \Vert_{L_t^pL_x^2}\Big)=:\text{(I.21)+(I.22)}
	\end{align*} 
	\begin{align*}
	\text{(I.21)}&\lesssim   \lambda^l \Vert \chi(t)\Vert_{L_t^p L_x^{\infty}} \Big\Vert \int_0^t P_{\gtrsim \lambda^2}^{(s)}e^{-is|\nabla|}G_{\lambda}(s)ds\Big\Vert_{L_t^{\infty} L_x^2} \lesssim \lambda^{l-2}\Vert Q_{\gtrsim \lambda^2}G_{\lambda}\Big\Vert_{L_t^{\infty}L_x^2}.
	\end{align*}
	\begin{align*}
	\text{(I.22)}&\lesssim \lambda^l\Big \Vert (P_{\ll \lambda^2}^{(t)} (\chi'(t)\int_0^t P_{\gtrsim \lambda^2}^{(s)}e^{-is|\nabla|}G_{\lambda}(s)ds)\Big\Vert_{L_t^p L_x^2} + \Vert P_{\ll \lambda^2}^{(t)} (\chi(t) P_{\gtrsim \lambda^2}^{(t)}e^{-it|\nabla|}G_{\lambda})\Vert_{L_t^pL_x^2}=: \text{(I.221)+(I.222)}
	\end{align*}
	(I.221) can be treated like (I.21) while for (I.222), we decompose the time cutoff to get
	%	\begin{align*}
	%		B_1 &\lesssim \chi_T'(t)\Vert_{L_t^p L_x^{\infty}} \Vert \int_0^t P_{\gtrsim \lambda^2}^{(t)}e^{-is|\nabla|}G_{\lambda}(s)ds)\Vert_{L_t^{\infty} L_x^2} \lesssim \lambda^{l-2}\Vert Q_{\gtrsim \lambda^2}G_{\lambda}\Vert_{L_t^{\infty}L_x^2}
	%	\end{align*}
	\begin{align*}
	\text{(I.222)} &\lesssim \lambda^l \Vert P_{\ll \lambda^2}^{(t)}(P_{\ll \lambda^2}^{(t)}\chi(t) + P_{\gtrsim \lambda^2}^{(t)}\chi(t)) P_{\gtrsim \lambda^2}^{(t)}e^{-it|\nabla|}G_{\lambda}\Vert_{L_t^pL_x^2}.
	\end{align*}
	The first term in the above sum vanishes and for the second, we use Bernstein's and H\"older's inequality to obtain the bound
	\begin{align*}
	\lambda^l \Vert P_{\ll \lambda^2}^{(t)}(P_{\gtrsim \lambda^2}^{(t)}\chi(t) P_{\gtrsim \lambda^2}^{(t)}e^{-it|\nabla|}G_{\lambda})\Vert_{L_t^pL_x^2}& \lesssim \lambda^{l+2\theta-1}\Vert P_{\ll \lambda^2}^{(t)}(P_{\gtrsim \lambda^2}^{(t)}\chi(t) P_{\gtrsim \lambda^2}^{(t)}e^{-it|\nabla|}G_{\lambda})\Vert_{L_t^1L_x^2}\\
	&\lesssim \lambda^{l+2\theta-3}\Vert Q_{\gtrsim \lambda^2}G_{\lambda}\Vert_{L_t^{\infty}L_x^2}.
	\end{align*}
	The high modulation Duhamel integral is again decomposed with the first term controlled using $(2.4)$ i.e 
	\begin{align*}
	\lambda^{\beta-1}\Big \Vert \langle \tau-|\xi|\rangle \mathcal{F}(Q_{\gtrsim \lambda^2}(\chi (t) \int_0^t e^{i(t-s)|\nabla|}Q_{\gtrsim \lambda^2}G_{\lambda}(s)ds))\Big\Vert_{L_{\tau,\xi}^2} \lesssim \lambda^{\beta-1}\Vert \mathcal{F}(Q_{\gtrsim \lambda^2}G_{\lambda})\Vert_{L_{\tau,\xi}^2}.
	\end{align*}
	For the other term, we have
	\begin{align*}
	&\lambda^{\beta-1} \Big \Vert \langle \tau-|\xi|\rangle \mathcal{F}(Q_{\gtrsim \lambda^2}(\chi(t) \int_0^t e^{i(t-s)|\nabla|}Q_{\ll \lambda^2}G_{\lambda}(s)ds))\Big\Vert_{L_{\tau,\xi}^2}\\
	&= \lambda^{\beta-1}\Big\Vert P_{\gtrsim \lambda^2}^{(t)}(\chi(t) \int_0^t P_{\ll \lambda^2}^{(s)}e^{-is|\nabla|}G_{\lambda}(s)ds) \Big\Vert_{H_t^1 L_x^2}\\
	&\lesssim  \lambda^{\beta-1}\Big\Vert P_{\gtrsim \lambda^2}^{(t)}(\chi(t) \int_0^t P_{\ll \lambda^2}^{(s)}e^{-is|\nabla|}G_{\lambda}(s)ds))\Big\Vert_{L_{t,x}^2} + \lambda^{\beta-1}\Big\Vert P_{\gtrsim \lambda^2}^{(t)}(\chi(t) \int_0^t P_{\ll \lambda^2}^{(s)}e^{-is|\nabla|}G_{\lambda}(s)ds)'\Big\Vert_{L_{t,x}^2}\\
	& =: \text{(II)+(III)}
	\end{align*}
	Then, using $(2.4)$, the choice of $\beta$ and $|m_W| \gtrsim 1$,
	\begin{align*}
	\text{(II)} \lesssim \lambda^{\beta-1} \Big\Vert \chi(t) \int_0^t P_{\ll \lambda^2}^{(s)}e^{-is|\nabla|}G_{\lambda}(s)ds\Big\Vert_{L_{t,x}^2} 
	%&= \lambda^{\beta-1}\Vert e^{-it|\nabla|}\chi (t)\int_0^t e^{i(t-s)|\nabla|}Q_{\ll \lambda^2}G_{\lambda}(s)ds\Vert_{L_{t,x}^2}\\
	&\lesssim \lambda^{\beta-1}\Vert\langle \tau-|\xi|\rangle^{-1} \mathcal{F}(Q_{\ll \lambda^2}G_{\lambda})\Vert_{L_{\tau,\xi}^2}\\
	&\lesssim \lambda^{l-a} \Vert \langle \tau-|\xi|\rangle^{\theta-1}\mathcal{F}((\lambda+|\partial_t|)^a Q_{\ll \lambda^2}G_{\lambda})\Vert_{L_{\tau,\xi}^2}.
	\end{align*}
	\begin{align*}
	\text{(III)}&\lesssim \lambda^{\beta-1}\Big\Vert P_{\gtrsim \lambda^2}^{(t)}(\chi'(t)\int_0^t P_{\ll \lambda^2}^{(s)}e^{-is|\nabla|}G_{\lambda}(s)ds)\Big\Vert_{L_{t,x}^2} + \lambda^{\beta-1} \Vert P_{\gtrsim \lambda^2}^{(t)}(\chi(t) P_{\ll \lambda^2}^{(t)} e^{-it|\nabla|}G_{\lambda})\Vert_{L_{t,x}^2}.
	\end{align*}
	The first term of (III) is treated like (II) and the second decomposed as follows
	\begin{align*}
	\lambda^{\beta-1}\Vert (P_{\gtrsim \lambda^2}^{(t)} (P_{\ll \lambda^2}^{(t)}\chi(t) P_{\ll \lambda^2}^{(t)}e^{-it|\nabla|}G_{\lambda})\Vert_{L_{t,x}^2} + \lambda^{\beta-1}\Vert (P_{\gtrsim \lambda^2}^{(t)} (P_{\gtrsim \lambda^2}^{(t)}\chi(t) P_{\ll \lambda^2}^{(t)}e^{-it|\nabla|}G_{\lambda})\Vert_{L_{t,x}^2}.
	\end{align*}
	The first term above does not contribute and for the second, we have
	\begin{align*}
%	& \lambda^{\beta-1}\Vert (P_{\gtrsim \lambda^2}^{(t)} (P_{\gtrsim \lambda^2}^{(t)}\chi(t) P_{\ll \lambda^2}^{(t)}e^{-it|\nabla|}G_{\lambda})\Vert_{L_{t,x}^2}\\
	&\lesssim \lambda^{\beta-1}\Vert (P_{\gtrsim \lambda^2}^{(t)}\chi(t) - \chi(t) P_{\gtrsim \lambda^2}^{(t)})   P_{\ll \lambda^2}^{(t)}e^{-it|\nabla|}G_{\lambda})\Vert_{L_{t,x}^2} + \lambda^{\beta-1}\Vert \chi(t) P_{\gtrsim \lambda^2}^{(t)} P_{\ll \lambda^2}^{(t)}e^{-it|\nabla|}G_{\lambda})\Vert_{L_{t,x}^2}\\
	&\lesssim \lambda^{\beta-3}\Vert  P_{\ll \lambda^2}^{(t)}e^{-it|\nabla|}G_{\lambda})\Vert_{L_{t,x}^2}
	\lesssim \lambda^{l-a}\Vert \langle \tau-|\xi|\rangle^{\theta-1}\mathcal{F}((\lambda+|\partial_t|)^a Q_{\ll \lambda^2}G_{\lambda})\Vert_{L_{\tau,\xi}^2},
	\end{align*}
	where the choice of the parameter $\beta$  and $|m_W| \gtrsim 1$ ensure that the last inequality holds.
\end{proof}

\subsection{Multilinear estimates} We recall some already known estimates in this section.

\begin{lemma}(\cite[Proposition 4.4]{zbMATH05555745}, \cite[Corollary 3.6]{zbMATH05964196})
	Let $d\in \{2,3\}$ and $f,g_1,g_2\in L^2(\mathbb{R}^{d+1})$ be such that $\Vert f\Vert_{L^2} = \Vert g_1\Vert_{L^2} = \Vert g_2\Vert_{L^2}=1$. For $k=1,2$ let
	\begin{equation*}
	supp(f)\subset \Big\{(\tau,\xi)\in \mathbb{R}\times \mathbb{R}^d: \frac{\lambda}{2}\leqslant |\xi|\leqslant 2\lambda\Big\}~\cap ~\mathcal{W}_L^{\pm}, \quad supp(g_k)\subset \Big\{(\tau,\xi)\in \mathbb{R} \times \mathbb{R}^d: \frac{\lambda_k}{2}\leqslant |\xi|\leqslant 2\lambda_k\Big\}\cap ~\mathcal{S}_{L_k},
	\end{equation*}
	where the frequencies $\lambda, \lambda_1, \lambda_2$ and the modulations $L,L_1,L_2$ satisfy
	\begin{equation*}
	1\ll  \lambda\lesssim \lambda_1\sim \lambda_2, \quad L,L_1,L_2\lesssim \lambda_1^2.
	\end{equation*}
	Then, for 
	\begin{equation}
	I(f,g_1,g_2) = \int f(\zeta_1-\zeta_2)g_1(\zeta_1)g_2(\zeta_2)d\zeta_1\zeta_2, \quad \zeta_i=(\tau_i,\xi_i), i=1,2,
	\end{equation}the following estimate holds
	\begin{equation}
	|I(f,g_1,g_2)|\lesssim \frac{{(LL_1L_2)}^{\frac{1}{2}}}{{ \lambda_1}^{\frac{1}{2}}}\log{\lambda_1}.
	\end{equation}
\end{lemma}
\begin{remark}
	For $d=2$, there is no $\log$ term in the RHS of $(2.11)$ but this does not affect the following analysis.
\end{remark}

	\begin{lemma}(Bilinear Strichartz estimates)(\cite[Proposition 4.3]{zbMATH05555745}, \cite[Proposition 3.3]{zbMATH05964196}) Let $\hat{u_i}, 
	\hat{v}$ denote the space-time Fourier transforms $u_i, v$ respectively for  $u_i, v \in L^2(\mathbb{R}^{d+1}), ~ i=1, 2$.\\
	\textbf{(i)} Let $d\in \{2,3\}$. Let $u_i$ be dyadically Fourier-localised such that 
	\begin{equation*}
	supp(\hat{u_i}) \in \Big\{(\tau,\xi)\in \mathbb{R} \times \mathbb{R}^d :~\frac{\lambda_i}{2}\leqslant |\xi|\leqslant 2\lambda_i\Big\}\cap \mathcal{S}_{L_i}
	\end{equation*}
	for $L_i, \lambda_i \geqslant 1$. Then the following estimate holds
	\begin{equation}
	\Vert u_1 u_2\Vert_{L_{t,x}^2} \lesssim \lambda_1^{\frac{d-1}{2}} \lambda_2^{-\frac{1}{2}} (L_1L_2)^{\frac{1}{2}}\Vert u_1\Vert_{L_{t,x}^2}\Vert u_2\Vert_{L_{t,x}^2}.
	\end{equation}
	\textbf{(ii)} Let $d\leqslant 3$. Let $u,v$ be dyadically Fourier-localised such that
	\begin{equation*}
	supp(\hat{u}) \in \Big\{(\tau,\xi)\in \mathbb{R} \times \mathbb{R}^d :~\frac{\lambda_1}{2}\leqslant |\xi|\leqslant 2\lambda_1\Big\}\cap \mathcal{S}_{L_1}, ~~supp(\hat{v}) \in \Big\{(\tau,\xi)\in \mathbb{R} \times \mathbb{R}^d :~\frac{\lambda_2}{2}\leqslant |\xi|\leqslant 2\lambda_2\Big\}\cap \mathcal{W}_{L_2}^{\pm}
	\end{equation*}
	for $L_i, \lambda_i \geqslant 1$. Then the following estimate holds
	\begin{equation}
	\Vert u v\Vert_{L_{t,x}^2} \lesssim \min\{\lambda_1,\lambda_2\}^{\frac{d-1}{2}} \lambda_1^{-\frac{1}{2}} (L_1L_2)^{\frac{1}{2}}\Vert u\Vert_{L_{t,x}^2}\Vert v\Vert_{L_{t,x}^2}.
	\end{equation}
	\end{lemma}

\begin{remark}
(i)	If the frequencies $\lambda_1$ and $\lambda_2$ are such that $\lambda_1\ll \lambda_2$ or $\lambda_2\ll \lambda_1$, then the  estimate $(2.12)$ holds for $d=1$ as well.\\
(ii) The estimates $(2.12)$ and $(2.13)$ remain valid if we replace the functions on the LHS by their complex conjugates.
\end{remark}

\begin{lemma}(Product estimate)\cite[Lemma 2.7]{candy2019zakharov}
	Let $a \in \mathbb{R},  1 \leqslant {\tilde{p}}, {\tilde{q}}, {\tilde{r}}, p, q, r \leqslant \infty$ with $\frac{1}{p} = \frac{1}{q} + \frac{1}{r}$ and  $\frac{1}{\tilde{p}} = \frac{1}{\tilde{q}} + \frac{1}{\tilde{r}}$. Then, for all $\mu >0$,
	\begin{equation*}
	\Vert (\mu +|\partial_t|)^a(vu)\Vert_{L_t^{\tilde{p}}L_x^p} \lesssim \mu^{-|a|} \Vert (\mu +|\partial_t|)^{|a|}v \Vert_{L_t^{\tilde{r}}L_x^r}  \Vert (\mu +|\partial_t|)^{|a|}u \Vert_{L_t^{\tilde{q}}L_x^q}.
	\end{equation*} 
\end{lemma}
\noindent With all the required tools at our disposal, we head to prove the multilinear estimates.

\section{Multilinear estimates for Schr{\"o}dinger non-linearity}
\begin{theorem}Let $d\leqslant 3$ and $s,l$ in the range $(1.3)$. There exist $a,b,s,l,\beta,\theta \in \mathbb{R}$ such that the estimate 
	\begin{equation}
	\Vert u Re(v)\Vert_{N^{s,l,\theta-1}} \lesssim \Vert u\Vert_{S^{s,l,\theta}}\Vert v\Vert_{W^{l,s,\theta}}
	\end{equation}
	holds.
\end{theorem}
\begin{proof}

 We choose the parameters $a,b,s',l,\beta,\theta'$ as in $(2.2)$ and begin by noting the following characterisation of the $N^{s,l,\theta-1}$ norm for the non-linearity in the case $0\leqslant a< \frac{1}{2}$:
\begin{equation}
\Vert F_{\lambda} \Vert_{N_{\lambda}^{s,l,\theta-1}} \approx  \lambda^{s}\Vert \langle \tau+|\xi|^2\rangle^{\theta-1} \mathcal{F}(C_{\ll \lambda^2}F_{\lambda})\Vert_{L_{\tau,\xi}^2}
+\lambda^{s'-2a+b}\Vert \langle \tau+|\xi|^2\rangle^{\theta'-1}\mathcal{F}((\lambda+|\partial_t|)^a C_{\gtrsim \lambda^2}F_{\lambda})\Vert_{L_{\tau,\xi}^2}.
\end{equation}For $\frac{1}{r} = \frac{1}{2}-a$, using Bernstein's inequality and Sobolev embedding, we have
\begin{align}
\lambda^{s+2\theta-3}	\Vert P_{\ll \lambda^2}^{(t)} F_{\lambda} \Vert_{L_t^{\infty} L_x^2}\lesssim \lambda^{s+2\theta-3+\frac{2}{r}}\Vert P_{\ll \lambda^2}^{(t)}F_{\lambda}\Vert_{L_t^r L_x^2}
&\lesssim \lambda^{s+2\theta-2-a}\Vert P_{\ll \lambda^2}^{(t)}F_{\lambda}\Vert_{L_{t,x}^2}\\ \nonumber
&\lesssim \lambda^{s+2\theta-2-2a}\Vert (\lambda+|\partial_t|)^a P_{\ll \lambda^2}^{(t)}F_{\lambda}\Vert_{L_{t,x}^2}\\ \nonumber
&\lesssim \lambda^{s+2\theta-2-2a+b}\Vert \mathcal{F}((\lambda+|\partial_t|)^a P_{\ll \lambda^2} \nonumber
^{(t)}C_{\approx \lambda^2}F_{\lambda})\Vert_{L_{\tau,\xi}^2}\\ \nonumber
&\lesssim \lambda^{s'-2a+b}\Vert \langle \tau+|\xi|^2\rangle^{\theta'-1}\mathcal{F}((\lambda+|\partial_t|)^a C_{\gtrsim \lambda^2}F_{\lambda})\Vert_{L_{\tau,\xi}^2 },
\end{align}
where we obtain the second last inequality by noting that $b\geqslant 0$.\\
So, for $0\leqslant a <\frac{1}{2}$, it suffices to show:
\begin{equation}
\Big(\sum_{\lambda_0\in 2^{\mathbb{N}}}\lambda_0^{2s} \Vert \langle \tau+|\xi|^2\rangle^{\theta-1} \mathcal{F}(C_{\ll\lambda_0^2} P_{\lambda_0}(uv))\Vert_{L_{\tau,\xi}}^2\Big)^{1/2} \lesssim \Vert u \Vert_{S^{s,l,\theta}} \Vert v \Vert_{W^{l,s,\theta}}
\end{equation}
	\begin{equation}
\Big(\sum_{\lambda_0\in 2^{\mathbb{N}}}\lambda_0^{2(s'-2a+b)} \Vert \langle \tau+|\xi|^2\rangle^{\theta'-1} \mathcal{F}((\lambda_0+|\partial_t|)^a C_{\gtrsim\lambda_0^2} P_{\lambda_0}(uv))\Vert_{L_{\tau,\xi}^2}^2\Big)^{1/2} \lesssim \Vert u \Vert_{S^{s,l,\theta}} \Vert v \Vert_{W^{l,s,\theta}}.
\end{equation}
In the case $\frac{1}{2}\leqslant a < 1$, we additionally need to prove the following:
\begin{equation}
\Big(\sum_{\lambda_0\in 2^{\mathbb{N}}}\lambda_0^{2(s+2\theta-3)} \Vert P_{\ll \lambda_0^2}^{(t)}P_{\lambda_0}(uv)\Vert_{L_{t}^{\infty}L_x^2}^2\Big)^{\frac{1}{2}} \lesssim \Vert u \Vert_{L_{t}^{\infty}H_x^s} \Vert v \Vert_{L_{t}^{\infty}H_x^l}.
\end{equation}

We now proceed to prove $(3.4)$ and $(3.5)$.
We decompose the non-linearity $uv$ as
\begin{equation*}
uv = \sum_{\lambda_0 \in 2^{\mathbb{N}}}P_{\lambda_0}(uv)
\end{equation*}
Further, we distinguish the high-low, low-high and the balanced interactions as follows:
\begin{equation*}
P_{\lambda_0}(uv) = \sum_{\lambda_1 \in 2^{\mathbb{N}}}P_{\lambda_0}(u_{\lambda_1}v) = \sum_{\lambda_1 \ll \lambda_0 \in 2^{\mathbb{N}}}P_{\lambda_0}(u_{\lambda_1}v) + \sum_{\lambda_0 \ll \lambda_1 \in 2^{\mathbb{N}}}P_{\lambda_0}(u_{\lambda_1}v) + \sum_{\lambda_0 \sim \lambda_1 \in 2^{\mathbb{N}}}P_{\lambda_0}(u_{\lambda_1}v).
\end{equation*}

	We state that high (low) modulation for the spatially localised Schr{\"o}dinger solution $u_{|\xi|}$ means $\langle \tau+|\xi|^2\rangle \gtrsim (\ll) |\xi|^2$ while for the wave $v_{|\xi|}$, it means $\langle \tau - |\xi|\rangle \gtrsim (\ll) |\xi|^2$. We abbreviate high modulation by $H$ and low modulation by $L$.
The subscripts $1$ and $2$ are appended with $\tau$  and $\xi$ to distinguish the temporal frequencies and the spatial frequencies of the Schr{\"o}dinger and wave solutions.
The output temporal frequencies are given by $\tau_0 = \tau_1 +\tau_2$. For the Schr{\"o}dinger solution, low modulation occurs when the temporal frequencies are of size $\sim |\xi|^2 (\tau=-|\xi|^2)$. 
For the wave solution, temporal frequencies of size $\sim |\xi| (\tau = |\xi|)$ lead to a free wave solution. Other than that, the sizes of the modulation and the temporal frequencies go hand in hand for a frequency localised wave solution.\\
All the possible interactions are treated individually.  In cases with low output modulation, the required bilinear estimates are reduced to trilinear estimates by duality using the property (vii) from section 2.4. To distinguish the dual term, we use $-s$ and $-l$ instead of $s$ and $l$, respectively, in the subscripts of the norms for the Schr\"odinger and wave components.
Each subcase can be summed up depending on the size of the interacting frequencies. The constraints required to obtain the estimate for the full norms are listed at the end of each subcase, and we do not exclude points with logarithmic losses since our result does not cover the boundaries.
\begin{remark}
	In the following, we treat the case of high spatial (and temporal) frequencies ($|\xi_i|,|\tau_i| \gg 1, i=0,1,2$). In most of the cases, the estimates for the low frequencies (all the frequencies are of size $\sim 1$ or the size of the lowest frequency in the interaction is $\sim 1$) follow using the same arguments without having to decompose the space-time Fourier supports of the interacting solutions. Hence, we do not mention it explicitly for each case. However, in some cases the arguments need to be modified for the low frequency cases. We shall do it wherever required.
\end{remark}

	\textbf{Case I. Low to high interaction (\boldmath $\lambda_1 \ll \lambda_0$)}\\
We decompose $u_{\lambda_1}$ and $v_{\lambda_0}$ as follows:
\begin{equation*}
u_{\lambda_1} = C_{\ll \lambda_1^2}u_{\lambda_1} + C_{\gtrsim \lambda_1^2}u_{\lambda_1} , \quad v_{\lambda_0} = Q_{\ll \lambda_0^2}v_{\lambda_0} + Q_{\gtrsim \lambda_0^2}v_{\lambda_0}.
\end{equation*}
The following interactions can be distinguished on the basis of the size of the modulation:\\
	{1. \boldmath$H \times H \rightarrow H$}\\
We require to prove $(3.5)$. Using the size of the modulation, the product estimate and Bernstein's inequality respectively, we have
\begin{align*}
&\lambda_0^{s'-2a+b}\Vert \langle\tau + |\xi|^2\rangle^{\theta' -1}\mathcal{F}((\lambda_0 + |\partial_t|)^a C_{\gtrsim \lambda_0^2} (C_{\gtrsim \lambda_1^2}u_{\lambda_1}Q_{\gtrsim \lambda_0^2}v_{\lambda_0})) \Vert_{L_{\tau,\xi}^2}\\
&\lesssim \lambda_0^{s'-3a+b+2\theta'-2}\lambda_1^{\frac{d}{2}}\Vert (\lambda_0 +|\partial_t|)^a C_{\gtrsim \lambda_1^2}u_{\lambda_1}\Vert_{L_t^{\infty}L_x^2} \Vert (\lambda_0 +|\partial_t|)^a Q_{\gtrsim \lambda_0^2}v_{\lambda_0} \Vert_{L_{t,x}^2}\\
&\lesssim \lambda_0^{s'-3a+b+2\theta'-\beta-1}\lambda_1^{\frac{d}{2}-s+2a}\sup_{|\tau_1|}\Big(\frac{\lambda_0+|\tau_1|}{\lambda_1 +|\tau_1|}\Big)^a \sup_{|\tau_2|\gtrsim \lambda_0^2}\frac{(\lambda_0+|\tau_2|)^a}{\langle \tau_2-\lambda_0\rangle}\\
&\quad \times\lambda_1^{s'-2a}\Vert(\lambda_1 +|\partial_t|)^a C_{\gtrsim \lambda_1^2}u_{\lambda_1}\Vert_{L_{t}^{\infty}L_x^2}~ \lambda_0^{\beta-1}\Vert \langle \tau-|\xi|\rangle \mathcal{F}(Q_{\gtrsim \lambda_0^2}v_{\lambda_0})\Vert_{L_{\tau,\xi}^2}\\
&\lesssim \lambda_0^{s'+b+2\theta'-\beta-3}\lambda_1^{\frac{d}{2}-s+a} \Vert u_{\lambda_1}\Vert_{S_{\lambda_1}^{s,l,\theta}} \Vert v_{\lambda_0}\Vert_{W_{\lambda_0}^{l,s,\theta}}.
\end{align*}
The last inequality above follows from $(2.5)$. Since $s'+2\theta' = s+2\theta$, we can sum up the subcase in $\lambda_1\ll \lambda_0$ to obtain $(3.5)$ provided $s-\beta\leqslant 3-2\theta-b$ and $\beta\geqslant -3+\frac{d}{2}+2\theta+a+b$.\\

{2. \boldmath $H\times H \rightarrow L$}\\
%If $\tau_S$ and $\tau_W$ are such that $|\tau_S|, |\tau_W| \lesssim \lambda_0^2$
We consider two cases for the temporal frequencies:\\
\textbf{a.}  $|\tau_1|\lesssim \lambda_0^2$, $|\tau_2|\sim \lambda_0^2$\\
We prove $(3.4)$ using duality. We consider the expression $I$ defined in $(2.10)$ and use Cauchy-Schwarz inequality as follows:
\begin{align}
&|I(\mathcal{F}(Q_{\sim \lambda_0^2}v_{\lambda_0}),  \mathcal{F}(C_{\ll\lambda_0^2}w_{\lambda_0}),\mathcal{F}(C_{\lambda_1^2\lesssim  \cdot \lesssim \lambda_0^2}u_{\lambda_1}))|\lesssim \Vert Q_{\sim \lambda_0^2}v_{\lambda_0}\Vert_{L_{t,x}^2} \Vert C_{\lambda_1^2\lesssim \cdot \lesssim \lambda_0^2}u_{\lambda_1} \overline{C_{\ll\lambda_0^2}w_{\lambda_0}}\Vert_{L_{t,x}^2}.
\end{align}
On decomposing the space-time Fourier supports of $C_{\lambda_1^2\lesssim \cdot \lesssim \lambda_0^2}u_{\lambda_1}$, $\overline{C_{\ll\lambda_0^2}w_{\lambda_0}}$ into pieces $L_1$ and $L_0$, respectively and applying the bilinear estimate (2.12), we get
\begin{align*}
(3.7) &\lesssim \Vert Q_{\sim \lambda_0^2}v_{\lambda_0}\Vert_{L_{t,x}^2} \sum_{\substack{L_0\ll \lambda_0^2\\L_1\lesssim \lambda_0^2}}(L_0L_1)^{\frac{1}{2}}\frac{\lambda_1^{\frac{d-1}{2}}}{\lambda_0^{\frac{1}{2}}}\Vert C_{L_1}u_{\lambda_1}\Vert_{L_{t,x}^2}\Vert C_{L_0}w_{\lambda_0}\Vert_{L_{t,x}^2}\\
&\lesssim \lambda_0^{-\beta-\frac{5}{2}+s+2\theta}\lambda_1^{\frac{d-1}{2}-s+a} \lambda_1^{s'-2a} \Vert \langle \tau+|\xi|^2\rangle^{\theta'}\mathcal{F}((\lambda_1+|\partial_t|)^a C_{\lambda_1^2\lesssim \cdot \lesssim \lambda_0^2}u_{\lambda_1})\Vert_{L_{\tau,\xi}^2} \\
&\quad  \times  \lambda_0^{\beta-1}\Vert \langle \tau - |\xi|\rangle \mathcal{F}(Q_{\sim \lambda_0^2}v_{\lambda_0})\Vert_{L_{\tau,\xi}^2} \lambda_0^{-s} \Vert \langle \tau+|\xi|^2\rangle^{1-\theta}\mathcal{F}(C_{\ll\lambda_0^2}w_{\lambda_0})\Vert_{L_{\tau,\xi}^2}\\
&\lesssim \lambda_0^{-\beta-\frac{5}{2}+s+2\theta}\lambda_1^{\frac{d-1}{2}-s+a} \Vert u_{\lambda_1}\Vert_{S_{\lambda_1}^{s,l,\theta}} \Vert v_{\lambda_0}\Vert_{W_{\lambda_0}^{l,s,\theta}}\Vert w_{\lambda_0}\Vert_{S_{\lambda_0}^{-s,l,1-\theta}}.
\end{align*}
We require $s-\beta\leqslant \frac{5}{2}-2\theta$ and $\beta\geqslant -3+2\theta+\frac{d}{2}+a$ to sum the above up and obtain the estimate for the full norms.\\
In case $\lambda_1\sim \lambda_0\sim 1$, we have
\begin{align*}
|I(\mathcal{F}(Q_{\sim \lambda_0^2}v_{\lambda_0}),  \mathcal{F}(C_{\ll\lambda_0^2}w_{\lambda_0}),\mathcal{F}(C_{\lambda_1^2\lesssim  \cdot \lesssim \lambda_0^2}u_{\lambda_1}))|& \lesssim \Vert Q_{\sim \lambda_0^2}v_{\lambda_0}\Vert_{L_{t,x}^2} \Vert C_{\lambda_1^2\lesssim \cdot \lesssim \lambda_0^2}u_{\lambda_1} \Vert_{L_{t}^{\infty}L_x^2} 	\Vert C_{\ll\lambda_0^2}w_{\lambda_0}\Vert_{L_{t,x}^2}\\
&\lesssim \Vert u_{\lambda_1}\Vert_{S_{\lambda_1}^{s,l,\theta}} \Vert v_{\lambda_0}\Vert_{W_{\lambda_0}^{l,s,\theta}}\Vert w_{\lambda_0}\Vert_{S_{\lambda_0}^{-s,l,1-\theta}}.
\end{align*}
\textbf{b.} $|\tau_1|,|\tau_2|\gg \lambda_0^2$\\
Using H{\"o}lder's and Bernstein's inequality, we have
\begin{align*}
&	|I(\mathcal{F}(Q_{\gg \lambda_0^2}v_{\lambda_0}), \mathcal{F}(C_{\ll \lambda_0^2}w_{\lambda_0}),\mathcal{F}(C_{\gg \lambda_0^2}u_{\lambda_1})) |\\
&\lesssim \lambda_1^{\frac{d}{2}}\Vert C_{\gg \lambda_0^2}u_{\lambda_1} \Vert_{L_{t,x}^2} \Vert Q_{\gg \lambda_0^2}v_{\lambda_0}\Vert_{L_{t,x}^2} \Vert C_{\ll \lambda_0^2}w_{\lambda_0}\Vert_{L_t^{\infty} L_x^2}\\
&\lesssim\lambda_1^{\frac{d}{2}-s'+2a}\lambda_0^{-2a-\beta-2+2(\theta-\theta')+s}\lambda_1^{s'-2a}\Vert \langle \tau+|\xi|^2\rangle^{\theta'}\mathcal{F}((\lambda_1+|\partial_t|)^a C_{\gg \lambda_0^2}u_{\lambda_1})\Vert_{L_{\tau,\xi}^2}\\
&\quad \times\lambda_0^{\beta-1}\Vert \langle \tau-|\xi|\rangle \mathcal{F}(Q_{\gg \lambda_0^2}v_{\lambda_0})\Vert_{L_{\tau,\xi}^2}
 \lambda_0^{-s}\Vert C_{\ll \lambda_0^2}w_{\lambda_0}\Vert_{L_t^pL_x^2}\\
&\lesssim\lambda_1^{\frac{d}{2}-s'+2a}\lambda_0^{-2a-\beta-2+2(\theta-\theta')+s}\Vert u_{\lambda_1}\Vert_{S_{\lambda_1}^{s,l,\theta}}\Vert v_{\lambda_0}\Vert_{W^{l,s,\theta}_{\lambda_0}}\Vert w_{\lambda_0}\Vert_{S_{\lambda_0}^{-s,l,1-\theta}}.
\end{align*}
Note that the penultimate inequality follows from  Bernstein's inequality for $\frac{1}{p} = \theta-\frac{1}{2}$ while the ultimate comes from the embedding relation $(2.7)$. For summability, we require $s-\beta\leqslant 2+2(\theta'-\theta)+2a$ and $\beta\geqslant -2+\frac{d}{2}$.\\

	{3. \boldmath $H \times L \rightarrow H$}\\
Using the size of the modulation, the product estimate and Bernstein's inequality, we obtain
\begin{align*}
&\lambda_0^{s'-2a+b}\Vert \langle\tau + |\xi|^2\rangle^{\theta' -1}\mathcal{F}((\lambda_0 + |\partial_t|)^a C_{\gtrsim \lambda_0^2}(C_{\gtrsim \lambda_1^2}u_{\lambda_1}Q_{\ll \lambda_0^2}v_{\lambda_0})) \Vert_{L_{\tau,\xi}^2}\\
&\lesssim  \lambda_0^{s'-3a+b+2\theta' -2}\lambda_1^{\frac{d}{2}}\Vert (\lambda_0 + |\partial_t|)^a C_{\gtrsim \lambda_1^2}u_{\lambda_1}\Vert_{L_{t,x}^2}\Vert(\lambda_0 + |\partial_t|)^a Q_{\ll \lambda_0^2}v_{\lambda_0} \Vert_{L_t^{\infty}L_x^2}\\ 
&\lesssim  \lambda_0^{s'-2a+b+2\theta' -2-l}\lambda_1^{\frac{d}{2}-s+2a-2\theta}\sup_{|\tau_1|}\Big(\frac{\lambda_0 +|\tau_1|}{\lambda_1+|\tau_1|}\Big)^a \lambda_1^{s'-2a}\Vert \langle \tau+|\xi|^2\rangle^{\theta'} \mathcal{F}((\lambda_1 + |\partial_t|)^a C_{\gtrsim \lambda_1^2}u_{\lambda_1})\Vert_{L_{\tau,\xi}^2}\\
&\quad \times \lambda_0^{l-a}\Vert(\lambda_0 + |\partial_t|)^a Q_{\ll \lambda_0^2}v_{\lambda_0} \Vert_{L_t^{\infty}L_x^2}\\ 
&\lesssim \lambda_0^{s'-a+b+2\theta'-2-l}\lambda_1^{\frac{d}{2}-s+a-2\theta} \Vert u_{\lambda_1}\Vert_{S_{\lambda_1}^{s,l,\theta}} \Vert v_{\lambda_0}\Vert_{W_{\lambda_0}^{l,s,\theta}}.
\end{align*}
Provided $s-l \leqslant2-2\theta+a-b$ and $l \geqslant-2+\frac{d}{2}+b$, we can sum the above to obtain $(3.5)$.\\

{4. \boldmath $H \times L \rightarrow L$}\\
From the relation $|\tau_0 +|\xi_0|^2| = |\tau_1 +\underbrace{\tau_2-|\xi_2|}_{\ll \lambda_0^2}+|\xi_2|+|\xi_0|^2| \ll \lambda_0^2$, we conclude that $|\tau_1| \sim \lambda_0^2$. Using H{\"o}lder's inequality, we have
\begin{align}
|I(\mathcal{F}(Q_{\ll \lambda_0^2}v_{\lambda_0}),\mathcal{F}(C_{\ll \lambda_0^2}w_{\lambda_0}) ,\mathcal{F}(C_{\sim \lambda_0^2}u_{\lambda_1}))|&\lesssim \Vert C_{\sim \lambda_0^2}u_{\lambda_1}\Vert_{L_{t,x}^2} \Vert P_{\lambda_1}(\overline{Q_{\ll \lambda_0^2}v_{\lambda_0}}C_{\ll \lambda_0^2}w_{\lambda_0})\Vert_{L_{t,x}^2}.
\end{align}
The spatial frequency support of $C_{\sim \lambda_0^2}u_{\lambda_1}$ is localised to frequencies of size $\sim \lambda_1$. Using orthogonality, we can reduce the estimate to the case when the spatial supports of $Q_{\ll \lambda_0^2}v_{\lambda_0}$ and $C_{\ll \lambda_0^2}w_{\lambda_0}$ are also localised to frequencies of size $\sim \lambda_1$. Noting this, decomposing the space-time Fourier supports of  $\overline{Q_{\ll \lambda_0^2}v_{\lambda_0}}$ and $C_{\ll \lambda_0^2}w_{\lambda_0}$ into pieces $L_2$ and $L_0$ respectively and applying the bilinear estimate $(2.13)$, we have
\begin{align*}
(3.8)&\lesssim\Vert C_{\sim \lambda_0^2}u_{\lambda_1}\Vert_{L_{t,x}^2} \sum_{L_0,L_2\ll \lambda_0^2}(L_0L_2)^{\frac{1}{2}}\frac{\lambda_1^{\frac{d-1}{2}}}{\lambda_0^{\frac{1}{2}}}\Vert Q_{L_2}v_{\lambda_0}\Vert_{L_{t,x}^2}\Vert C_{L_0}w_{\lambda_0}\Vert_{L_{t,x}^2}\\
& \lesssim \lambda_0^{s-l-\frac{3}{2}+2(\theta-\theta')-2a} \lambda_1^{\frac{d-1}{2}-s'+2a} \lambda_1^{s'-2a}\Vert \langle \tau+|\xi|^2\rangle^{\theta}\mathcal{F}((\lambda_1+|\partial_t|)^a C_{\sim \lambda_0^2}u_{\lambda_1})\Vert_{L_{\tau,\xi}^2}\\
&\quad \times \lambda_0^{l-a}\Vert \langle\tau-|\xi|\rangle^{\theta}\mathcal{F}((\lambda_0+|\partial_t|)^a Q_{\ll \lambda_0^2}v_{\lambda_0})\Vert_{L_{\tau,\xi}^2} \lambda_0^{-s}\Vert \langle \tau+|\xi|^2\rangle^{1-\theta}\mathcal{F}(C_{\ll \lambda_0^2}w_{\lambda_0})\Vert_{L_{\tau,\xi}^2}\\
&\lesssim \lambda_0^{s-l-\frac{3}{2}+2(\theta-\theta')-2a} \lambda_1^{\frac{d-1}{2}-s'+2a}\Vert u_{\lambda_1}\Vert_{S_{\lambda_1}^{s,l,\theta}} \Vert v_{\lambda_0}\Vert_{W_{\lambda_0}^{l,s,\theta}}\Vert w_{\lambda_0}\Vert_{S_{\lambda_0}^{-s,l,1-\theta}}.
\end{align*}
We require $s-l \leqslant \frac{3}{2}-2(\theta-\theta')+2a$ and $l\geqslant -2+\frac{d}{2}$ for the summability of the above estimate.\\

	{5. \boldmath $L \times H \rightarrow H$}\\
Using the size of the modulation, the product estimate and Bernstein's inequality, we have
\begin{align*}
&\lambda_0^{s'-2a+b}\Vert \langle \tau+|\xi|^2\rangle^{\theta'-1}\mathcal{F}((\lambda_0+|\partial_t|)^a C_{\ll \lambda_1^2}u_{\lambda_1}Q_{\gtrsim \lambda_0^2}v_{\lambda_0})\Vert_{L_{\tau,\xi}^2}\\
&\lesssim \lambda_0^{s'-3a+b+2\theta'-2}\lambda_1^{\frac{d}{2}}\Vert (\lambda_0 +|\partial_t|)^a C_{\ll \lambda_1^2}u_{\lambda_1}\Vert_{L_{t}^{\infty}L_x^2} \Vert (\lambda_0+|\partial_t|)^a Q_{\gtrsim \lambda_0^2}v_{\lambda_0}\Vert_{L_{t,x}^2}\\
&\lesssim \lambda_0^{s'-3a+b+2\theta'-1-\beta}\lambda_1^{\frac{d}{2}-s}(\lambda_0+\lambda_1^2)^a \sup_{|\tau_2|\gtrsim \lambda_0^2}\frac{(\lambda_0+|\tau_2|)^a}{\langle \tau_2-\lambda_0\rangle}\lambda_1^{s}\Vert C_{\ll \lambda_1^2}u_{\lambda_1}\Vert_{L_{t}^{\infty}L_x^2}~\lambda_0^{\beta-1}\Vert \langle \tau-|\xi|\rangle \mathcal{F}(Q_{\gtrsim \lambda_0^2}v_{\lambda_0})\Vert_{L_{\tau,\xi}^2}\\
&\lesssim \lambda_0^{s'-a+b+2\theta'-3-\beta}\lambda_1^{\frac{d}{2}-s}(\lambda_0+\lambda_1^2)^a \Vert u_{\lambda_1}\Vert_{S_{\lambda_1}^{s,l,\theta}} \Vert v_{\lambda_0}\Vert_{W_{\lambda_0}^{l,s,\theta}}\\
&\lesssim \begin{cases}
\lambda_0^{s'+b+2\theta'-3-\beta}\lambda_1^{\frac{d}{2}-s} \Vert u_{\lambda_1}\Vert_{S_{\lambda_1}^{s,l,\theta}} \Vert v_{\lambda_0}\Vert_{W_{\lambda_0}^{l,s,\theta}}, ~~\lambda_0+\lambda_1^2\sim \lambda_0\\
\lambda_0^{s'-a+b+2\theta'-3-\beta}\lambda_1^{\frac{d}{2}-s+2a} \Vert u_{\lambda_1}\Vert_{S_{\lambda_1}^{s,l,\theta}} \Vert v_{\lambda_0}\Vert_{W_{\lambda_0}^{l,s,\theta}}, ~~\lambda_0+\lambda_1^2\sim \lambda_1^2.
\end{cases}
\end{align*}
Both the cases above can be summed up provided $s-\beta\leqslant 3-2\theta-b$ and $\beta\geqslant -3+\frac{d}{2}+2\theta+a+b$.\\

{6. \boldmath $L \times H \rightarrow L$}\\
If $|\tau_2|\gg \lambda_0^2$, the output will have a high modulation. Hence it suffices to consider $|\tau_2 |\sim \lambda_0^2$. Then, we have
\begin{align}
|I(\mathcal{F}(Q_{\sim \lambda_0^2}v_{\lambda_0}),\mathcal{F}(C_{\ll \lambda_0^2}w_{\lambda_0}),\mathcal{F}(C_{\ll \lambda_1^2}u_{\lambda_1}) )|&\lesssim \Vert Q_{\sim \lambda_0^2}v_{\lambda_0}\Vert_{L_{t,x}^2} \Vert C_{\ll \lambda_1^2}u_{\lambda_1}\overline{C_{\ll \lambda_0^2}w_{\lambda_0}}\Vert_{L_{t,x}^2}.
\end{align}
We decompose the space-time Fourier supports of $C_{\ll \lambda_1^2}u_{\lambda_1}$ and $\overline{C_{\ll \lambda_0^2}w_{\lambda_0}}$ into pieces $L_1$ and $L_0$ respectively, and apply the bilinear estimate $(2.12)$ to obtain
\begin{align*}
(3.9)&\lesssim  \Vert Q_{\sim \lambda_0^2}v_{\lambda_0}\Vert_{L_{t,x}^2} \sum_{L_0\ll \lambda_0^2,L_1\ll \lambda_1^2}(L_0L_1)^{\frac{1}{2}}\frac{\lambda_1^{\frac{d-1}{2}}}{\lambda_0^{\frac{1}{2}}}\Vert C_{L_1}u_{\lambda_1}\Vert_{L_{t,x}^2}\Vert C_{L_0}w_{\lambda_0}\Vert_{L_{t,x}^2}\\
&\lesssim \lambda_0^{-\beta-\frac{5}{2}+s+2\theta}\lambda_1^{\frac{d-1}{2}-s}~\lambda_1^s \Vert \langle \tau+|\xi|^2\rangle^{\theta}\mathcal{F}(C_{\ll \lambda_1^2}u_{\lambda_1})\Vert_{L_{\tau,\xi}^2} \lambda_0^{\beta-1}\Vert \langle \tau-|\xi|\rangle \mathcal{F}(Q_{\sim \lambda_0^2}v_{\lambda_0})\Vert_{L_{\tau,\xi}^2}\\
&\quad \times \lambda_0^{-s}\Vert \langle \tau+|\xi|^2\rangle^{1-\theta} \mathcal{F}(C_{\ll \lambda_0^2}w_{\lambda_0})\Vert_{L_{\tau,\xi}^2}\\
&\lesssim \lambda_0^{-\beta-\frac{5}{2}+s+2\theta}\lambda_1^{\frac{d-1}{2}-s}\Vert u_{\lambda_1}\Vert_{S_{\lambda_1}^{s,l,\theta}} \Vert v_{\lambda_0}\Vert_{W_{\lambda_0}^{l,s,\theta}}\Vert w_{\lambda_0}\Vert_{S_{\lambda_0}^{-s,l,1-\theta}}.
\end{align*}
We require $s-\beta\leqslant \frac{5}{2}-2\theta$ and $\beta\geqslant-3+\frac{d}{2}+2\theta$ to sum the above estimate and obtain $(3.4)$.\\
For $\lambda_1\sim \lambda_0 \sim 1, d=1$, we have
\begin{align*}
|I(\mathcal{F}(Q_{\sim \lambda_0^2}v_{\lambda_0}),\mathcal{F}(C_{\ll \lambda_0^2}w_{\lambda_0}),\mathcal{F}(C_{\ll \lambda_1^2}u_{\lambda_1}) )|&\lesssim \Vert Q_{\sim \lambda_0^2}v_{\lambda_0}\Vert_{L_{t,x}^2} \Vert C_{\ll \lambda_1^2}u_{\lambda_1}\Vert_{L_t^{\infty}L_x^2} \Vert C_{\ll \lambda_0^2}w_{\lambda_0}\Vert_{L_{t,x}^2}\\
&\lesssim \Vert v_{\lambda_0}\Vert_{W^{l,s,\theta}_{\lambda_0}} \Vert u_{\lambda_1}\Vert_{S^{s,l,\theta}_{\lambda_1}} \Vert w_{\lambda_0}\Vert_{S^{-s,l,1-\theta}_{\lambda_0}}.
\end{align*}

	{7. \boldmath $L \times L \rightarrow H$}\\
{\boldmath	$d=3:$} Using the size of the modulation, the product estimate and the endpoint Strichartz space $L_t^2L_x^6$, we have
\begin{align*}
&\lambda_0^{s'-2a+b}\Vert \langle\tau + |\xi|^2\rangle^{\theta' -1}\mathcal{F}((\lambda_0 + |\partial_t|)^a C_{\gtrsim \lambda_0^2}(C_{\ll \lambda_1^2}u_{\lambda_1}Q_{\ll \lambda_0^2}v_{\lambda_0})) \Vert_{L_{\tau,\xi}^2}\\
&\lesssim \lambda_0^{s'-3a+b+2\theta'-2}\lambda_1^{\frac{1}{2}}\Vert (\lambda_0+|\partial_t|)^a C_{\ll \lambda_1^2}u_{\lambda_1}\Vert_{L_t^2L_x^6}\Vert (\lambda_0+|\partial_t|)^a Q_{\ll \lambda_0^2}v_{\lambda_0}\Vert_{L_t^{\infty}L_x^2}\\
&\lesssim \begin{cases}
\lambda_0^{s'-2a+b+2\theta'-2-l}\lambda_1^{\frac{1}{2}-s+2a}\Vert u_{\lambda_1}\Vert_{S_{\lambda_1}^{s,l,\theta}} \Vert v_{\lambda_0}\Vert_{W_{\lambda_0}^{l,s,\theta}}, ~~~\lambda_0 +\lambda_1^2 \sim \lambda_1^2\\
\lambda_0^{s'-a+b+2\theta'-2-l}\lambda_1^{\frac{1}{2}-s} \Vert u_{\lambda_1}\Vert_{S_{\lambda_1}^{s,l,\theta}}\Vert v_{\lambda_0}\Vert_{W_{\lambda_0}^{l,s,\theta}}, ~~~\lambda_0 +\lambda_1^2 \sim \lambda_0.
\end{cases}
\end{align*}
The constraints $s-l\leqslant 2-2\theta+a-b$ and $l \geqslant -\frac{3}{2}+2\theta+b$ are required to sum the above cases.\\

\noindent	{\boldmath	$d=2:$} We consider two subcases for the size of the wave modulation:\\
\textbf{a.}	{\boldmath $\langle \tau-|\xi|\rangle\ll \lambda_0$}:
For $\lambda_0 \ll |\tau_2|\ll \lambda_0^2$, the wave has a modulation of size $\gg \lambda_0$, so it suffices to consider $|\tau_2|\sim \lambda_0$. Using the size of the output modulation and the bilinear estimate $(2.13)$, we have
\begin{align*}
&\lambda_0^{s'-2a+b}\Vert \langle\tau + |\xi|^2\rangle^{\theta' -1}\mathcal{F}((\lambda_0 + |\partial_t|)^a C_{\gtrsim \lambda_0^2} (C_{\ll \lambda_1^2}u_{\lambda_1}P_{\sim \lambda_0}^{(t)}Q_{\ll \lambda_0^2}v_{\lambda_0})) \Vert_{L_{\tau,\xi}^2}\\
&\lesssim \lambda_0^{s'-2a+b+2\theta'-2}(\lambda_0+\lambda_1^2)^a\sum_{\substack{L_1\ll \lambda_1^2\\L_2\ll \lambda_0}}(L_1L_2)^{\frac{1}{2}}\frac{\lambda_1^{\frac{1}{2}}}{\lambda_1^{\frac{1}{2}}}\Vert C_{L_1}u_{\lambda_1}\Vert_{L_{t,x}^2}\Vert Q_{L_2}P_{\sim \lambda_0}^{(t)}v_{\lambda_0}\Vert_{L_{t,x}^2}\\
&\lesssim	\begin{cases}
\lambda_0^{s'-a+b+2\theta'-2-l}\lambda_1^{-s}  \Vert u_{\lambda_1}\Vert_{S^{s,l,\theta}_{\lambda_1}}\Vert v_{\lambda_0}\Vert_{W^{l,s,\theta}_{\lambda_0}}, ~~~\lambda_0+\lambda_1^2 \sim \lambda_0\\
\lambda_0^{s'-2a+b+2\theta'-2-l}\lambda_1^{-s+2a} \Vert u_{\lambda_1}\Vert_{S^{s,l,\theta}_{\lambda_1}}\Vert v_{\lambda_0}\Vert_{W^{l,s,\theta}_{\lambda_0}}, ~~~\lambda_0+\lambda_1^2 \sim \lambda_1^2.
\end{cases}
\end{align*}
In the case $1\sim \lambda_1 \ll \lambda_0$, we have
\begin{align*}
&\lambda_0^{s'-2a+b}\Vert \langle\tau + |\xi|^2\rangle^{\theta' -1}\mathcal{F}((\lambda_0 + |\partial_t|)^a C_{\gtrsim \lambda_0^2} (C_{\ll \lambda_1^2}u_{\lambda_1}P_{\sim \lambda_0}^{(t)}Q_{\ll \lambda_0^2}v_{\lambda_0})) \Vert_{L_{\tau,\xi}^2}
\lesssim \lambda_0^{s'-a+b+2\theta'-2-l} \Vert u_{\lambda_1}\Vert_{S^{s,l,\theta}_{\lambda_1}}\Vert v_{\lambda_0}\Vert_{W^{l,s,\theta}_{\lambda_0}}.
\end{align*}

	\textbf{b.}	{\boldmath $\lambda_0 \lesssim \langle \tau-|\xi| \rangle\ll \lambda_0^2$}: We use the size of the modulation, the product estimate and Bernstein's inequality to obtain
\begin{align*}
&\lambda_0^{s'-2a+b}\Vert \langle\tau + |\xi|^2\rangle^{\theta' -1}\mathcal{F}((\lambda_0 + |\partial_t|)^a C_{\gtrsim \lambda_0^2}(C_{\ll \lambda_1^2}u_{\lambda_1}Q_{\lambda_0\lesssim \cdot \ll \lambda_0^2}v_{\lambda_0})) \Vert_{L_{\tau,\xi}^2}\\
&\lesssim \lambda_0^{s'-3a+b+2\theta'-2}\lambda_1 \Vert (\lambda_0+|\partial_t|)^a C_{\ll \lambda_1^2}u_{\lambda_1}\Vert_{L_{t}^{\infty}L_x^2} \Vert (\lambda_0+|\partial_t|)^a  Q_{\lambda_0\lesssim \cdot \ll \lambda_0^2}v_{\lambda_0} \Vert_{L_{t,x}^2}\\
&\lesssim \lambda_0^{s'-2a+b+2\theta'-2-l-\theta}\lambda_1^{1-s}(\lambda_0+\lambda_1^2)^a\lambda_1^{s}\Vert C_{\ll \lambda_1^2}u_{\lambda_1}\Vert_{L_t^{\infty}L_x^2} \lambda_0^{l-a}\Vert \langle \tau-|\xi|\rangle^{\theta}\mathcal{F}((\lambda_0+|\partial_t|)^a Q_{\lambda_0\lesssim \cdot \ll \lambda_0^2}v_{\lambda_0})\Vert_{L_{\tau,\xi}^2}\\
&\lesssim \begin{cases}
\lambda_0^{s'-a+b+2\theta'-2-l-\theta}\lambda_1^{1-s}\Vert u_{\lambda_1}\Vert_{S^{s,l,\theta}_{\lambda_1}}\Vert v_{\lambda_0}\Vert_{W^{l,s,\theta}_{\lambda_0}}, ~~~\lambda_0+\lambda_1^2 \sim \lambda_0\\
\lambda_0^{s'-2a+b+2\theta'-2-l-\theta}\lambda_1^{1-s+2a}\Vert u_{\lambda_1}\Vert_{S^{s,l,\theta}_{\lambda_1}}\Vert v_{\lambda_0}\Vert_{W^{l,s,\theta}_{\lambda_0}}, ~~~\lambda_0+\lambda_1^2 \sim \lambda_1^2.
\end{cases}
\end{align*}
Cases a and b can be summed up provided $s-l \leqslant 2-\theta+a-b$ and $l\geqslant -1+\theta+b$.\\

\noindent			{\boldmath$d=1:$} As for $d=3$, we have
\begin{align*}
&\lambda_0^{s'-2a+b}\Vert \langle\tau + |\xi|^2\rangle^{\theta' -1}\mathcal{F}((\lambda_0 + |\partial_t|)^a C_{\gtrsim \lambda_0^2}(C_{\ll \lambda_1^2}u_{\lambda_1}Q_{\ll \lambda_0^2}v_{\lambda_0}) )\Vert_{L_{\tau,\xi}^2}\\
&\lesssim \lambda_0^{s'-3a+b+2\theta'-2}\lambda_1^{\frac{1}{2}}\Vert (\lambda_0+|\partial_t|)^a C_{\ll \lambda_1^2}u_{\lambda_1}\Vert_{L_t^{\infty}L_x^2}\Vert (\lambda_0+|\partial_t|)^a Q_{\ll \lambda_0^2}v_{\lambda_0}\Vert_{L_{t,x}^2}\\
&\lesssim \lambda_0^{s'-2a+b+2\theta'-2-l}\lambda_1^{\frac{1}{2}-s} (\lambda_0+\lambda_1^2)^a \lambda_1^s \Vert C_{\ll \lambda_1^2}u_{\lambda_1}\Vert_{L_{t}^{\infty}L_x^2} \lambda_0^{l-a}\Vert \langle \tau-|\xi|\rangle^{\theta}\mathcal{F}((\lambda_0+|\partial_t|)^a Q_{\ll \lambda_0^2}v_{\lambda_0})\Vert_{L_{\tau,\xi}^2}\\
&\lesssim \begin{cases}
\lambda_0^{s'-2a+b+2\theta'-2-l}\lambda_1^{\frac{1}{2}-s+2a}\Vert u_{\lambda_1}\Vert_{S_{\lambda_1}^{s,l,\theta}} \Vert v_{\lambda_0}\Vert_{W_{\lambda_0}^{l,s,\theta}}, ~~~\lambda_0 +\lambda_1^2 \sim \lambda_1^2\\
\lambda_0^{s'-a+b+2\theta'-2-l}\lambda_1^{\frac{1}{2}-s} \Vert u_{\lambda_1}\Vert_{S_{\lambda_1}^{s,l,\theta}}\Vert v_{\lambda_0}\Vert_{W_{\lambda_0}^{l,s,\theta}}, ~~~\lambda_0 +\lambda_1^2 \sim \lambda_0.
\end{cases}
\end{align*}
Provided $s-l\leqslant 2-2\theta+a-b$ and $l \geqslant-\frac{3}{2}+2\theta+b$, we can sum the above estimates to obtain $(3.5)$.\\

From the conclusions made at the end of each subcase, we see that following conditions on the parameters are required to ensure the validity of the estimates $(3.4)$ and $(3.5)$ for $d\leqslant 3$:
\begin{equation}
\bullet~	l\geqslant  -\frac{3}{2}+2\theta+b  \quad \bullet ~
\beta \geqslant -\frac{3}{2}+2\theta+a+b \quad \bullet ~
s-l \leqslant 2-2\theta+a-b \quad \bullet ~
s-\beta \leqslant \min\Big\{3-2\theta-b,\frac{5}{2}-2\theta \Big\}.
\end{equation}\\

\textbf{Case II. High to high interaction  (\boldmath $\lambda_1 \sim \lambda_0$)}\\
For $\mu\lesssim \lambda_1$, we decompose $u_{\lambda_1}$ and $v_{\mu}$ as follows:
\begin{equation*}
u_{\lambda_1} = C_{\ll \lambda_1^2}u_{\lambda_1} + C_{\gtrsim \lambda_1^2}u_{\lambda_1}, \quad v_{\mu} = Q_{\ll \mu^2}v_{\mu} + Q_{\gtrsim \mu^2}v_{\mu}.
\end{equation*}
The following interactions can be distinguished on the basis of the size of the modulation:\\

{1. \boldmath $H\times H \rightarrow H$}\\
We consider two cases for the temporal frequencies:\\
\textbf{a.} $|\tau_1|\lesssim \lambda_0^2$ or $|\tau_2|\lesssim \lambda_0^2$\\
Since at least one of the temporal frequencies $\tau_S$ or $\tau_W$ has size $\lesssim \lambda_0^2$, we can apply Bernstein's inequality in the time variable. We use the size of the modulation, the product estimate and Bernstein's inequality in space and time variables to obtain
\begin{align*}
&\lambda_0^{s'-2a+b}\Vert \langle \tau+|\xi|^2\rangle^{\theta'-1}\mathcal{F}((\lambda_0+|\partial_t|)^a C_{\gtrsim \lambda_0^2}(C_{\gtrsim \lambda_1^2}u_{\lambda_1}Q_{\gtrsim \mu^2}v_{\mu}))\Vert_{L_{\tau,\xi}^2}\\
&\lesssim \lambda_0^{s'-3a+b+2\theta'-1}\mu^{\frac{d}{2}} \Vert (\lambda_1 +|\partial_t|)^a C_{\gtrsim \lambda_1^2}u_{\lambda_1}\Vert_{L_{t,x}^2} \Vert (\lambda_0 +|\partial_t|)^a Q_{\gtrsim \mu^2} v_{\mu}\Vert_{L_{t,x}^2}\\
&\lesssim \lambda_0^{-a+b-1}\mu^{\frac{d}{2}+1-\beta}\sup_{|\tau_2|\gtrsim \mu^2}\frac{(\lambda_0+|\tau_2|)^a}{\langle \tau_2-\mu\rangle}\lambda_1^{s'-2a}\Vert \langle \tau+|\xi|^2\rangle^{\theta'} \mathcal{F}((\lambda_1+|\partial_t|)^a C_{\gtrsim \lambda_1^2}u_{\lambda_1})\Vert_{L_{\tau,\xi}^2}\\
&\quad \times \mu^{\beta-1}\Vert \langle \tau-|\xi|\rangle \mathcal{F}(Q_{\gtrsim \mu^2}v_{\mu})\Vert_{L_{\tau,\xi}^2}\\
&\lesssim \begin{cases}
\lambda_0^{b-1}\mu^{\frac{d}{2}-1-\beta}  \Vert u_{\lambda_1}\Vert_{S_{\lambda_1}^{s,l,\theta}} \Vert v_{\mu}\Vert_{W_{\mu}^{l,s,\theta}}, ~~|\tau_2|\lesssim \lambda_0\\
\lambda_0^{-a+b-1}\mu^{\frac{d}{2}-\beta+2a-1}  \Vert u_{\lambda_1}\Vert_{S_{\lambda_1}^{s,l,\theta}} \Vert v_{\mu}\Vert_{W_{\mu}^{l,s,\theta}}, ~~~\lambda_0\lesssim |\tau_2|.
\end{cases}
\end{align*}
Since $b< 1$ and $-a+b <1$, we require $\beta\geqslant-2+\frac{d}{2}+a+b$ to sum the above estimates.\\

\noindent
\textbf{b.} $|\tau_1|, |\tau_2|\gg \lambda_0^2$\\
We employ the same steps as in the previous case but cannot use Bernstein's inequality wrt time. Instead, we make use of the high modulation of the wave.
\begin{align*}
&\lambda_0^{s'-2a+b}\Vert \langle \tau+|\xi|^2\rangle^{\theta'-1}\mathcal{F}((\lambda_0+|\partial_t|)^a C_{\gtrsim \lambda_0^2}(C_{\gg \lambda_0^2}u_{\lambda_1}Q_{\gg \lambda_0^2}v_{\mu}))\Vert_{L_{\tau,\xi}^2}\\
&\lesssim \lambda_0^{s'-3a+b+2\theta'-2}\mu^{\frac{d}{2}} \Vert (\lambda_1 +|\partial_t|)^a C_{\gg \lambda_1^2}u_{\lambda_1}\Vert_{L_{t}^{\infty}L_x^2} \Vert (\lambda_0 +|\partial_t|)^a Q_{\gg \lambda_0^2} v_{\mu}\Vert_{L_{t,x}^2}\\
&\lesssim \lambda_0^{-a+b+2\theta-2}\mu^{\frac{d}{2}-\beta+1}\sup_{|\tau_2|\gg \lambda_0^2}\frac{(\lambda_0+|\tau_2|)^a}{\langle \tau_2 - \lambda_0\rangle}\lambda_1^{s'-2a}\Vert (\lambda_1+|\partial_t|)^a C_{\gg \lambda_1^2}u_{\lambda_1}\Vert_{L_{t}^{\infty}L_x^2} \mu^{\beta-1}\Vert \langle \tau-|\xi|\rangle \mathcal{F}(Q_{\gg \lambda_0^2}v_{\mu})\Vert_{L_{\tau,\xi}^2}\\
&\lesssim \lambda_0^{a+b+2\theta-4}\mu^{\frac{d}{2}-\beta+1}
\Vert u_{\lambda_1}\Vert_{S_{\lambda_1}^{s,l,\theta}} \Vert v_{\mu}\Vert_{W_{\mu}^{l,s,\theta}}.
\end{align*}
We can sum the above up provided $\beta\geqslant-3+\frac{d}{2}+2\theta+a+b$.\\

	{2. \boldmath $H\times H \rightarrow L$}\\
We consider the expression $I$ and apply H{\"o}lder's and Bernstein's inequality to obtain
\begin{align*}
&|I( \mathcal{F}(Q_{\gtrsim \mu^2}v_{\mu}), \mathcal{F}(C_{\ll \lambda_0^2}w_{\lambda_0}),\mathcal{F}(C_{\gtrsim \lambda_1^2}u_{\lambda_1}))|\\
& \lesssim \mu^{\frac{d}{2}}\Vert C_{\gtrsim \lambda_1^2}u_{\lambda_1} \Vert_{L_{t,x}^2} \Vert Q_{\gtrsim \mu^2}v_{\mu} \Vert_{L_{t,x}^2} \Vert C_{\ll \lambda_0^2}w_{\lambda_0}\Vert_{L_{t}^{\infty}L_x^2}\\
&\lesssim \mu^{\frac{d}{2}-1-\beta}\lambda_1^{a-1} \lambda_1^{s'-2a}\Vert \langle \tau+|\xi|^2\rangle^{\theta'}\mathcal{F}((\lambda_1 +|\partial_t|)^a C_{\gtrsim \lambda_1^2}u_{\lambda_1})\Vert_{L_{\tau,\xi}^2} ~\mu^{\beta-1}\Vert\langle \tau -|\xi|\rangle \mathcal{F}(Q_{\gtrsim \mu^2}v_{\mu})\Vert_{L_{\tau,\xi}^2}\\
&\quad \times \lambda_0^{-s}\Vert C_{\ll \lambda_0^2}w_{\lambda_0}\Vert_{L_t^p L_x^2}\\
&\lesssim \lambda_1^{a-1}\mu^{\frac{d}{2}-1-\beta} \Vert u_{\lambda_1}\Vert_{S_{\lambda_1}^{s,l,\theta}} \Vert v_{\mu}\Vert_{W_{\mu}^{l,s,\theta}}\Vert w_{\lambda_0}\Vert_{S_{\lambda_0}^{-s,l,1-\theta}}.
\end{align*}
To obtain the last second inequality, we use Bernstein's inequality for $\frac{1}{p} = \theta-\frac{1}{2}$ and the last inequality follows from $(2.7)$.\\
We require $\beta\geqslant-2+\frac{d}{2}+a$ to sum this estimate and obtain $(3.4)$.\\

{3. \boldmath $H\times L \rightarrow H$}\\
Using the size of the modulation weight, the product estimate and Bernstein's inequality, we get
\begin{align*}
&\lambda_0^{s'-2a+b}\Vert \langle \tau +|\xi|^2\rangle^{\theta'-1}\mathcal{F}((\lambda_0+|\partial_t|)^a C_{\gtrsim \lambda_0^2}(C_{\gtrsim \lambda_1^2}u_{\lambda_1}Q_{\ll \mu^2}v_{\mu}))\Vert_{L_{\tau,\xi}^2}\\
&\lesssim \lambda_0^{s'-3a+b+2\theta'-2}\mu^{\frac{d}{2}}\Vert (\lambda_1+|\partial_t|)^a C_{\gtrsim \lambda_1^2}u_{\lambda_1} \Vert_{L_{t,x}^2} \Vert (\lambda_0+|\partial_t|)^a Q_{\ll \mu^2}v_{\mu}\Vert_{L_{t}^{\infty}L_x^2}\\
&\lesssim \lambda_0^{-a+b-2}\mu^{\frac{d}{2}-l+a}\sup_{|\tau_2|\ll \mu^2}\Big(\frac{\lambda_0 +|\tau_2|}{\mu+|\tau_2|}\Big)^a \lambda_1^{s'-2a}\Vert \langle \tau+|\xi|^2\rangle^{\theta'}\mathcal{F}((\lambda_1+|\partial_t|)^a C_{\gtrsim \lambda_1^2}u_{\lambda_1})\Vert_{L_{\tau,\xi}^2}\\
&\quad \times \mu^{l-a}\Vert (\mu +|\partial_t|)^a Q_{\ll \mu^2}v_{\mu}\Vert_{L_{t}^{\infty}L_x^2}\\
&\lesssim \begin{cases} 
\lambda_0^{b-2}\mu^{\frac{d}{2}-l} \Vert u_{\lambda_1}\Vert_{S_{\lambda_1}^{s,l,\theta}} \Vert v_{\mu}\Vert_{W_{\mu}^{l,s,\theta}}, ~~|\tau_2|\lesssim \lambda_0,\\
\lambda_0^{-a+b-2}\mu^{\frac{d}{2}-l+a} \Vert u_{\lambda_1}\Vert_{S_{\lambda_1}^{s,l,\theta}} \Vert v_{\mu}\Vert_{W_{\mu}^{l,s,\theta}}, ~~\lambda_0 \lesssim |\tau_2|
\end{cases} \lesssim \lambda_0^{b-2}\mu^{\frac{d}{2}-l} \Vert u_{\lambda_1}\Vert_{S_{\lambda_1}^{s,l,\theta}} \Vert v_{\mu}\Vert_{W_{\mu}^{l,s,\theta}}.
\end{align*}\\
The constraint $l \geqslant -2+\frac{d}{2}+b$ is required to sum the above up.\\

{4. \boldmath $L\times H \rightarrow H$}\\
Using the size of the modulation, the product estimate and Bernstein's inequality, we have
\begin{align*}
&\lambda_0^{s'-2a+b}\Vert \langle \tau+|\xi|^2\rangle^{\theta'-1}\mathcal{F}((\lambda_0+|\partial_t|)^a C_{\gtrsim \lambda_0^2}(C_{\ll \lambda_1^2}u_{\lambda_1}Q_{\gtrsim \mu^2}v_{\mu}))\Vert_{L_{\tau,\xi}^2}\\
&\lesssim \lambda_0^{s'-3a+b+2\theta'-2}\mu^{\frac{d}{2}}\Vert (\lambda_1+|\partial_t|)^a C_{\ll \lambda_1^2}u_{\lambda_1}\Vert_{L_{t}^{\infty}L_x^2} \Vert (\lambda_1+|\partial_t|)^a  Q_{\gtrsim \mu^2}v_{\mu}\Vert_{L_{t,x}^2}\\
&\lesssim \lambda_0^{-a+b+2\theta-2}\mu^{\frac{d}{2}-\beta+1}\sup_{|\tau_2|\gtrsim \mu^2}\frac{(\lambda_0+|\tau_2|)^a}{\langle \mu -\tau_2 \rangle}~\lambda_1^s \Vert C_{\ll \lambda_1^2}u_{\lambda_1}\Vert_{L_{t}^{\infty}L_x^2}~ \mu^{\beta-1}\Vert \langle \tau-|\xi|\rangle \mathcal{F}(Q_{\gtrsim \mu^2}v_{\mu})\Vert_{L_{\tau,\xi}^2}\\
&\lesssim \begin{cases}
\lambda_0^{b+2\theta-2}\mu^{\frac{d}{2}-\beta-1}\Vert u_{\lambda_1}\Vert_{S^{s,l,\theta}_{\lambda_1}} \Vert v_{\mu}\Vert_{W^{l,s,\theta}_{\mu}}, ~~~|\tau_2|\lesssim \lambda_0\\
\lambda_0^{-a+b+2\theta-2}\mu^{\frac{d}{2}-\beta-1+2a}\Vert u_{\lambda_1}\Vert_{S^{s,l,\theta}_{\lambda_1}} \Vert v_{\mu}\Vert_{W^{l,s,\theta}_{\mu}}, ~~~\lambda_0\lesssim |\tau_2|.
\end{cases} 
\end{align*}
With $b\leqslant 2-2\theta$ and $\beta\geqslant-3+\frac{d}{2}+2\theta+a+b$, we can sum this estimate to obtain $(3.5)$.\\

	{5. \boldmath $L\times H \rightarrow L$}\\
Since $u_{\lambda_1}$ has temporal frequencies of size $\sim \lambda_1^2$, we observe that the temporal frequencies $\tau_2$ are such that $\mu^2 \lesssim |\tau_2|\lesssim \lambda_0^2$.
\noindent{\boldmath $d=2,3:$} We have the standard decomposition:
\begin{align}\nonumber
|I(\mathcal{F}(Q_{\mu^2\lesssim \cdot \lesssim \lambda_0^2}v_{\mu}), \mathcal{F}(C_{\ll \lambda_0^2}w_{\lambda_0}),\mathcal{F}(C_{\ll \lambda_1^2}u_{\lambda_1}))|
&\lesssim \Big |I \Big( \sum_{\mu^2\lesssim L_2 \lesssim \lambda_0^2}\mathcal{F}(Q_{L_2} v_{\mu}), \sum_{L_0 \ll \lambda_0^2} \mathcal{F}(C_{L_0}w_{\lambda_0}),\sum_{L_1 \ll \lambda_1^2}\mathcal{F}(C_{L_1}u_{\lambda_1})\Big)\Big|\\
&\lesssim \sum_{\substack {L_1,L_0\ll \lambda_0^2,\\ \mu^2\lesssim L_2 \lesssim \lambda_0^2}}\Big|I(\mathcal{F}(Q_{L_2}v_{\mu}), \mathcal{F}(C_{L_0}w_{\lambda_0}),\mathcal{F}(C_{L_1}u_{\lambda_1}))\Big|.
\end{align}
We apply lemma $2.4$ to the above to obtain
\begin{align*}
(3.11)&\lesssim \sum_{\substack {L_1,L_0\ll \lambda_0^2,\\ \mu^2\lesssim L_2 \lesssim \lambda_0^2}}\lambda_1^{-\frac{1}{2}}\log{\lambda_1}(L_0L_1L_2)^{\frac{1}{2}}\Vert C_{L_1}u_{\lambda_1}\Vert_{L_{t,x}^2} \Vert Q_{L_2}v_{\mu} \Vert_{L_{t,x}^2} \Vert C_{L_0}w_{\lambda_0} \Vert_{L_{t,x}^2}\\
&\lesssim \mu^{-\beta}\lambda_0^{2\theta-\frac{3}{2}} \log{\lambda_1} \log{\mu}~\lambda_1^s \Vert \langle\tau+|\xi|^2 \rangle^{\theta} \mathcal{F}(C_{\ll \lambda_1^2}u_{\lambda_1}) \Vert_{L_{\tau,\xi}^2}  \\
&\quad \times\mu^{\beta}\Vert \langle \tau-|\xi|\rangle  \mathcal{F}(Q_{\mu^2 \lesssim \cdot \lesssim \lambda_0^2})v_{\mu} \Vert_{L_{\tau,\xi}^2} \lambda_0^{-s}\Vert \langle\tau+|\xi|^2 \rangle^{1-\theta}\mathcal{F}(C_{\ll \lambda_0^2} w_{\lambda_0})\Vert_{L_{\tau,\xi}^2}\\
&\lesssim\mu^{-\beta}\lambda_0^{2\theta-\frac{3}{2}} \log{\lambda_1} \log{\mu} \Vert u_{\lambda_1}\Vert_{S_{\lambda_1}^{s,l,\theta}} \Vert v_{\mu}\Vert_{W_{\mu}^{l,s,\theta}}\Vert  w_{\lambda_0}\Vert_{S_{\lambda_0}^{-s,l,1-\theta}}.
\end{align*}

\noindent{\boldmath $d=1:$} We apply H\"older's inequality and bilinear Strichartz estimate for wave-Schr\"odinger interaction $(2.13)$ as follows:
\begin{align*}
|I&(\mathcal{F}(Q_{\mu^2\lesssim \cdot \lesssim \lambda_0^2}v_{\mu}), \mathcal{F}(C_{\ll \lambda_0^2}w_{\lambda_0}),\mathcal{F}(C_{\ll \lambda_1^2}u_{\lambda_1}))|\\
&\lesssim \Vert C_{\ll \lambda_1^2}u_{\lambda_1}\Vert_{L_{t,x}^2} \Vert   \overline{Q_{\mu^2\lesssim \cdot \lesssim \lambda_0^2}v_{\mu}} C_{\ll \lambda_0^2}w_{\lambda_0}\Vert_{L_{t,x}^2}\\
&\lesssim \Vert C_{\ll \lambda_1^2}u_{\lambda_1}\Vert_{L_{t,x}^2} \sum_{\substack{L_0\ll \lambda_0^2\\\mu^2\lesssim L_2\lesssim \lambda_0^2}}(L_0L_2)^{\frac{1}{2}}\frac{1}{\lambda_0^{\frac{1}{2}}}\Vert Q_{L_2}v_{\mu}\Vert_{L_{t,x}^2}\Vert C_{L_0}w_{\lambda_0}\Vert_{L_{t,x}^2}\\
&\lesssim \lambda_0^{-\frac{3}{2}+2\theta}\log{\mu}~\mu^{-\beta}\lambda_1^{s}\Vert \langle \tau+|\xi|^2\rangle^{\theta}C_{\ll \lambda_1^2}u_{\lambda_1}\Vert_{L_{t,x}^2} \mu^{\beta-1}\Vert \langle \tau-|\xi|\rangle \mathcal{F}(Q_{\mu^2\lesssim \cdot \lesssim \lambda_0^2}v_{\mu}) \Vert_{L_{\tau,\xi}^2}\\
&\quad \times \lambda_0^{-s}\Vert \langle \tau+|\xi|^2\rangle^{1-\theta}\mathcal{F}(C_{\ll \lambda_0^2}w_{\lambda_0})\Vert_{L_{\tau,\xi}^2}\\
&\lesssim \lambda_0^{-\frac{3}{2}+2\theta}\log{\mu}~\mu^{-\beta}\Vert u_{\lambda_1}\Vert_{S^{s,l,\theta}_{\lambda_1}} \Vert v_{\mu}\Vert_{W^{l,s,\theta}_{\mu}}\Vert w_{\lambda_0}\Vert_{S^{-s,l,1-\theta}_{\lambda_0}}.
\end{align*}
The constraint $\beta>-\frac{3}{2}+2\theta$ enables us to sum the above estimate for $d\leqslant 3$.\\	
For $\mu\sim \lambda_1\sim 1$ in $d\leqslant 3$, we have
\begin{align*}
|I(\mathcal{F}(Q_{\mu^2\lesssim \cdot \lesssim \lambda_0^2}v_{\mu}), \mathcal{F}(C_{\ll \lambda_0^2}w_{\lambda_0}),\mathcal{F}(C_{\ll \lambda_1^2}u_{\lambda_1}))|& \lesssim \Vert C_{\ll \lambda_1^2}u_{\lambda_1}\Vert_{L_{t}^{\infty}L_x^2} \Vert Q_{\mu^2 \lesssim \cdot \lesssim \lambda_0^2}v_{\mu}\Vert_{L_{t,x}^2} \Vert C_{\ll \lambda_0^2}w_{\lambda_0}\Vert_{L_{t,x}^2}\\
&\lesssim \Vert u_{\lambda_1}\Vert_{S^{s,l,\theta}_{\lambda_1}} \Vert v_{\mu}\Vert_{W^{l,s,\theta}_{\mu}}\Vert w_{\lambda_0}\Vert_{S^{-s,l,1-\theta}_{\lambda_0}}.
\end{align*}
\begin{remark}
	In case the wave temporal frequencies are of size $\sim \lambda_0^2$, a simpler argument applies.
\end{remark}

	{6. \boldmath $L\times L \rightarrow L$}\\
{\boldmath $d=2,3:$} We consider the expression $I$, decompose the space-time Fourier supports of $C_{\ll \lambda_1^2}u_{\lambda_1}, Q_{\ll \lambda_0^2}v_{\mu}$ and $ C_{\ll \lambda_0^2}w_{\lambda_0}$ into pieces of size $L_1, L_2$ and $L_0$, respectively to obtain
\begin{align}
|I( \mathcal{F}(Q_{\ll \lambda_0^2}v_{\mu}), \mathcal{F}(C_{\ll \lambda_0^2}w_{\lambda_0}),\mathcal{F}(C_{\ll \lambda_1^2}u_{\lambda_1}))|
%&\lesssim \Big |I \Big(\sum_{L_1 \ll \lambda_1^2}C_{L_1}u_{\lambda_1}, \sum_{L_2 \ll \lambda_0^2}Q_{L_2} v_{\mu}, \sum_{L_0 \ll \lambda_0^2} C_{L_0}w_{\lambda_0}\Big)\Big|\\
&\lesssim \sum_{L_1,L_2,L_0\ll \lambda_0^2}\Big|I( \mathcal{F}(Q_{L_2}v_{\mu}), \mathcal{F}(C_{L_0}w_{\lambda_0}),\mathcal{F}(C_{L_1}u_{\lambda_1}))\Big|.
\end{align}
From lemma $2.4$, we have
\begin{align*}
(3.12)&\lesssim \sum_{L_1,L_2,L_0\ll \lambda_0^2}\lambda_1^{-\frac{1}{2}}\log{\lambda_1}(L_0L_1L_2)^{\frac{1}{2}}\Vert C_{L_1}u_{\lambda_1}\Vert_{L_{t,x}^2} \Vert Q_{L_2}v_{\mu} \Vert_{L_{t,x}^2} \Vert C_{L_0}w_{\lambda_0} \Vert_{L_{t,x}^2}\\
&\lesssim \mu^{-l}\lambda_0^{2\theta-\frac{3}{2}} \log{\lambda_1}~\lambda_1^s \Vert \langle\tau+|\xi|^2 \rangle^{\theta} \mathcal{F}(C_{\ll \lambda_1^2}u_{\lambda_1}) \Vert_{L_{\tau,\xi}^2}  \\
&\quad \times\mu^{l-a}\Vert \langle \tau-|\xi|\rangle^{\theta}(\mu+|\partial_t|)^a  \mathcal{F}(Q_{\ll \lambda_0^2}v_{\mu}) \Vert_{L_{\tau,\xi}^2} \lambda_0^{-s}\Vert \langle\tau+|\xi|^2 \rangle^{1-\theta} \mathcal{F}(C_{\ll \lambda_0^2} w_{\lambda_0})\Vert_{L_{\tau,\xi}^2}\\
&\lesssim\mu^{-l}\lambda_0^{2\theta-\frac{3}{2}} \log{\lambda_1} \Vert u_{\lambda_1}\Vert_{S_{\lambda_1}^{s,l,\theta}} \Vert v_{\mu}\Vert_{W_{\mu}^{l,s,\theta}}\Vert  w_{\lambda_0}\Vert_{S_{\lambda_0}^{-s,l,1-\theta}}.
\end{align*}

\noindent{\boldmath $d=1$:} Using Cauchy Schwarz inequality, we have
\begin{align}
|I( \mathcal{F}(Q_{\ll \lambda_0^2}v_{\mu}), \mathcal{F}(C_{\ll \lambda_0^2}w_{\lambda_0}),\mathcal{F}(C_{\ll \lambda_1^2}u_{\lambda_1}))|\lesssim \Vert C_{\ll \lambda_1^2}u_{\lambda_1}\Vert_{L_{t,x}^2} \Vert \overline{Q_{\ll \lambda_0^2}v_{\mu}} C_{\ll \lambda_0^2}w_{\lambda_0}\Vert_{L_{t,x}^2}.
\end{align}
We decompose the space-time Fourier support of $ \overline{Q_{\ll \lambda_0^2}v_{\mu}}$ and $C_{\ll \lambda_0^2}w_{\lambda_0}$ into pieces of size $L_2$ and $L_0$ respectively and apply the bilinear estimate $(2.13)$ for wave-Schr{\"o}dinger interaction:
\begin{align*}
(3.13)&\lesssim  \Vert C_{\ll \lambda_1^2}u_{\lambda_1}\Vert_{L_{t,x}^2} \sum_{L_0,L_2\ll \lambda_0^2}(L_0L_2)^{\frac{1}{2}}\frac{1}{\lambda_0^{\frac{1}{2}}}\Vert Q_{L_2}v_{\mu}\Vert_{L_{t,x}^2}\Vert C_{L_0}w_{\lambda_0}\Vert_{L_{t,x}^2}\\
&\lesssim \lambda_0^{-\frac{3}{2}+2\theta}\mu^{-l}	~\lambda_1^s \Vert \langle \tau+|\xi|^2\rangle^{\theta}\mathcal{F}(C_{\ll \lambda_1^2}u_{\lambda_1})\Vert_{L_{\tau,\xi}^2} \mu^{l-a}\Vert \langle \tau-|\xi|\rangle^{\theta}\mathcal{F}((\lambda_0+|\partial_t|)^a Q_{\ll \lambda_0^2}v_{\mu})\Vert_{L_{\tau,\xi}^2}\\
&\quad \times \lambda_0^{-s}\Vert \langle \tau+|\xi|^2\rangle^{1-\theta} \mathcal{F}(C_{\ll \lambda_0^2}w_{\lambda_0})\Vert_{L_{\tau,\xi}^2}\\
&\lesssim \lambda_0^{-\frac{3}{2}+2\theta}\mu^{-l} \Vert u_{\lambda_1}\Vert_{S_{\lambda_1}^{s,l,\theta}} \Vert v_{\mu}\Vert_{W_{\mu}^{l,s,\theta}}\Vert w_{\lambda_0}\Vert_{S_{\lambda_0}^{-s,l,1-\theta}}.
\end{align*}
Provided $l >-\frac{3}{2}+2\theta$, we can sum the estimate for $d\leqslant 3$ to obtain $(3.4)$.\\
For $\mu\sim \lambda_0\sim 1$, the estimate $(3.4)$ holds with trivial modification.\\
We conclude that we require
\begin{equation}
\bullet ~l>\max\Big\{-\frac{1}{2}+b, -\frac{3}{2}+2\theta\Big\} \quad 
\bullet ~\beta >-\frac{3}{2}+2\theta+a+b
\end{equation} 
to obtain the estimates $(3.4)$ and $(3.5)$ when $\lambda_0\sim \lambda_1$ for $d\leqslant 3$.\\

\textbf{Case III. Low to high interaction  (\boldmath $\lambda_0 \ll \lambda_1$)}\\
We decompose $u_{\lambda_1}$ and $v_{\lambda_1}$ as follows:
\begin{equation*}
u_{\lambda_1} = C_{\ll \lambda_1^2}u_{\lambda_1} + C_{\gtrsim \lambda_1^2}u_{\lambda_1}, \quad v_{\lambda_1} = Q_{\ll \lambda_1^2}v_{\lambda{_1}} + Q_{\gtrsim \lambda_1^2}v_{\lambda_1}
\end{equation*}

{1. \boldmath $H \times H \rightarrow H$}\\
Using the size of the modulation, Bernstein's inequality and the product estimate, we get
%(the case in which the output has temporal frequencies of size $\sim \lambda_0^2$ will be considered as a low modualtion case later)
\begin{align*}
&\lambda_0^{s'-2a+b}\Vert \langle \tau+|\xi|^2\rangle^{\theta'-1}\mathcal{F}((\lambda_0+|\partial_t|)^a C_{\gtrsim \lambda_0^2}(C_{\gtrsim \lambda_1^2}u_{\lambda_1}Q_{\gtrsim \lambda_1^2}v_{\lambda_1}))\Vert_{L_{\tau,\xi}^2}\\
&\lesssim \lambda_0^{s'-2a+b+2\theta'-2+\frac{d}{2}}\lambda_1^{-a}\Vert (\lambda_1+|\partial_t|)^a C_{\gtrsim \lambda_1^2}u_{\lambda_1}\Vert_{L_{t}^{\infty}L_x^2} \Vert (\lambda_1+|\partial_t|)^a Q_{\gtrsim \lambda_1^2}v_{\lambda_1}\Vert_{L_{t,x}^2}\\
&\lesssim \lambda_0^{s-2a+b+2\theta-2+\frac{d}{2}}\lambda_1^{-s+a-\beta+1}\sup_{|\tau_2|\gtrsim \lambda_1^2}\frac{(\lambda_1+|\tau_2|)^a}{\langle \tau_2-\lambda_1\rangle}\lambda_1^{s'-2a}\Vert (\lambda_1+|\partial_t|)^a C_{\gtrsim \lambda_1^2}u_{\lambda_1}\Vert_{L_{t}^{\infty}L_x^2}\\
&\quad \times\lambda_1^{\beta-1}\Vert \langle\tau - |\xi|\rangle \mathcal{F}(Q_{\gtrsim \lambda_1^2}v_{\lambda_1})\Vert_{L_{\tau,\xi}^2}\\
&\lesssim \lambda_0^{s-2a+b+2\theta-2+\frac{d}{2}}\lambda_1^{-s-\beta+3a-1} \Vert u_{\lambda_1}\Vert_{S_{\lambda_1}^{s,l,\theta}} \Vert v_{\lambda_1}\Vert_{W_{\lambda_1}^{l,s,\theta}}.
\end{align*}
If $s+\beta\geqslant -1+3a$ and $\beta\geqslant-3+\frac{d}{2}+2\theta+a+b$, we can obtain $(3.5)$ by summing up the above estimate.\\

{2. \boldmath $H\times H \rightarrow L$}\\
Since $|\tau_2|\gtrsim \lambda_1^2$, from the relation $|\tau_0 +|\xi_0|^2| = |\tau_1 +\tau_2+|\xi_0|^2|\ll \lambda_0^2$, it follows that $|\tau_1| \gtrsim \lambda_1^2$.
We consider the expression $I$ and use H{\"o}lder's and Bernstein's inequality as follows:
\begin{align*}
&|I(\mathcal{F}(Q_{\gtrsim \lambda_1^2}v_{\lambda_1}), \mathcal{F}(C_{\ll \lambda_0^2}w_{\lambda_0}),\mathcal{F}(C_{\gtrsim \lambda_1^2}u_{\lambda_1}))| \\
&\lesssim \lambda_0^{\frac{d}{2}}\Vert C_{\gtrsim \lambda_1^2}u_{\lambda_1}\Vert_{L_{t,x}^2} \Vert Q_{\gtrsim \lambda_1^2}v_{\lambda_1} \Vert_{L_{t,x}^2} \Vert C_{\ll \lambda_0^2}w_{\lambda_0} \Vert_{L_{t}^{\infty}L_x^2}\\
&\lesssim \lambda_0^{\frac{d}{2}+s+2\theta-1}\lambda_1^{-s-2\theta-\beta-1}\lambda_1^{s'-2a}\Vert \langle \tau+|\xi|^2\rangle^{\theta'}\mathcal{F}((\lambda_1+|\partial_t|)^a C_{\gtrsim \lambda_1^2}u_{\lambda_1})\Vert_{L_{\tau,\xi}^2}\\
&\quad \times  \lambda_1^{\beta-1}\Vert \langle \tau-|\xi|\rangle \mathcal{F}(Q_{\gtrsim \lambda_1^2}v_{\lambda_1})\Vert_{L_{\tau,\xi}^2} \lambda_0^{-s}\Vert C_{\ll \lambda_0^2}w_{\lambda_0}\Vert_{L_{t}^{p}L_x^2}\\
&\lesssim \lambda_0^{\frac{d}{2}+s+2\theta-1}\lambda_1^{-s-2\theta-\beta-1} \Vert u_{\lambda_1}\Vert_{S_{\lambda_1}^{s,l,\theta}} \Vert v_{\lambda_1}\Vert_{W_{\lambda_1}^{l,s,\theta}} \Vert w_{\lambda_0}\Vert_{S^{-s,l,1-\theta}_{\lambda_0}},
\end{align*}
where the last inequality follows from the embedding $(2.7)$ applied for $\frac{1}{p}=\theta-\frac{1}{2}$.\\
Since $s+\frac{d}{2}+2\theta-1>0$, we require $\beta
\geqslant-2+\frac{d}{2}$ in order to obtain $(3.4)$.\\

{3. \boldmath $H\times L \rightarrow H$}\\
Using the size of the modulation, the product estimate and Bernstein's inequality, we have
%This interaction can be treated by assuming that the Schr{\"o}dinger solution $u_{\lambda_1}$ has modulation $\gg \lambda_1^2$, because the case in which the modulation is $\sim \lambda_1^2$ (i.e. the temporal frequencies are $\ll \lambda_1^2$) might lead to a low output modulation, which is considered in the subsequent case.
\begin{align*}
&\lambda_0^{s'-2a+b}\Vert \langle \tau+|\xi|^2\rangle^{\theta'-1}\mathcal{F}(	(\lambda_0 + |\partial_t|)^a C_{\gtrsim \lambda_0^2}(C_{\gtrsim \lambda_1^2}u_{\lambda_1}Q_{\ll \lambda_1^2}v_{\lambda_1} ))\Vert_{L_{\tau,\xi}^2}\\
&\lesssim \lambda_0^{s'-2a+b+2\theta'-2+\frac{d}{2}}\lambda_1^{-a}\Vert(\lambda_1+|\partial_t|)^a C_{\gtrsim \lambda_1^2}u_{\lambda_1}\Vert_{L_{t,x}^2}\Vert(\lambda_1+|\partial_t|)^a Q_{\ll \lambda_1^2}v_{\lambda_1}\Vert_{L_t^{\infty}L_x^2}\\
& \lesssim \lambda_0^{s-2a+b+2\theta-2+\frac{d}{2}}\lambda_1^{-s-2\theta-l+2a}\lambda_1^{s'-2a}\Vert\langle \tau+|\xi|^2\rangle^{\theta'}\mathcal{F}( (\lambda_1+|\partial_t|)^a C_{\gtrsim \lambda_1^2}u_{\lambda_1})\Vert_{L_{\tau,\xi}^2}\lambda_1^{l-a}\Vert(\lambda_1+|\partial_t|)^a Q_{\ll \lambda_1^2}v_{\lambda_1}\Vert_{L_t^{\infty}L_x^2}\\
&\lesssim \lambda_0^{s-2a+b+2\theta-2+\frac{d}{2}}\lambda_1^{-s-2\theta-l+2a}\Vert u_{\lambda_1} \Vert_{S_{\lambda_1}^{s,l,\theta}} \Vert v_{\lambda_1}\Vert_{W_{\lambda_1}^{l,s,\theta}}.
\end{align*}
The estimate can be summed up provided $s+l\geqslant-2\theta+2a$ and $l \geqslant-2+\frac{d}{2}+b$.\\

{4. \boldmath $H\times L \rightarrow L$}\\
Since low modulation for the wave $v_{\lambda_1}$ implies $|\tau_2|\ll \lambda_1^2$, we observe from the relation \begin{equation*}
|\tau_0+|\xi_0|^2| = |\tau_1 +\underbrace{\tau_2 - |\xi_2|}_{\ll \lambda_1^2}+\underbrace{|\xi_2|+|\xi_0|^2}_{\sim \max\{\lambda_1, \lambda_0^2\}}|\ll \lambda_0^2,
\end{equation*} that the temporal frequencies $|\tau_1|\ll \lambda_1^2$.  Using this information, we consider the following subcases for the size of the temporal frequencies $\tau_1$ and $\tau_2$ and make conclusions in the last two columns:
\begin{table}[H]
	\centering
	\caption{$H \times L\rightarrow L$}
	\vspace{-2mm}
	\begin{tabular}{|l|l|l|l|l|}
		\hline	$|\tau_1|$ & $|\tau_2|$	& $|\tau_0|$&  Conclusion \\ \hline
		$\lesssim \lambda_1$& $\lesssim \lambda_1$ & $\lesssim \lambda_1$ & $\lambda_0^2\lesssim \lambda_1$ \\ \hline

		$\lesssim \lambda_1$& $\lambda_1 \ll \cdot \ll \lambda_1^2$ & $\lambda_1 \ll \cdot \ll \lambda_1^2$ &$|\tau_W|\sim |\tau_0|,~ \lambda_1 \ll \lambda_0^2$	\\ \hline
		
		$\lambda_1\ll\cdot \ll \lambda_1^2$& $\lesssim \lambda_1$ & $\lambda_1 \ll \cdot \ll \lambda_1^2$ &$|\tau_S|\sim |\tau_0|, ~\lambda_1 \ll \lambda_0^2$	\\ \hline
		
		$\lambda_1 \ll\cdot \ll \lambda_1^2$& $\lambda_1 \ll \cdot \ll \lambda_1^2$ & $\lambda_1 \ll \cdot \ll \lambda_1^2$ &$|\tau_S|,|\tau_W|\sim \lambda_0^2, ~\lambda_1 \ll \lambda_0^2$	\\ \hline
		
		$\lambda_1 \ll\cdot \ll \lambda_1^2$& $\lambda_1 \ll \cdot \ll \lambda_1^2$ & $\ll \lambda_1$ &$\lambda_0^2 \ll \lambda_1$	\\ \hline	
	\end{tabular}
\end{table}
\vspace{-2mm}
In all the above cases, we can bound the weight $(\lambda_1+|\tau_1|)^{-a}(\lambda_1+|\tau_2|)^{-a}$ by $\lambda_1^{-a}\lambda_0^{-2a}$. We proceed to prove $(3.4)$ by duality and consider the expression $I$. Using Cauchy-Schwarz inequality, we get
\begin{align}
&|I(\mathcal{F}(Q_{\ll \lambda_1^2}v_{\lambda_1}), \mathcal{F}(C_{\ll \lambda_0^2}w_{\lambda_0}),\mathcal{F}(C_{\sim \lambda_1^2}u_{\lambda_1}))| \lesssim \Vert C_{\sim \lambda_1^2}u_{\lambda_1}\Vert_{L_{t,x}^2} \Vert \overline{Q_{\ll \lambda_1^2}v_{\lambda_1}}C_{\ll \lambda_0^2}w_{\lambda_0}\Vert_{L_{t,x}^2}.
\end{align}
In order to apply the bilinear estimate $(2.13)$, we decompose the space-time Fourier supports of the two terms in RHS of the above display into pieces of size $L_2$ and $L_0$, respectively and obtain:
\begin{align*}
(3.15)&\lesssim \Vert C_{\sim \lambda_1^2}u_{\lambda_1}\Vert_{L_{t,x}^2}\sum_{\substack{L_0\ll \lambda_0^2\\L_2 \ll \lambda_1^2}}\frac{\lambda_0^{\frac{d-1}{2}}}{\lambda_0^{\frac{1}{2}}}(L_0L_2)^{1/2}\Vert  Q_{L_2}v_{\lambda_1}\Vert_{L_{t,x}^2} \Vert C_{L_0}w_{\lambda_0}\Vert_{L_{t,x}^2}\\
&\lesssim \lambda_0^{\frac{d}{2}-2+2\theta+s-2a} \lambda_1^{-s-l+2a-2\theta}\lambda_1^{s'-2a}\Vert  \langle\tau + |\xi|^2\rangle^{\theta'}\mathcal{F}((\lambda_1 + |\partial_t|)^a C_{\sim\lambda_1^2}u_{\lambda_1})\Vert_{L_{\tau,\xi}^2}\\
&\quad\times \lambda_1^{l-a}\Vert \langle \tau-|\xi|\rangle^{\theta}\mathcal{F}((\lambda_1+|\partial_t|)^a Q_{\ll \lambda_1^2}v_{\lambda_1})\Vert_{L_{\tau,\xi}^2}\lambda_0^{-s}\Vert\langle\tau + |\xi|^2\rangle^{1-\theta}\mathcal{F}(C_{\ll \lambda_0^2}w_{\lambda_0})\Vert_{L_{\tau,\xi}^2}\\
&\lesssim \lambda_0^{\frac{d}{2}-2+2\theta+s-2a} \lambda_1^{-s-l+2a-2\theta} \Vert u_{\lambda_1}\Vert_{S_{\lambda_1}^{s,l,\theta}} \Vert v_{\lambda_1}\Vert_{W_{\lambda_1}^{l,s,\theta}}\Vert w_{\lambda_0}\Vert_{S_{\lambda_0}^{-s,l,1-\theta}}.
\end{align*}
$s+l\geqslant-2\theta+2a$ and $l \geqslant -2+\frac{d}{2}$ are the requirements for the summability of the estimate.\\

{5. \boldmath $L\times H \rightarrow L$}\\
We observe that the temporal frequencies $|\tau_2|\sim \lambda_1^2$.
Using H{\"o}lder's inequality, we have
\begin{align}
&|I(\mathcal{F}(Q_{\sim \lambda_1^2}v_{\lambda_1}), \mathcal{F}(C_{\ll\lambda_0^2}w_{\lambda_0}),\mathcal{F}(C_{\ll \lambda_1^2}u_{\lambda_1}))|\lesssim \Vert Q_{\sim \lambda_1^2}v_{\lambda_1}\Vert_{L_{t,x}^2} \Vert C_{\ll \lambda_1^2}u_{\lambda_1}\overline{C_{\ll \lambda_0^2}w_{\lambda_0}}\Vert_{L_{t,x}^2}.
\end{align}
We apply the bilinear estimate $(2.12)$ to the above by decomposing the space-time Fourier supports of $C_{\ll \lambda_1^2}u_{\lambda_1}$ and $\overline{C_{\ll \lambda_0^2}w_{\lambda_0}}$ into pieces of size $L_1$ and $L_0$ respectively:
\begin{align*}
(3.16)&\lesssim \Vert Q_{\sim \lambda_1^2}v_{\lambda_1}\Vert_{L_{t,x}^2} \sum_{\substack{L_0\ll \lambda_0^2\\L_1\ll \lambda_1^2}}(L_0L_1)^{\frac{1}{2}}\frac{\lambda_0^{\frac{d-1}{2}}}{\lambda_1^{\frac{1}{2}}}\Vert C_{L_1}u_{\lambda_1}\Vert_{L_{t,x}^2}\Vert C_{L_0}w_{\lambda_0}\Vert_{L_{t,x}^2}\\
&\lesssim \lambda_0^{\frac{d-3}{2}+s+2\theta}\lambda_1^{-\beta-s-\frac{3}{2}}  \lambda_1^{s}\Vert \langle \tau+|\xi|^2\rangle^{\theta}\mathcal{F}(C_{\ll \lambda_1^2}u_{\lambda_1})\Vert_{L_{\tau,\xi}^2} \lambda_1^{\beta-1}\Vert \langle \tau-|\xi|\rangle \mathcal{F}(Q_{\sim \lambda_1^2}v_{\lambda_1})\Vert_{L_{\tau,\xi}^2}\\
&\quad \times\lambda_0^{-s}\Vert \langle \tau+|\xi|^2\rangle^{1-\theta} \mathcal{F}(C_{\ll \lambda_0^2}w_{\lambda_0})\Vert_{L_{\tau,\xi}^2}\\
&\lesssim \lambda_0^{\frac{d-3}{2}+s+2\theta}\lambda_1^{-\beta-s-\frac{3}{2}}\Vert u_{\lambda_1} \Vert_{S_{\lambda_1}^{s,l,\theta}} \Vert v_{\lambda_1}\Vert_{W_{\lambda_1}^{l,s,\theta}}\Vert w_{\lambda_0}\Vert_{S^{-s,l,1-\theta}_{\lambda_0}}.
\end{align*}
One can sum the above estimate to obtain $(3.4)$ provided $s+\beta\geqslant-\frac{3}{2}$ and $\beta\geqslant -3+\frac{d}{2}+2\theta$.\\
For $d=1$, in the case $\lambda_0\sim \lambda_1\sim 1$, we have
\begin{align*}
|I(\mathcal{F}(Q_{\sim \lambda_1^2}v_{\lambda_1}), \mathcal{F}(C_{\ll\lambda_0^2}w_{\lambda_0}),\mathcal{F}(C_{\ll \lambda_1^2}u_{\lambda_1}))|
&\lesssim \Vert Q_{\sim \lambda_1^2}v_{\lambda_1}\Vert_{L_{t,x}^2} \Vert C_{\ll \lambda_1^2}u_{\lambda_1}\Vert_{L_{t}^{\infty}L_x^2} \Vert C_{\ll \lambda_0^2}w_{\lambda_0}\Vert_{L_{t,x}^2}\\
&\lesssim \Vert v_{\lambda_1}\Vert_{W^{l,s,\theta}_{\lambda_1}}\Vert u_{\lambda_1}\Vert_{S^{s,l,\theta}_{\lambda_1}} \Vert w_{\lambda_0}\Vert_{S^{-s,l,1-\theta}_{\lambda_0}}.
\end{align*}

{6. \boldmath $L\times L \rightarrow H$}\\
From the constraint $\tau_0 = \tau_1+\tau_2$, we note that the output temporal frequencies are of size $\sim \lambda_1^2$. This implies that the output has a modulation of size $\sim \lambda_1^2$. We use the size of the modulation and the temporal frequencies and employ the bilinear estimate for wave-Schr{\"o}dinger interaction. By orthogonality arguments, we reduce the estimates to the case when the spatial Fourier supports of $C_{\ll \lambda_1^2}u_{\lambda_1}$ and $Q_{\ll \lambda_1^2}v_{\lambda_1}$ are of size $\sim \lambda_0$.
\begin{align*}
&\lambda_0^{s'-2a+b}\Vert \langle \tau+|\xi|^2\rangle^{\theta'-1}	\mathcal{F}((\lambda_0 + |\partial_t|)^a C_{\gtrsim \lambda_0^2}(C_{\ll \lambda_1^2}u_{\lambda_1}Q_{\ll \lambda_1^2}v_{\lambda_1})) \Vert_{L_{\tau,\xi}^2}\\
&\lesssim \lambda_0^{s'-2a+b}\lambda_1^{2\theta'-2+2a}\Vert C_{\ll \lambda_1^2}u_{\lambda_1}Q_{\ll \lambda_1^2}v_{\lambda_1} \Vert_{L_{t,x}^2}\\
&\lesssim \lambda_0^{s'-2a+b}\lambda_1^{2\theta'-2+2a}\sum_{L_1,L_2\ll \lambda_1^2}\frac{\lambda_0^{\frac{d-1}{2}}}{\lambda_1^{\frac{1}{2}}}(L_1L_2)^{\frac{1}{2}} \Vert C_{L_1}u_{\lambda_1}\Vert_{L_{t,x}^2}\Vert  Q_{L_2}v_{\lambda_1}\Vert_{L_{t,x}^2}\\
&\lesssim \lambda_0^{s'-2a+b+\frac{d-1}{2}}\lambda_1^{2\theta'-\frac{5}{2}+2a-s-l} \Vert u_{\lambda_1}\Vert_{S_{\lambda_1}^{s,l,\theta}} \Vert v_{\lambda_1}\Vert_{W_{\lambda_1}^{l,s,\theta}}. 
\end{align*}
To sum this estimate, we require $s+l\geqslant -\frac{5}{2}+2\theta'+2a$ and  $l\geqslant -3+\frac{d}{2}+2\theta+b$.\\

{7. \boldmath $L\times H \rightarrow H$}\\
With Bernstein's inequality and the product estimate, we get
%(the case in which the output has temporal frequencies of size $\sim \lambda_0^2$ will be considered as a low modualtion case later)
\begin{align*}
&\lambda_0^{s'-2a+b}\Vert \langle \tau+|\xi|^2\rangle^{\theta'-1} \mathcal{F}((\lambda_0+|\partial_t|)^a C_{\gtrsim \lambda_0^2}(C_{\ll \lambda^2}u_{\lambda_1}Q_{\gtrsim \lambda_1^2}v_{\lambda_1}))\Vert_{L_{\tau,\xi}^2}\\
&\lesssim \lambda_0^{s'-2a+b+2\theta'-2+\frac{d}{2}}\lambda_1^{-a}\Vert (\lambda_1+|\partial_t|)^a C_{\ll \lambda^2}u_{\lambda_1}\Vert_{L_{t}^{\infty}L_x^2} \Vert (\lambda_1+|\partial_t|)^a Q_{\gtrsim \lambda_1^2}v_{\lambda_1}\Vert_{L_{t,x}^2}\\
&\lesssim \lambda_0^{s'-2a+b+2\theta'-2+\frac{d}{2}}\lambda_1^{-s+a-\beta+1}\sup_{|\tau_2|\gtrsim \lambda_1^2}\frac{(\lambda_1+|\tau_2|)^a}{\langle \tau_2-\lambda_1\rangle}~\lambda_1^{s}\Vert \langle \tau+|\xi|^2\rangle^{\theta}\mathcal{F}(C_{\ll \lambda^2}u_{\lambda_1})\Vert_{L_{\tau,\xi}^2}\\
 &\quad \times \lambda_1^{\beta-1}\Vert \langle\tau - |\xi|\rangle \mathcal{F}(Q_{\gtrsim \lambda_1^2}v_{\lambda_1})\Vert_{L_{\tau,\xi}^2}\\
&\lesssim \lambda_0^{s'-2a+b+2\theta'-2+\frac{d}{2}}\lambda_1^{-s-\beta+3a-1} \Vert u_{\lambda_1}\Vert_{S_{\lambda_1}^{s,l,\theta}} \Vert v_{\lambda_1}\Vert_{W_{\lambda_1}^{l,s,\theta}}.
\end{align*}
If $s+\beta\geqslant -1+3a$ and $\beta\geqslant-3+\frac{d}{2}+2\theta+a+b$, we can obtain $(3.5)$ by summing up the above estimate.\\

Hence, for $\lambda_0\ll \lambda_1$, we conclude that the following constraints are required for summability in dimension $d\leqslant 3$:
\begin{equation}
\bullet~  l \geqslant -\frac{3}{2}+2\theta+b \quad \bullet~ \beta\geqslant -\frac{3}{2}+2\theta+a+b \quad \bullet~ s+l \geqslant \max\Big\{-2\theta+2a, -\frac{5}{2}+2\theta'+2a\Big\} \quad  \bullet ~s+\beta\geqslant -1+3a.
\end{equation}

For  $\frac{1}{2}\leqslant a <1$, the estimate $(3.6)$ holds provided $l\geqslant  -\frac{3}{2}+2\theta, s+l\geqslant 0$ and $s-l\leqslant 3- 2\theta$. It can be proved by imitating the proof of the product estimate $\Vert fg\Vert_{H^{s-3+2\theta}}\lesssim \Vert f\Vert_{H^s}\Vert g\Vert_{H^l}$ for each case above.\\
From $(3.10),(3.14)$ and $(3.17)$, we conclude that
\begin{align*}
&\bullet~l>-\frac{3}{2}+2\theta+b \quad \bullet ~ \beta>-\frac{3}{2}+2\theta+a+b  \quad \bullet~s-l\leqslant  \min\{2-2\theta+a-b, 3-2\theta\} \\
& \bullet~s-\beta\leqslant \min\Big\{3-2\theta-b,\frac{5}{2}-2\theta\Big\} \quad \bullet~s+l\geqslant \max\Big\{-\frac{5}{2}+2\theta'+2a, -2\theta+2a\Big\} \quad \bullet s+\beta\geqslant -1+3a
\end{align*}
are the requirements for the estimate $(3.1)$ to hold.
\end{proof}

\section{Multilinear estimates for wave non-linearity}
\begin{theorem}
	Let $d\leqslant 3$ and $s,l$ satisfy $(1.3)$. There exist $a,b,s,l, \beta,\theta\in \mathbb{R}$ such that the estimate
	\begin{equation}
	\Vert |\nabla|(\overline{\varphi}\psi)\Vert_{R^{l,s,\theta-1}} \lesssim \Vert \varphi\Vert_{S^{s,l,\theta}} \Vert \psi\Vert_{S^{s,l,\theta}}
	\end{equation}
	holds.
	
\end{theorem}
\begin{proof}
	We choose the parameters as in $(2.2)$. Using the definition of the norms, it suffices to show the following:
	\begin{align}
	\Big(\sum_{\mu \in 2^{\mathbb{N}}}\mu^{2(l-a+1)}\Vert \langle \tau-|\xi|\rangle ^{\theta'-1}\mathcal{F}((\mu+|\partial_t|)^aQ_{\ll\mu^2}P_{\mu}(\bar{\varphi}\psi))\Vert_{L_{\tau,\xi}^2}^2\Big)^{\frac{1}{2}}\lesssim \Vert \varphi\Vert_{S^{s,l,\theta}}\Vert \psi \Vert_{S^{s,l,\theta}}\\
	\Big(\sum_{\mu \in 2^{\mathbb{N}}}\mu^{2\beta}\Vert \mathcal{F}(Q_{\gtrsim\mu^2}P_{\mu}(\bar{\varphi}\psi))\Vert_{L_{\tau,\xi}^2}^2\Big)^{\frac{1}{2}}\lesssim \Vert \varphi\Vert_{S^{s,l,\theta}}\Vert \psi \Vert_{S^{s,l,\theta}}\\
	\Big(\sum_{\mu \in 2^{\mathbb{N}}}\mu^{2(l+2\theta-2)}\Vert P_{\mu}(\bar{\varphi}\psi)\Vert_{L_{t}^{\infty}L_x^2}^2\Big)^{\frac{1}{2}}\lesssim \Vert \varphi\Vert_{S^{s,l,\theta}}\Vert \psi \Vert_{S^{s,l,\theta}}
	\end{align}
	The definition of high and low modulation remains the same as in section 3, so do the abbreviations $H$ and $L$.  We now append $1$ or $2$ as subscripts with $\tau$ to denote the temporal frequencies of $\varphi$ and $\psi$. The same is done for the spatial frequencies.\\
	Employing the frequency trichotomy, we decompose the non-linearity $P_{\mu}(\bar{\varphi}\psi)$ into high-low, balanced (high-high) and low-high interactions as follows:
\begin{equation}
P_{\mu}(\overline{\varphi}\psi) = P_{\mu}(\overline{\varphi}\psi_{\ll\mu})	+\sum_{\mu \lesssim \lambda_1\sim \lambda_2 }P_{\mu}(\overline{\varphi_{\lambda_1}}\psi_{\lambda_2})+P_{\mu}(\overline{\varphi_{\ll\mu}}\psi) =\sum_{\lambda\ll \mu} P_{\mu}(\overline{\varphi_{\mu}}\psi_{\lambda})	+\sum_{\mu \lesssim \lambda_1\sim \lambda_2 }P_{\mu}(\overline{\varphi_{\lambda_1}}\psi_{\lambda_2})+\sum_{\lambda\ll\mu}P_{\mu}(\overline{\varphi_{\lambda}}\psi_{\mu})
\end{equation}
It suffices to consider the first two interactions above. 
\begin{remark}
	As in section 3, we consider the cases where $|\xi_i|,|\tau_i|\gg 1, i=1,2$. For the very low frequency cases ($|\xi_i|\sim 1, i=1,2$ or the lowest frequency is very small), most of the estimates follow analogously without having to decompose the space-time Fourier supports. The cases in which these arguments do not follow will be treated separately.
\end{remark}
	
\textbf{Case I.  High-low interaction (\boldmath $\mu\gg \lambda$)}\\
We decompose the spatially localised waves $\overline{\varphi_{\mu}}$ and $\psi_{\lambda}$ as follows:
	\begin{equation*}
\overline{\varphi_{\mu}} = \overline{C_{\ll \mu^2}\varphi_{\mu}} 
%+ \overline{P_{\mu \ll\cdot \lesssim \mu^2}^{(t)}\varphi_{\mu}} 
+ \overline{C_{\gtrsim \mu^2}\varphi_{\mu}}, \quad
\psi_{\lambda} = C_{\ll \lambda^2}\psi_{\lambda}
% + P_{\lambda \ll\cdot \lesssim \lambda^2}^{(t)}\psi_{\mu}
+ C_{\gtrsim \lambda^2} \psi_{\lambda}.\\
\end{equation*}
The following cases can be identified from the size of the modulation.\\
\textbf{1. \boldmath $H\times H \rightarrow H$}\\
Since $\mu^2 \lesssim |\tau_0 -|\xi_0||  = |\tau_1 -\tau_2-|\xi_0||$, we conclude that atleast one of the temporal frequencies $\tau_1$ or $\tau_2$  has size $\sim \mu^2$. WLOG, let $\tau_1$ be of size $\sim \mu^2$.
Using H{\"o}lder's and Bernstein's inequality, we have
\begin{align*}
&\mu^{\beta}\Vert \mathcal{F}(Q_{\gtrsim \mu^2}(\overline{C_{\gtrsim \mu^2}\varphi_{\mu}} C_{\gtrsim \lambda^2}\psi_{\lambda}))\Vert_{L_{\tau,\xi}^2}\\
&\lesssim \mu^{\beta}\lambda^{\frac{d}{2}}\Vert \overline{C_{\gtrsim \mu^2}\varphi_{\mu}} \Vert_{L_{t,x}^2} \Vert C_{\gtrsim \lambda^2}\psi_{\lambda}\Vert_{L_t^{\infty}L_x^2}\\
%&\lesssim \mu^{\beta-s-2\theta-b}\lambda^{\frac{d}{2}}\mu^{s'-2a+b}\Vert \langle \tau+|\xi|^2\rangle^{\theta'} \mathcal{F}((\mu+|\partial_t|)^a C_{\gtrsim \mu^2}\varphi_{\mu})\Vert_{L_{\tau,\xi}^2}  \Vert C_{\gtrsim \lambda^2}\psi_{\lambda}\Vert_{L_t^{\infty}L_x^2}\\
&\lesssim \mu^{\beta-s-2\theta-b}\lambda^{\frac{d}{2}-s+a}
 \mu^{s'-2a+b}\Vert \langle \tau+|\xi|^2 \rangle^{\theta'}\mathcal{F}((\mu+|\partial_t|)^a C_{\gtrsim \mu^2}\varphi_{\mu})\Vert_{L_{\tau,\xi}^2 }\\
 &\quad \times \lambda^{s'-2a}\Vert \langle \tau+|\xi|^2\rangle^{\theta'} \mathcal{F}((\lambda+|\partial_t|)^a C_{\gtrsim \lambda^2}\psi_{\lambda})\Vert_{L_{\tau,\xi}^2 }\\
&\lesssim \mu^{\beta-s-2\theta-b}\lambda^{\frac{d}{2}-s+a} \Vert \varphi_{\mu}\Vert_{S_{\mu}^{s,l,\theta}}\Vert \psi_{\lambda}\Vert_{S_{\lambda}^{s,l,\theta}}.
\end{align*}
We require $s-\beta\geqslant-2\theta-b$ and $2s-\beta\geqslant\frac{d}{2}-2\theta+a-b$.\\

\textbf{2. \boldmath $H \times H \rightarrow L$}\\
We consider the expression $I$ and employ  H{\"o}lder's and Bernstein's inequality to obtain
\begin{align*}
&	|I(\mathcal{F}(Q_{\ll \mu^2}\eta_{\mu}),\mathcal{F}( C_{\gtrsim \lambda^2}\psi_{\lambda}),\mathcal{F}(C_{\gtrsim \mu^2}\varphi_{\mu}))|\\
&\lesssim \lambda^{\frac{d}{2}}\Vert C_{\gtrsim \mu^2}\varphi_{\mu} \Vert_{L_{t,x}^2} \Vert C_{\gtrsim \lambda^2}\psi_{\lambda}\Vert_{L_{t,x}^2} \Vert Q_{\ll \mu^2}\eta_{\mu}\Vert_{L_{t}^{\infty}L_x^2} \\
&\lesssim \lambda^{\frac{d}{2}-s+a-2
	\theta}\mu^{l-s+a-b}\mu^{s'-2a+b}\Vert \langle \tau+|\xi|^2 \rangle^{\theta'} \mathcal{F}((\mu+|\partial_t|)^a C_{\gtrsim \mu^2}\varphi_{\mu}) \Vert_{L_{\tau,\xi}^2} \\
&\quad \times\lambda^{s'-2a} \Vert \langle \tau+|\xi|^2 \rangle^{\theta'} \mathcal{F}((\lambda+|\partial_t|)^a C_{\gtrsim \lambda^2}\psi_{\lambda})\Vert_{L_{\tau,\xi}^2}~\mu^{-l-1} \Vert Q_{\ll \mu^2}\eta_{\mu}\Vert_{L_t^pL_x^2}\\
&\lesssim \lambda^{\frac{d}{2}-s+a-2\theta} \mu^{l-s+a-b} \Vert \varphi_{\mu}\Vert_{S_{\mu}^{s,l,\theta}} \Vert \psi_{\lambda} {\Vert_{S_{\lambda}^{s,l,\theta}}} \Vert \eta_{\mu} \Vert_{W_{\mu}^{-l,s,1-\theta}}.
\end{align*}
The last inequality follows from the embedding $(2.7)$ for $\frac{1}{p}=\theta-\frac{1}{2}$. We require $s-l\geqslant a-b$, $2s-l\geqslant \frac{d}{2}-1+2a-b$ to obtain $(4.2)$.\\

\textbf{3. \boldmath  $H \times L \rightarrow H$}\\
Using the relation $\mu^2 \lesssim |\tau_0 -|\xi_0|| = |\tau_1 -\tau_2  -\mu|$ and $|\tau_2|\sim \lambda^2$, we find that $|\tau_1|\gtrsim \mu^2$. Employing H{\"o}lder's and Bernstein's inequality, we obtain
\begin{align*}
&\mu^{\beta}\Vert\mathcal{F}( Q_{\gtrsim \mu^2}(\overline{P_{\gtrsim \mu^2}^{(t)}C_{\gtrsim \mu^2}\varphi_{\mu}}C_{\ll \lambda^2}\psi_{\lambda}))\Vert_{L_{\tau,\xi}^2}\\
&\lesssim \mu^{\beta}\lambda^{\frac{d}{2}}\Vert \overline{P_{\gtrsim \mu^2}^{(t)}C_{\gtrsim \mu^2}\varphi_{\mu}} \Vert_{L_{t,x}^2} \Vert C_{\ll \lambda^2}\psi_{\lambda}\Vert_{L_{t}^{\infty}L_x^2}\\
&\lesssim \mu^{\beta-s-b-2\theta}\lambda^{\frac{d}{2}-s}\mu^{s'-2a+b}\Vert \langle \tau+|\xi|^2\rangle^{\theta'}\mathcal{F}((\mu+|\partial_t|)^a C_{\gtrsim \mu^2}\varphi_{\mu})\Vert_{L_{\tau,\xi}^2} \lambda^{s}\Vert C_{\ll \lambda^2}\psi_{\lambda} \Vert_{L_{t}^{\infty}L_x^2}\\
&\lesssim \mu^{\beta-s-b-2\theta}\lambda^{\frac{d}{2}-s} \Vert \varphi_{\mu}\Vert_{S_{\mu}^{s,l,\theta}} \Vert \psi_{\lambda} {\Vert_{S_{\lambda}^{s,l,\theta}}}.
\end{align*}
For $s-\beta\geqslant-b-2\theta$ and $2s-\beta\geqslant\frac{d}{2}-2\theta-b$, we can sum this estimate and obtain $(4.3)$.\\

	\textbf{4. \boldmath $H \times L \rightarrow L$}\\
From the relation $|\tau_0| = |\tau_1-\tau_2|\ll \mu^2$ and $|\tau_2|\sim \lambda^2$, we conclude that the temporal frequencies $\tau_1$ are such that $|\tau_1| \ll \mu^2$. We use Cauchy-Schwarz inequality for the expression $I$ as follows 
\begin{align}
|I( \mathcal{F}(Q_{\ll \mu^2}\eta_{\mu}), \mathcal{F}(C_{\ll \lambda^2}\psi_{\lambda}),\mathcal{F}(C_{\sim \mu^2}\varphi_{\mu}))|
\lesssim \Vert C_{\sim \mu^2}\varphi_{\mu}\Vert_{L_{t,x}^2}\Vert C_{\ll \lambda^2}\psi_{\lambda} \overline{Q_{\ll \mu^2}\eta_{\mu}}\Vert_{L_{t,x}^2}.
\end{align}
To employ the bilinear estimate $(2.13)$ for the last term in the above display, we decompose the space-time Fourier supports of $C_{\ll \lambda^2}\psi_{\lambda}$ and $\overline{Q_{\ll \mu^2}\eta_{\mu}}$ into pieces of size $L_2$ and $L_0$ respectively:
\begin{align*}
(4.6)\lesssim \Vert C_{\sim \mu^2}\varphi_{\mu}&\Vert_{L_{t,x}^2}\sum_{\substack{L_2\ll \lambda^2\\L_0\ll \mu^2}}\frac{\lambda^{\frac{d-1}{2}}}{\lambda^{\frac{1}{2}}}(L_1L_0)^{\frac{1}{2}}\Vert C_{L_1}\psi_{\lambda}\Vert_{L_{t,x}^2}\Vert Q_{L_0}\eta_{\mu}\Vert_{L_{t,x}^2}\\
&\lesssim \mu^{-s+a-b+l}\lambda^{\frac{d}{2}-1-s}~\mu^{s'-2a+b}\Vert\langle \tau+|\xi|^2\rangle^{\theta'}\mathcal{F}((\mu+|\partial_t|)^a C_{\sim \mu^2}\varphi_{\mu})\Vert_{L_{\tau,\xi}^2}\\
&\quad \times \lambda^{s}\Vert\langle \tau+|\xi|^2\rangle^{\theta} 
\mathcal{F}(C_{\ll \lambda^2}\psi_{\lambda})\Vert_{L_{\tau,\xi}^2}~
\mu^{-l-1}\Vert \langle \tau-|\xi|\rangle^{1-\theta}\mathcal{F}(Q_{\ll \mu^2}\eta_{\mu})\Vert_{L_{\tau,\xi}^2}\\
&\lesssim \mu^{-s+a-b+l}\lambda^{\frac{d}{2}-1-s}~ \Vert \varphi_{\mu}\Vert_{S_{\mu}^{s,l,\theta}} \Vert \psi_{\lambda} {\Vert_{S_{\lambda}^{s,l,\theta}}} \Vert \eta_{\mu} \Vert_{W_{\mu}^{-l,s,1-\theta}}.
\end{align*}
Provided $s-l \geqslant a-b$ and $2s-l\geqslant \frac{d}{2}-1+a-b$, we can sum this estimate.\\

\textbf{5. \boldmath $L \times H \rightarrow H$}\\
We consider two cases for the temporal frequencies $\tau_2$:\\
\textbf{a.	\boldmath $|\tau_2|\lesssim \mu^2$}:
We apply the bilinear estimate $(2.12)$ directly by decomposing the space time Fourier supports of $\overline{C_{\ll\mu^2}\varphi_{\mu}}$ and $C_{\gtrsim \lambda^2}\psi_{\lambda}$ into pieces of size $L_1$ and $L_2$ respectively.
\begin{align*}
&\mu^{\beta}\Vert \mathcal{F}(Q_{\sim \mu^2}( \overline{C_{\ll\mu^2}\varphi_{\mu}}C_{ \lambda^2 \lesssim\cdot \lesssim \mu^2}\psi_{\lambda})) \Vert_{L_{\tau,\xi}^2}\\
&\lesssim
\mu^{\beta}\sum_{\substack{L_1\ll \mu^2\\L_2\lesssim \mu^2}}\frac{\lambda^{\frac{d-1}{2}}}{\mu^{\frac{1}{2}}}(L_1L_2)^{\frac{1}{2}}\Vert C_{L_1} \varphi_{\mu}\Vert_{L_{t,x}^2}\Vert C_{L_2}\psi_{\lambda}\Vert_{L_{t,x}^2}\\
&\lesssim \mu^{\beta-\frac{1}{2}-s}\lambda^{\frac{d-1}{2}-s+a-b}~\mu^s\Vert \langle \tau+|\xi|^2\rangle^{\theta} \mathcal{F}(C_{\ll\mu^2}\varphi_{\mu})\Vert_{L_{\tau,\xi}^2} \lambda^{s'-2a+b}\Vert \langle \tau+|\xi|^2\rangle^{\theta'} \mathcal{F}((\lambda+|\partial_t|)^a C_{\lambda^2 \lesssim\cdot \lesssim \mu^2}\psi_{\lambda})\Vert_{L_{\tau,\xi}^2}\\
& \lesssim \mu^{\beta-\frac{1}{2}-s}\lambda^{\frac{d-1}{2}-s+a-b}  \Vert \varphi_{\mu}\Vert_{S_{\mu}^{s,l,\theta}} \Vert \psi_{\lambda} {\Vert_{S_{\lambda}^{s,l,\theta}}}.
\end{align*}
Provided $s-\beta\geqslant-\frac{1}{2}$ and $2s-\beta\geqslant\frac{d}{2}-1+a-b$, we can sum this estimate to obtain $(4.3)$.\\
For very small spatial frequencies i.e. $\lambda\sim \mu\sim 1$, we have
\begin{align}
\mu^{\beta}\Vert \mathcal{F}(Q_{\sim \mu^2}( \overline{C_{\ll\mu^2}\varphi_{\mu}}C_{ \lambda^2 \lesssim\cdot \lesssim \mu^2}\psi_{\lambda})) \Vert_{L_{\tau,\xi}^2} \lesssim \Vert C_{\ll\mu^2}\varphi_{\mu}\Vert_{L_t^{\infty}L_x^2} \Vert C_{ \lambda^2 \lesssim\cdot \lesssim \mu^2}\psi_{\lambda}\Vert_{L_{t,x}^2}\lesssim  \Vert \varphi_{\mu}\Vert_{S_{\mu}^{s,l,\theta}} \Vert \psi_{\lambda} {\Vert_{S_{\lambda}^{s,l,\theta}}}.
\end{align}
\noindent
\textbf{b.	\boldmath $|\tau_2|\gg \mu^2$}: We make use of the high modulation of $\psi_{\lambda}$ by using H{\"o}lder's and Bernstein's inequality:
\begin{align*}
\mu^{\beta}\Vert \mathcal{F}(Q_{\gg \mu^2}&(\overline{C_{\ll\mu^2}\varphi_{\mu}}C_{\gg \mu^2}\psi_{\lambda}))\Vert_{L_{\tau,\xi}^2} \\
&\lesssim\mu^{\beta}\lambda^{\frac{d}{2}} \Vert C_{\ll\mu^2}\varphi_{\mu} \Vert_{L_{t}^{\infty}L_x^2} \Vert C_{\gg \mu^2}\psi_{\lambda}\Vert_{L_{t,x}^2}\\
&\lesssim \mu^{\beta-2\theta-2a-s}\lambda^{\frac{d}{2}-s+2a-b}\mu^{s}\Vert C_{\ll\mu^2}\varphi_{\mu} \Vert_{L_{t}^{\infty}L_x^2} \lambda^{s'-2a+b}\Vert \langle \tau+|\xi|^2\rangle^{\theta'} \mathcal{F}((\lambda +|\partial_t|)^a C_{\gg \mu^2}\psi_{\lambda})\Vert_{L_{\tau,\xi}^2}\\
&\lesssim \mu^{\beta-2\theta-2a-s}\lambda^{\frac{d}{2}-s+2a-b}\Vert \varphi_{\mu} \Vert_{S^{s,l,\theta}_{\mu}} \Vert \psi_{\lambda}\Vert_{S^{s,l,\theta}_{\lambda}}.
\end{align*}
We require $s-\beta\geqslant -2\theta-2a$ and $2s-\beta\geqslant\frac{d}{2}-2\theta-b$ for summability.\\

	\textbf{6. \boldmath $L \times H \rightarrow L$}\\
From the relation $|\tau_0| = |\tau_1-\tau_2|\ll \mu^2$ nd $|\tau_1| \sim \mu^2$, we conclude that the temporal frequencies $|\tau_2|\sim \mu^2$. Cauchy-Schwarz inequality then gives
%\begin{equation*}
%|\tau_0 -|\xi_0|| = |\underbrace{\tau_1 +|\xi_1|^2}_{\ll \mu^2} -\tau_2 -|\xi_1|^2-|\xi_0||\ll \mu^2,
%\end{equation*}
%we conclude that the temporal frequencies $|\tau_2|\sim \mu^2$. Cauchy-Schwarz inequality then gives
\begin{align}
|I(\mathcal{F}(Q_{\ll \mu^2}\eta_{\mu}), \mathcal{F}(C_{\sim \mu^2}\psi_{\lambda}),\mathcal{F}(C_{\ll\mu^2}\varphi_{\mu}) )|\lesssim \Vert C_{\sim \mu^2}\psi_{\lambda}\Vert_{L_{t,x}^2} \Vert P_{\lambda}(C_{\ll \mu^2}\varphi_{\mu} Q_{\ll \mu^2}\eta_{\mu})\Vert_{L_{t,x}^2} 
\end{align}
In order to apply the bilinear estimate $(2.13)$, we decompose the Fourier supports of the last two terms in the above display into pieces of size $L_1$ and $L_0$ respectively. Since the spatial support of $\psi$ is localised to frequencies of size $\sim \lambda$, the estimate reduces to the case where the spatial supports of $\varphi$ and $\eta$ are also localised to frequencies of size $\sim \lambda$. Using this reduction, we have
\begin{align*}
(4.8)&\lesssim \Vert C_{\sim \mu^2}\psi_{\lambda}\Vert_{L_{t,x}^2}\sum_{L_0,L_1\ll \mu^2}\frac{\lambda^{\frac{d-1}{2}}}{\mu^{\frac{1}{2}}}(L_1L_0)^{\frac{1}{2}}\Vert C_{L_1}\varphi_{\mu}\Vert_{L_{t,x}^2}\Vert Q_{L_0}\eta_{\mu}\Vert_{L_{t,x}^2}\\
&\lesssim \mu^{-\frac{1}{2}+2(\theta-\theta')-2a+l-s}\lambda^{\frac{d-1}{2}-s'-b+2a}\mu^s \Vert \langle \tau+|\xi|^2\rangle^{\theta}\mathcal{F}(C_{\ll\mu^2}\varphi_{\mu}) \Vert_{L_{\tau,\xi}^2} \lambda^{s'-2a+b}\Vert \langle\tau+|\xi|^2\rangle \mathcal{F}(C_{\sim \mu^2}\psi_{\lambda})\Vert_{L_{\tau,\xi}^2}\\
&\quad \times \mu^{-l-1}\Vert \langle\tau-|\xi|\rangle^{1-\theta}\mathcal{F}(Q_{\ll \mu^2}\eta_{\mu})\Vert_{L_{\tau,\xi}^2}\\
&\lesssim \mu^{-\frac{1}{2}+2(\theta-\theta')-2a+l-s}\lambda^{\frac{d-1}{2}-s'-b+2a} \Vert \varphi_{\mu}\Vert_{S_{\mu}^{s,l,\theta}} \Vert \psi_{\lambda} {\Vert_{S_{\lambda}^{s,l,\theta}}} \Vert \eta_{\mu} \Vert_{W_{\mu}^{-l,s,1-\theta}}.
\end{align*}
Provided $s-l \geqslant -\frac{1}{2}+2(\theta-\theta')-2a$ and $2s-l \geqslant\frac{d}{2}-1-b$, we can obtain $(4.1)$ by summing the estimate.\\
For very small spatial frequencies $\lambda \sim \mu \sim 1$, the conclusion $|\tau_2|\sim \mu^2$ is not necessarily true. We have
\begin{align}
|I(\mathcal{F}(Q_{\ll \mu^2}\eta_{\mu}), \mathcal{F}(C_{\sim \mu^2}\psi_{\lambda}),\mathcal{F}(C_{\ll\mu^2}\varphi_{\mu}) )|&\lesssim \Vert C_{\ll \mu^2}\varphi_{\mu}\Vert_{L_t^{\infty}L_x^2} \Vert C_{\gtrsim \lambda^2}\psi_{\lambda}\Vert_{L_{t,x}^2} \Vert Q_{\ll \mu^2}\eta_{\mu}\Vert_{L_{t,x}^2}\\\nonumber
&\lesssim \Vert \varphi_{\mu}\Vert_{S_{\mu}^{s,l,\theta}} \Vert \psi_{\lambda} {\Vert_{S_{\lambda}^{s,l,\theta}}} \Vert \eta_{\mu} \Vert_{W_{\mu}^{-l,s,1-\theta}}.
\end{align}

	\textbf{7.   \boldmath $L \times L \rightarrow H$}\\
We employ the bilinear Strichartz estimate $(2.12)$ by dividing the space time Fourier support of $\overline{C_{\ll \mu^2}\psi_{\mu}}$ and $C_{\sim \lambda^2}\psi_{\lambda}$ into pieces of size $L_1$ and $L_2$ respectively:
\begin{align*}
&	\mu^{\beta}\Vert \mathcal{F}(Q_{\sim \mu^2}(\overline{C_{\ll \mu^2}\psi_{\mu}}C_{\ll \lambda^2}\psi_{\lambda}))\Vert_{L_{\tau,\xi}^2}\\
&\lesssim
\mu^{\beta}\sum_{\substack{L_1\ll \mu^2\\L_2\ll \lambda^2}}\frac{\lambda^{\frac{d-1}{2}}}{\mu^{\frac{1}{2}}}(L_1L_2)^{\frac{1}{2}}\Vert C_{L_1} \varphi_{\mu}\Vert_{L_{t,x}^2}\Vert C_{L_2}\psi_{\lambda}\Vert_{L_{t,x}^2}\\
&\lesssim \mu^{\beta-\frac{1}{2}-s}\lambda^{\frac{d-1}{2}-s}~\mu^s\Vert \langle \tau+|\xi|^2\rangle^{\theta} \mathcal{F}(C_{\ll\mu^2}\varphi_{\mu})\Vert_{L_{\tau,\xi}^2} \lambda^{s}\Vert \langle \tau+|\xi|^2\rangle^{\theta} \mathcal{F}(C_{\ll \lambda^2}\psi_{\lambda})\Vert_{L_{\tau,\xi}^2}\\
&\lesssim \mu^{\beta-\frac{1}{2}-s}\lambda^{\frac{d-1}{2}-s}~ \Vert \varphi_{\mu}\Vert_{S_{\mu}^{s,l,\theta}} \Vert \psi_{\lambda} {\Vert_{S_{\lambda}^{s,l,\theta}}}.
\end{align*}
The constraints  $2s-\beta\geqslant \frac{d}{2}-1$ and $s-\beta\geqslant -\frac{1}{2}$ are required for summability.\\
The argument in $(4.9)$ can be mimicked for the vey small spatial frequency case.\\
We conclude that the following are the requirements for the validity of the estimates in the case $\lambda \lesssim \mu$ for $d\leqslant 3$:
\begin{align}
\bullet~ s-\beta\geqslant -\frac{1}{2} \quad \bullet s-l \geqslant \max\Big\{a-b, -\frac{1}{2}+2(\theta-\theta')-2a\Big\} \quad \bullet 2s-l\geqslant \frac{1}{2}+2a-b \quad \bullet 2s-\beta \geqslant \frac{1}{2}+a-b.
\end{align}

	\textbf{Case II. High-high interaction (\boldmath $\mu \lesssim\lambda_1\sim \lambda_2$)}\\ 
We decompose $\overline{\varphi_{\lambda_1}}$ and $\psi_{\lambda_2}$ as follows:\begin{equation*}
\overline{\varphi_{\lambda_1}}  = \overline{C_{\ll \lambda_1^2}\varphi_{\lambda_1}} + \overline{C_{\gtrsim \lambda_1^2}\varphi_{\lambda_1}}, \quad \psi_{\lambda_2} = C_{\ll \lambda_2^2}\psi_{\lambda_2} + C_{\gtrsim \lambda_2^2}\psi_{\lambda_2}
\end{equation*}
The following cases can be distinguished on the basis of the size of the modulation:\\

\textbf{1. \boldmath $H\times H \rightarrow H$}\\
We use H{\"o}lder's and Bernstein's inequality as follows:
\begin{align*}
&	\mu^{\beta}\Vert \mathcal{F}(Q_{\gtrsim \mu^2}(\overline{C_{\gtrsim \lambda_1^2}\varphi_{\lambda_1}}C_{\gtrsim \lambda_2^2}\psi_{\lambda_2})) \Vert_{L_{\tau,\xi}^2}\\
&\lesssim \mu^{\beta+\frac{d}{2}} \Vert \overline{C_{\gtrsim \lambda_1^2}\varphi_{\lambda_1}} \Vert_{L_{t,x}^2} \Vert C_{\gtrsim \lambda_2^2}\psi_{\lambda_2}\Vert_{L_t^{\infty}L_x^2}\\
&\lesssim \mu^{\beta+\frac{d}{2}}\lambda_1^{-2s-2\theta+a-b}\lambda_1^{s'-2a+b}\Vert \langle \tau+|\xi|^2\rangle^{\theta'}\mathcal{F}( (\lambda_1+|\partial_t|)^a C_{\gtrsim \lambda_1^2}\varphi_{\lambda_1})\Vert_{L_{\tau,\xi}^2}  \lambda_2^{s} \Vert C_{\gtrsim \lambda_2^2}\psi_{\lambda_2}\Vert_{L_t^{\infty}L_x^2}\\
&\lesssim \mu^{\beta+\frac{d}{2}}\lambda_1^{-2s-2\theta+a-b} \Vert \varphi_{\lambda_1}\Vert_{S^{s,l,\theta}_{\lambda_1}}\Vert \psi_{\lambda_2}\Vert_{S^{s,l,\theta}_{\lambda_2}}.
\end{align*}
We can sum the estimate in spatial frequencies $\mu\ll \lambda_1\sim \lambda_2$ provided $2s-\beta\geqslant\frac{d}{2}-2\theta+a-b$, by noting that $\beta+\frac{d}{2}>0$ for $d\geqslant 1$ and $\beta$ as in $(2.2)$.\\

\textbf{2. \boldmath $H \times H \rightarrow L$}\\ 
We use H\"older's and Bernstein's inequality to obtain
\begin{align*}
&|I(\mathcal{F}(Q_{\ll \mu^2}\eta_{\mu}),\mathcal{F}(C_{\gtrsim \lambda_2^2}\psi_{\lambda_2}), \mathcal{F}(C_{\gtrsim \lambda_1^2}\varphi_{\lambda_1}))|\\
&\lesssim \mu^{\frac{d}{2}} \Vert C_{\gtrsim \lambda_1^2}\varphi_{\lambda_1} \Vert_{L_{t,x}^2} \Vert C_{\gtrsim \lambda_2^2}\psi_{\lambda_2}\Vert_{L_{t,x}^2} \Vert Q_{\ll \mu^2}\eta_{\mu}\Vert_{L_{t}^{\infty}L_x^2}\\
&	\lesssim\mu^{\frac{d}{2}+l+2\theta}\lambda_1^{-2s-4\theta+2a-b}\lambda_1^{s'-2a+b}\Vert \langle \tau+|\xi|^2\rangle^{\theta'}\mathcal{F}((\lambda_1+|\partial_t|)^a C_{\gtrsim \lambda_1^2}\varphi_{\lambda_1})\Vert_{L_{\tau,\xi}^2}\\ 
&\quad \times \lambda_2^{s'-2a}\Vert \langle \tau+|\xi|^2\rangle^{\theta'}\mathcal{F}( (\lambda_2+|\partial_t|)^a C_{\gtrsim \lambda_2^2}\psi_{\lambda_2}) \Vert_{L_{\tau,\xi}^2}~\mu^{-l-1}\Vert Q_{\ll \mu^2}\eta_{\mu}\Vert_{L_t^pL_x^2}\\
&\lesssim \mu^{\frac{d}{2}+l+2\theta}\lambda_1^{-2s-4\theta+2a-b} \Vert \varphi_{\lambda_1}\Vert_{S^{s,l,\theta}_{\lambda_1}} \Vert \psi_{\lambda_2}\Vert_{S^{s,l,\theta}_{\lambda_2}}\Vert \eta_{\mu}\Vert_{W^{-l,s,1-\theta}_{\mu}},
\end{align*}
where the last inequality follows from the embedding $(2.7)$.\\
By noting that $\frac{d}{2}+l +2\theta> 0$ we conclude that we require $2s-l \geqslant\frac{d}{2}-2\theta+2a-b$ for summability.\\

	\textbf{3.  \boldmath $H \times L \rightarrow H$}\\
From H{\"o}lder's and Bernstein's inequality, we get
\begin{align*}
&\mu^{\beta}\Vert \mathcal{F}(\overline{C_{\gtrsim \lambda_1^2}\varphi_{\lambda_1}}C_{\ll \lambda_2^2}\psi_{\lambda_2}) \Vert_{L_{\tau,\xi}^2}\\
&\lesssim \mu^{\beta+\frac{d}{2}}\Vert\overline{C_{\gtrsim \lambda_1^2}\varphi_{\lambda_1}}\Vert_{L_{t,x}^2} \Vert C_{\ll \lambda_2^2}\psi_{\lambda_2} \Vert_{L_{t}^{\infty}L_x^2}\\
&\lesssim \mu^{\beta+\frac{d}{2}}\lambda_1^{-2s+a-b-2\theta}\lambda_1^{s'-2a+b}\Vert \langle \tau+|\xi|\rangle^{\theta'}\mathcal{F}((\lambda_1+|\partial_t|)^a C_{\gtrsim \lambda_1^2}\varphi_{\lambda_1})\Vert_{L_{\tau,\xi}^2} \lambda_2^s\Vert C_{\ll \lambda_2^2}\psi_{\lambda_2} \Vert_{L_{t}^{\infty}L_x^2}\\
&\lesssim \mu^{\beta+\frac{d}{2}}\lambda_1^{-2s+a-b-2\theta} \Vert \varphi_{\lambda_1}\Vert_{S^{s,l,\theta}_{\lambda_1}} \Vert \psi_{\lambda_2}\Vert_{S^{s,l,\theta}_{\lambda_2}}.
\end{align*}
Provided $2s-\beta\geqslant \frac{d}{2}-2\theta+a-b$, we can obtain $(4.2)$.\\

	\textbf{4. \boldmath $L \times H \rightarrow H$}\\
This interaction can be handled like case 3 with the roles of $\varphi$ and $\psi$ reversed.\\

\textbf{5. \boldmath $L \times L \rightarrow H$}\\
We treat $d=1,2,3$ separately.\\
{\boldmath $d=3:$} We use H{\"o}lder's and Bernstein's inequality and the endpoint Strichartz space $L_t^2 L_x^6$.
\begin{align*}
\mu^{\beta}\Vert \mathcal{F}(C_{\gtrsim \mu^2}(\overline{C_{\ll \lambda_1^2}\varphi_{\lambda_1}} C_{\ll \lambda_1^2}\psi_{\lambda_2}))\Vert_{L_{\tau,\xi}^2}&\lesssim
\mu^{\beta+\frac{1}{2}} \Vert C_{\ll \lambda_1^2}\varphi_{\lambda_1} \Vert_{L_{t}^{2}L_x^6} \Vert C_{\ll \lambda_1^2}\psi_{\lambda_2}\Vert_{L_{t}^{\infty}L_x^2}\\
&\lesssim \mu^{\beta+\frac{1}{2}}\lambda_1^{-2s} \Vert \varphi_{\lambda_1}\Vert_{S^{s,l,\theta}_{\lambda_1}} \Vert \psi_{\lambda_2}\Vert_{S^{s,l,\theta}_{\lambda_2}}.
\end{align*}
{\boldmath $d=2:$} We employ the bilinear Strichartz estimate $(2.12)$ by decomposing the space-time Fourier supports of $\overline{C_{\ll \lambda_1^2}\varphi_{\lambda_1}}$ and $ C_{\ll \lambda_1^2}\psi_{\lambda_2}$ into pieces of size $L_1$ and $L_2$, respectively:
\begin{align*}
\mu^{\beta}\Vert \mathcal{F}(C_{\gtrsim \mu^2}(\overline{C_{\ll \lambda_1^2}\varphi_{\lambda_1}} C_{\ll \lambda_1^2}\psi_{\lambda_2}))\Vert_{L_{\tau,\xi}^2}&\lesssim \mu^{\beta}\sum_{L_1,L_2\ll \lambda_1^2}(L_1L_2)^{\frac{1}{2}}\Vert C_{L_1}\varphi_{\lambda_1}\Vert_{L_{t,x}^2} \Vert C_{L_2}\psi_{\lambda_2}\Vert_{L_{t,x}^2}\\
&\lesssim \mu^{\beta} \lambda_1^{-2s} \Vert \varphi_{\lambda_1}\Vert_{S^{s,l,\theta}_{\lambda_1}} \Vert \psi_{\lambda_1}\Vert_{S^{s,l,\theta}_{\lambda_2}}.
\end{align*}
%\vspace{-4mm}
%For $\mu\sim \lambda_1\sim 1$, we have
%\begin{align*}
%\mu^{\beta}\Vert \mathcal{F}(C_{\gtrsim \mu^2}(\overline{C_{\ll \lambda_1^2}\varphi_{\lambda_1}} C_{\ll \lambda_1^2}\psi_{\lambda_2}))\Vert_{L_{\tau,\xi}^2}\lesssim \Vert C_{\ll \lambda_1^2}\varphi_{\lambda_1}\Vert_{L_t^{\infty}L_x^2} \Vert C_{\ll \lambda_2^2}\psi_{\lambda_2}\Vert_{L_{t,x}^2}\lesssim \Vert \varphi_{\lambda_1}\Vert_{S^{s,l,\theta}_{\lambda_1}} \Vert \psi_{\lambda_1}\Vert_{S^{s,l,\theta}_{\lambda_2}}.
%\end{align*}
{\boldmath $d=1:$}	An application of  H\"older's and Bernstein's inequality gives
\begin{align*}
\mu^{\beta}\Vert \mathcal{F}(C_{\gtrsim \mu^2}(\overline{C_{\ll \lambda_1^2}\varphi_{\lambda_1}} C_{\ll \lambda_1^2}\psi_{\lambda_2}))\Vert_{L_{\tau,\xi}^2}&\lesssim
\mu^{\beta+\frac{1}{2}} \Vert C_{\ll \lambda_1^2}\varphi_{\lambda_1} \Vert_{L_{t}^{\infty}L_x^2} \Vert C_{\ll \lambda_1^2}\psi_{\lambda_2}\Vert_{L_{t,x}^2}
\lesssim \mu^{\beta+\frac{1}{2}}\lambda_1^{-2s} \Vert \varphi_{\lambda_1}\Vert_{S^{s,l,\theta}_{\lambda_1}} \Vert \psi_{\lambda_2}\Vert_{S^{s,l,\theta}_{\lambda_2}}.
\end{align*}
For $d\leqslant 3$, we require $s\geqslant 0, ~2s-\beta\geqslant\frac{1}{2}$ to sum the above subcase.\\

\textbf{6. \boldmath $L \times L \rightarrow L$}\\
\textbf{\boldmath $d=2,3:$} We decompose the Fourier supports of $C_{\ll \lambda_1^2}\varphi_{\lambda_1}, C_{\ll \lambda_2^2} \psi_{\lambda_2}$ and  $Q_{\ll \mu^2}\eta_{\mu}$ into pieces of size $L_1, L_2$ and $L_0$, respectively and then apply lemma $2.4$.
\begin{align*}
&|I(\mathcal{F}(Q_{\ll \mu^2}\eta_{\mu}),\mathcal{F}(C_{\ll \lambda_2^2} \psi_{\lambda_2}),\mathcal{F}(C_{\ll \lambda_1^2}\varphi_{\lambda_1}))|\\
%&\lesssim \Big|I\Big(\sum_{L_1\ll \lambda_1^2}\overline{C_{L_1}\varphi_{\lambda_1}}, \sum_{L_2 \ll \lambda_2^2}C_{L_2}\psi_{\lambda_2}, \sum_{L_0\ll \mu^2}Q_{L_0}\eta_{\mu}\Big)\Big|\\
&\lesssim \sum_{\substack{L_1,L_2\ll \lambda_1^2\\L_0\ll \mu^2}}|I(,\mathcal{F}(C_{L_0}\eta_{\mu}), \mathcal{F}(C_{L_2}\psi_{\lambda_2}), \mathcal{F}(C_{L_1}\varphi_{\lambda_1}))|\\
&\lesssim  \sum_{\substack{L_1,L_2\ll \lambda_1^2\\L_0\ll \mu^2}} \frac{(L_0L_1L_2)^{\frac{1}{2}}}{(\lambda_1)^{\frac{1}{2}}}\log{\lambda_1}\Vert C_{L_1}\varphi_{\lambda_1}\Vert_{L_{t,x}^2} \Vert C_{L_2}\psi_{\lambda_2}\Vert_{L_{t,x}^2} \Vert Q_{L_0}\eta_{\mu}\Vert_{L_{t,x}^2}\\
&\lesssim \lambda_1^{-\frac{1}{2}-2s}\mu^{2\theta+l}\log{\lambda_1} ~ \lambda_1^{s}\Vert \langle \tau+|\xi|^2\rangle^{\theta}\mathcal{F}(C_{\ll \lambda_1^2}\varphi_{\lambda_1})\Vert_{L_{\tau,\xi}^2} \lambda_2^{s}\Vert \langle \tau+|\xi|^2\rangle^{\theta} \mathcal{F}(C_{\ll \lambda_2^2}\psi_{\lambda_2})\Vert_{L_{\tau,\xi}^2}\\
&\quad \times \mu^{-l-1}\Vert \langle \tau-|\xi|\rangle^{1-\theta}\mathcal{F}(Q_{\ll \mu^2}\eta_{\mu})\Vert_{L_{\tau,\xi}^2}\\
&\lesssim \lambda_1^{-\frac{1}{2}-2s}\mu^{2\theta+l}\log{\lambda_1} ~ \Vert \varphi_{\lambda_1}\Vert_{S^{s,l,\theta}_{\lambda_1}} \Vert \psi_{\lambda_2}\Vert_{S^{s,l,\theta}_{\lambda_2}} \Vert \eta_{\mu}\Vert_{W^{-l,s,1-\theta}_{\mu}}.
\end{align*}
\textbf{\boldmath $d=1:$} We apply H{\"o}lder's inequality and bilinear Strichartz estimate $(2.13)$ by decomposing the space-time Fourier supports of $\overline{C_{\ll \lambda_2^2} \psi_{\lambda_2}}$ and $Q_{\ll \mu^2}\eta_{\mu}$  into pieces of size $L_2$ and $L_0$, respectively.
\begin{align*}
&|I(\mathcal{F}(Q_{\ll \mu^2}\eta_{\mu}),\mathcal{F}(C_{\ll \lambda_2^2} \psi_{\lambda_2}),\mathcal{F}(C_{\ll \lambda_1^2}\varphi_{\lambda_1}))|\\
&\lesssim \Vert C_{\ll \lambda_1^2}\varphi_{\lambda_1} \Vert_{L_{t,x}^2}\Vert \overline{C_{\ll \lambda_2^2} \psi_{\lambda_2}} Q_{\ll \mu^2}\eta_{\mu}\Vert_{L_{t,x}^2}\\
&\lesssim \Vert C_{\ll \lambda_1^2}\varphi_{\lambda_1} \Vert_{L_{t,x}^2} \sum_{\substack{L_0\ll \mu^2\\ L_2\ll \lambda_2^2}}\frac{1}{\lambda_2^{1/2}}(L_0L_2)^{1/2}\Vert C_{L_2}\psi_{\lambda_2}\Vert_{L_{t,x}^2} \Vert Q_{L_0}\eta_{\mu}\Vert_{L_{t,x}^2}\\
&\lesssim \mu^{l+2\theta}\lambda_1^{-1/2-2s}\lambda_1^{s}\Vert \langle \tau+|\xi|^2\rangle^{\theta}\mathcal{F}(C_{\ll \lambda_1^2}\varphi_{\lambda_1})\Vert_{L_{\tau,\xi}^2} \lambda_2^{s}\Vert \langle \tau+|\xi|^2\rangle^{\theta}\mathcal{F}(C_{\ll \lambda_2^2}\psi_{\lambda_2})\Vert_{L_{\tau,\xi}^2}\\
&\quad \times  \mu^{-l-1}\Vert \langle \tau-|\xi|\rangle^{1-\theta}\mathcal{F}(Q_{\ll \mu^2}\eta_{\mu})\Vert_{L_{\tau,\xi}^2}\\
&\lesssim \mu^{l+2\theta}\lambda_1^{-1/2-2s} \Vert \varphi_{\lambda_1}\Vert_{S^{s,l,\theta}_{\lambda_1}} \Vert \psi_{\lambda_2}\Vert_{S^{s,l,\theta}_{\lambda_2}} \Vert \eta_{\mu}\Vert_{W^{-l,s,1-\theta}_{\mu}}.
\end{align*}
One requires $2s-l>-\frac{1}{2}+2\theta$ to sum the subcase for $d\leqslant 3$.\\
%In the case $\mu \sim \lambda_1\sim 1$, 
%\begin{align*}
%|I(\overline{\mathcal{F}(C_{\ll \lambda_1^2}\varphi_{\lambda_1})}, \mathcal{F}(C_{\ll \lambda_2^2} \psi_{\lambda_2}), \mathcal{F}(Q_{\ll \mu^2}\eta_{\mu}))|&\lesssim \Vert C_{\ll \lambda_1^2}\varphi_{\lambda_1} \Vert_{L_{t}^{\infty}L_x^2} \Vert C_{\ll \lambda_2^2} \psi_{\lambda_2}\Vert_{L_{t,x}^2} \Vert Q_{\ll \mu^2}\eta_{\mu}\Vert_{L_{t,x}^2}\\ &\lesssim \Vert \varphi_{\lambda_1}\Vert_{S^{s,l,\theta}_{\lambda_1}} \Vert \psi_{\lambda_2}\Vert_{S^{s,l,\theta}_{\lambda_2}} \Vert \eta_{\mu}\Vert_{W^{-l,s,1-\theta}_{\mu}}.
%\end{align*}
To conclude, 
\begin{equation}
\bullet~2s-l>\max\Big \{-\frac{1}{2}+2\theta, \frac{3}{2}-2\theta+2a-b\Big\} \quad \bullet ~ 2s-\beta>\max\Big \{\frac{3}{2}-2\theta+a-b,\frac{1}{2}\Big\}
\end{equation}
are the requirements for summability in this case.

The estimate $(4.4)$ holds  provided $s-l \geqslant -2+2\theta, s\geqslant 0, 2s-l\geqslant -\frac{1}{2}+2\theta$, and follows by imitating the proof of the product estimate $\Vert fg\Vert_{H^{l-2+2\theta}}\lesssim \Vert f \Vert_{H^s} \Vert g \Vert_{H^s}$ for both the cases above.\\
From $(4.10)$ and $(4.11)$, we list the required constraints on the parameters for the estimate $(4.1)$ to hold true for $d\leqslant 3$:
\begin{align*}
&\bullet ~s-l\geqslant \max\{-2+2\theta,a-b,-\frac{1}{2}+2(\theta-\theta')-2a\}\quad \bullet~s-\beta\geqslant -\frac{1}{2} \\
&\bullet~2s-l>\max\Big\{-\frac{1}{2}+2\theta+2a-b,-\frac{1}{2}+2\theta\Big\} \quad
\bullet~2s-\beta\geqslant \max\Big\{\frac{3}{2}-2\theta+a-b,\frac{1}{2}\Big\}.
\end{align*}
\end{proof}

\section{Proof of Theorem 1.1}
Given the non-linear estimates proved in sections 3 and 4, we can achieve a small data local well posedness result, see \cite[section 5]{candy2019zakharov} for a simplified small data local well-posedness argument. To achieve a large data result, we need to extract a small power of $T$ on the RHS of the non-linear estimates.

We start with proving a slightly weaker form of the energy inequalities but with a small power of $T$ on the RHS. Note that for the $X^{s,\theta}$ type part of the norms, the required factor comes from property $(2.8)$. It remains to extract the factor for the $L_t^{\infty}L_x^2$ part of the norm. \\
To that end, we note for $0<T<1$ and a smooth time cutoff $\chi_T$
\begin{align*}
\lambda^{s} \Vert \chi_T\mathcal{I}_S [C_{\sim \lambda^2}P_{\ll \lambda^2}^{(t)} F_{\lambda}]\Vert_{L_{t}^{\infty}L_x^2} \lesssim \lambda^{s} \Vert \chi_{4T} C_{\sim \lambda^2}P_{\ll \lambda^2}^{(t)}F_{\lambda}\Vert_{L_t^1L_x^2}\lesssim \lambda^s T\Vert C_{\sim \lambda^2}P_{\ll \lambda^2}^{(t)}F_{\lambda}\Vert_{L_{t}^{\infty}L_x^2}.
\end{align*}
Interpolating this with $\lambda^s \Vert \chi_T\mathcal{I}_S [C_{\sim \lambda^2}P_{\ll \lambda^2}^{(t)} F_{\lambda}]\Vert_{L_{t}^{\infty}L_x^2} \lesssim \lambda^{s-2}\Vert P_{\ll \lambda^2}^{(t)}F_{\lambda}\Vert_{L_t^{\infty}L_x^2}$, we obtain for some $0<\delta_1\leqslant -\frac{1}{2}+\theta \ll 1$,
\begin{align}
\lambda^{s}\Vert \chi_T\mathcal{I}_S [C_{\sim \lambda^2}P_{\ll \lambda^2}^{(t)} F_{\lambda}]\Vert_{L_{t}^{\infty}L_x^2} \lesssim \lambda^{s-2+2\delta_1}T^{\delta_1}\Vert C_{\sim \lambda^2}P_{\ll \lambda^2}^{(t)}F_{\lambda}\Vert_{L_{t}^{\infty}L_x^2}\lesssim T^{\delta_1}\Vert F_{\lambda}\Vert_{N^{s,l,\theta-1}}
\end{align}
Note that we considered only the low ($\ll \lambda^2$) temporal frequency localised non-linearity because for the other case, we can again use property $(2.8)$ since $m_S(\tau) \gtrsim 1$. \\
As stated, from property $(2.8)$ in section 2.4, for $\delta_2 >0, ~-\frac{1}{2}<\theta-1<\theta-1+\delta_2<\frac{1}{2}$, we have \textit{only} for the $X^{s,\theta}$ part of the $N$ norm
\begin{align}
\Vert F\Vert_{N^{s,l,\theta-1}(T)} \lesssim T^{\delta_2}\Vert F\Vert_{N^{s,l,\theta-1+\delta_2}(T)}.
\end{align}
%\begin{align}
%&\lambda^s \Vert \chi_T(t) \langle \tau+|\xi|^2\rangle^{\theta-1}C_{\ll \lambda^2}F_{\lambda}\Vert_{L_{t,x}^2}  \lesssim T^{\delta_2} \lambda^s \Vert \langle \tau+|\xi|^2\rangle^{\theta-1+\delta_2} F_{\lambda}\Vert_{L_{t,x}^2},\\ \nonumber
%&\lambda^{s-2a} \Vert \chi_T \langle \tau+|\xi|^2\rangle^{\theta-1}(\lambda+|\partial_t|)^a C_{\ll \lambda^2}F_{\lambda}\Vert_{L_{t,x}^2}  \lesssim T^{\delta_2} \lambda^{s-2a} \Vert \langle \tau+|\xi|^2\rangle^{\theta-1+\delta_2}(\lambda+|\partial_t|)^a F_{\lambda}\Vert_{L_{t,x}^2}.
%\end{align}
From $(5.1)$ and $(5.2)$ we conclude 
\begin{equation}
\Vert \chi_T \mathcal{I}_S[F]\Vert_{S^{s,l,\theta}(T)} \lesssim T^{\delta} \Vert F\Vert_{N^{s,l,\tilde{\theta}}(T)} 
\end{equation}
where $\delta = \min\{\delta_1, \delta_2\}>0$ and $\tilde{\theta} = \theta-1+\delta$.\\
Using the same arguments for the wave solution and wave non-linearity, we have
\begin{equation}
\Vert \chi_T \mathcal{I}_W[G]\Vert_{W^{l,s,\theta}(T)} \lesssim T^{\delta} \Vert G\Vert_{R^{l,s,\tilde{\theta}}(T)}.
\end{equation}
We now head to prove theorem $1.1$.
\begin{proof}
	We call $(u,v) \in S^{s,l,\theta}(T)\times W^{l,s,\theta}(T)$ a solution to the system $(2.1)-(2.2)$ with initial data $(u_0,v_0)\in H^s \times H^l$ if
	\begin{align}
	u(t) &= \chi(e^{it\Delta}u_0 +\chi_T\mathcal{I}_S[uRe(v)])(t)\\ \nonumber
	v(t) &= \chi(e^{it|\nabla|}v_0+\chi_T\mathcal{I}_W[|\nabla||u|^2])(t),
	\end{align}
for all $0<T<1$.\\
	To apply a contraction mapping argument, we write $(5.5)$ as $\Gamma(u,v)(t) = (u,v)(t)$. We prove that $\Gamma$ is a contraction on the space $S^{s,l,\theta}(T)\times W^{l,s,\theta}(T)$ for a suitably chosen $T$. Define $R:=\Vert u_0\Vert_{H_x^s} + \Vert v_0\Vert_{H_x^l}$. We shall drop the superscripts on the norms for notational convenience. Then, using lemma 2.2, lemma 2.3, $(5.3), (5.4), (3.1), (4.1)$, we have
	\begin{align*}
	\Vert \Gamma (u,v)\Vert_{S\times W(T)}& = \Vert\chi(t) e^{it\Delta}u_0\Vert_{S(T)} + \Vert \chi_T(t) \mathcal{I}_S[u Re(v)]\Vert_{S(T)} + \Vert \chi(t) e^{it|\nabla|}v_0\Vert_{W(T)} + \Vert \chi_T(t) \mathcal{I}_W[|\nabla||u|^2]\Vert_{W(T)}\\
	&\lesssim( \Vert u_0\Vert_{H^s_x} + T^{\delta}\Vert u\Vert_{S(T)} \Vert v\Vert_{W(T)} + \Vert v_0\Vert_{H^l_x}+ T^{\delta}\Vert u\Vert_{S(T)}^2)\\
	&\leqslant C( \Vert u_0\Vert_{H^s_x} + \Vert v_0\Vert_{H^l_x} + T^{\delta}(\Vert (u,v)\Vert_{S\times W(T)} + \Vert (u,v)\Vert_{S\times W(T)}^2)).
	\end{align*}
	Hence, $\Gamma$ is well-defined on a ball of radius $2CR$ in the space $S^{s,l,\theta}(T)\times W^{l,s,\theta}(T)$ if $T$ is chosen such that $C_1T^{\delta}R < 1$ for some constant $C_1$. In order to show that $\Gamma$ is a contraction, we consider
	\begin{align*}
	\Vert &\Gamma(u_1,v_1)-\Gamma(u_2,v_2)\Vert_{S \times W(T)}\\
	& = \Vert \chi_T(t)\mathcal{I}_S[u_1Re(v_1) - u_2Re(v_2)]\Vert_{S(T)} + \Vert \chi_T(t)\mathcal{I}_W[|\nabla|(|u_1|^2-|u_2|^2)]\Vert_{W(T)}\\
	&\lesssim T^{\delta}(\Vert u_1 Re(v_1)-u_2Re(v_1)+u_2Re(v_1)-u_2Re(v_2)\Vert_{N^{s,l,\tilde{\theta}}(T)} +\Vert |u_1|^2-|u_2|^2\Vert_{R^{l,s,\tilde{\theta}(T)}})\\
	&\lesssim T^{\delta}(\Vert v_1\Vert_{W(T)}\Vert u_1-u_2\Vert_{S(T)}+\Vert u_2\Vert_{S(T)}\Vert v_1-v_2\Vert_{W(T)} +\Vert u_1-u_2\Vert_{S(T)}\Vert u_1+u_2\Vert_{S(T)})\\
	&\lesssim T^{\delta}(\Vert u_1\Vert_{S(T)} +\Vert v_1\Vert_{W(T)}+\Vert u_2\Vert_{S(T)} +\Vert v_2\Vert_{W(T)})(\Vert u_1-u_2\Vert_{S(T)} +\Vert v_1-v_2\Vert_{W(T)})\\
	&=CT^{\delta}(\Vert(u_1,v_1)\Vert_{S \times W(T)} +\Vert (u_2,v_2)\Vert_{S \times W(T)} )(\Vert (u_1-u_2,v_1-v_2)\Vert_{S \times W(T)}),
	\end{align*}
	which becomes a contraction if $T$ is chosen such that $C_2T^{\delta}R < 1$, for a constant $C_2$. This proves that $(u,v)\in S^{s,l,\theta}(T)\times W^{s,l,\theta}(T)$ is a fixed point of $\Gamma$ i.e. a solution to the system $(2.1)-(2.2)$ with initial data $(u_0,v_0)$. Standard arguments then imply uniqueness of the solution and Lipschitz continuity of the data to solution map for the system.
\end{proof}

\textbf{Acknowledgements-} I would like to thank my advisor Prof. Dr Sebastian Herr for numerous discussions, ideas and encouragement throughout the work. I am also grateful to Dr Robert Schippa for helpful suggestions. Financial support by the German Research Foundation through IRTG 2235 is gratefully acknowledged.

	\bibliography{BIB}
	
	\bibliographystyle{apa}
	\vspace{5mm}
\noindent
\author{Akansha Sanwal,}
\address{Faculty of Mathematics,
	Bielefeld University,
	Bielefeld, 33615, Germany.}\\
E-mail - \href{mailto:asanwal@math.uni-bielefeld.de}{\nolinkurl{asanwal@math.uni-bielefeld.de}}
	
%	\nocite{bejenaru2010convolution}
%	\nocite{bejenaru2011convolutions}
%	\nocite{bejenaru2015well}
%	\nocite{bourgain1996wellposedness}
%	\nocite{candy2019zakharov}
%	\nocite{domingues2019note}
%	\nocite{ginibre1997cauchy}
%	\nocite{guo2012small}
%	\nocite{guo2016zakharov}
%	\nocite{ozawa1992existence}
%	\nocite{tao2006nonlinear}
%	\nocite{zakharov1972collapse}
%	\nocite{pecher2001global}
%	\nocite{holmer2006local}
%	\nocite{colliander2008low}
%	\nocite{pecher2004global}
	
	\end{document}